\documentclass[11pt,a4paper]{article}
\usepackage[utf8]{inputenc}
\usepackage[english]{babel}
\usepackage[T1]{fontenc} % -> Im PDF kann man Sonderzeichen wie bei Tur\'an kopieren
\usepackage{lmodern} %sinnvoll, Schrift nicht mehr pixlig
\usepackage{xcolor} %Paket fuer Farbe
\usepackage{color}
\usepackage{amssymb} %fuer natuerliche-zahlen-symbol
\usepackage{mathrsfs} %fuer geschwungenes L
\usepackage{amsmath}
\usepackage{amsthm}%fuer Theoreme, proof
\usepackage{amsopn}
\usepackage{enumitem}%for enumerate/itemize without indentation -> [wide=0pt] or [leftmargin=*] -> for [label=\textbf{S.\arabic*}]
\usepackage{graphicx} % für \bigcupdot notwendig
\usepackage{subcaption}
\usepackage{xspace} %bei Abkuerzung von d.h. \xspace in newcommand
\usepackage{mathtools} %noetig fuer declare paired delimiter -> gaussklammer so definiert
\usepackage{scalerel,stackengine} % for the newcommand of the of fourier hat
\usepackage{url}
\usepackage{hyperref}
\usepackage{doi}
\usepackage{tikz}
\usepackage{nomencl} % list of notation
\usepackage[margin=3cm]{geometry}
\usepackage{pictexwd,dcpic}

\theoremstyle{plain}

\newtheorem{Theorem}{Theorem}[section]
\newtheorem{Lemma}[Theorem]{Lemma}
\newtheorem{Proposition}[Theorem]{Proposition}

\theoremstyle{definition}

\newtheorem{Definition}[Theorem]{Definition}
\newtheorem{Example}[Theorem]{Example}

\newtheorem{Algorithm}[Theorem]{Algorithm}

\theoremstyle{remark}

\newtheorem{Remark}[Theorem]{Remark}

\makeatletter
\newtheoremstyle{case}{3pt}{3pt}{\addtolength{\@totalleftmargin}{2.0em}\addtolength{\linewidth}{-1.5em}\parshape 1 1.5em \linewidth}{}{\kern-1.5em\itshape}{.}{ }{\thmnote{#3}}
\makeatother
\theoremstyle{case}

\makeatletter
\newtheoremstyle{caseinside}{3pt}{3pt}{\addtolength{\@totalleftmargin}{4.0em}\addtolength{\linewidth}{-3.0em}\parshape 1 3.0em \linewidth}{}{\kern-1.5em\itshape}{.}{ }{\thmnote{#3}}
\makeatother
\theoremstyle{caseinside}

\DeclareMathOperator{\aff}{aff} % affine hull
\DeclareMathOperator*{\argmin}{arg\,min}
\DeclareMathOperator{\E}{E}     %Euclidean group
\DeclareMathOperator{\GL}{\mathrm{GL}}
\DeclareMathOperator{\Ind}{Ind} %induced representation

\DeclareMathOperator{\reldiff}{Relative\,difference}
\DeclareMathOperator{\Skew}{Skew}
\DeclareMathOperator{\SO}{\mathrm{SO}}
\DeclareMathOperator{\supp}{supp}
\DeclareMathOperator{\U}{\mathrm{U}} % unitary group
\renewcommand{\O}{\operatorname{{\mathrm O}}} %orthogonal group
\renewcommand{\Re}{\operatorname{Re}} % real part
\renewcommand{\Im}{\operatorname{Im}} % imaginary part

\DeclarePairedDelimiter{\abs}{\lvert}{\rvert}
\DeclarePairedDelimiter{\angles}{\langle}{\rangle}
\DeclarePairedDelimiter{\norm}{\lVert}{\rVert}
\DeclarePairedDelimiter{\parens}{\lparen}{\rparen}

\DeclarePairedDelimiter{\braces}{\lbrace}{\rbrace}

\DeclarePairedDelimiterX{\iso}[2]{(}{)}{#1,#2} % (A,b) is better than (A|b) since E(d)=O(d)\ltimes\R^d, but for examples like (0&0\\1&1|0\\1) is (A|b) better than (A,b)
\DeclarePairedDelimiterX{\set}[2]{\{}{\}}{#1\,\delimsize|\,#2}
\DeclarePairedDelimiterX{\scalar}[2]{\langle}{\rangle}{#1,#2}
\makeatletter
\newcommand{\@absonestar}[1]{\abs*{#1}_1}
\newcommand{\@absonenostar}[2][]{\abs[#1]{#2}_1}
\newcommand{\absone}{\@ifstar\@absonestar\@absonenostar}
\makeatother
\makeatletter
\newcommand{\@zeronormstar}[2]{\norm*{#1}_{#2,0}}
\newcommand{\@zeronormnostar}[3][]{\norm[#1]{#2}_{#3,0}}
\newcommand{\zeronorm}{\@ifstar\@zeronormstar\@zeronormnostar}
\makeatother
\makeatletter
\newcommand{\@nablanormstar}[2]{\norm*{#1}_{#2,\nabla}}
\newcommand{\@nablanormnostar}[3][]{\norm[#1]{#2}_{#3,\nabla}}
\newcommand{\nablanorm}{\@ifstar\@nablanormstar\@nablanormnostar}
\makeatother
\makeatletter
\newcommand{\@nablazeronormstar}[2]{\norm*{#1}_{#2,0,\nabla,0}}
\newcommand{\@nablazeronormnostar}[3][]{\norm[#1]{#2}_{#3,\nabla,0}}
\newcommand{\nablazeronorm}{\@ifstar\@nablazeronormstar\@nablazeronormnostar}
\makeatother
\makeatletter
\newcommand{\@newnormstar}[2]{\norm*{#1}_{#2,0,0}}
\newcommand{\@newnormnostar}[3][]{\norm[#1]{#2}_{#3,0,0}}
\newcommand{\newnorm}{\@ifstar\@newnormstar\@newnormnostar}
\makeatother
\makeatletter
\newcommand{\@newnablanormstar}[2]{\norm*{#1}_{#2,\nabla,0,0}}
\newcommand{\@newnablanormnostar}[3][]{\norm[#1]{#2}_{#3,\nabla,0,0}}
\newcommand{\newnablanorm}{\@ifstar\@newnablanormstar\@newnablanormnostar}
\makeatother
\makeatletter
\providecommand*{\cupdot}{%
  \mathbin{%
    \mathpalette\@cupdot{}%
  }%
}
\newcommand*{\@cupdot}[2]{%
  \ooalign{%
    $\m@th#1\cup$\cr
    \sbox0{$#1\cup$}%
    \dimen@=\ht0 %
    \sbox0{$\m@th#1\cdot$}%
    \advance\dimen@ by -\ht0 %
    \dimen@=.5\dimen@
    \hidewidth\raise\dimen@\box0\hidewidth
  }%
}
\providecommand*{\bigcupdot}{%
  \mathop{%
    \vphantom{\bigcup}%
    \mathpalette\@bigcupdot{}%
  }%
}
\newcommand*{\@bigcupdot}[2]{%
  \ooalign{%
    $\m@th#1\bigcup$\cr
    \sbox0{$#1\bigcup$}%
    \dimen@=\ht0 %
    \advance\dimen@ by -\dp0 %
    \sbox0{\scalebox{2}{$\m@th#1\cdot$}}%
    \advance\dimen@ by -\ht0 %
    \dimen@=.5\dimen@
    \hidewidth\raise\dimen@\box0\hidewidth
  }%
}
\makeatother

\stackMath
\newcommand\fourier[1]{%
\savestack{\tmpbox}{\stretchto{%
  \scaleto{%
    \scalerel*[\widthof{\ensuremath{#1}}]{\kern-.6pt\bigwedge\kern-.6pt}%
    {\rule[-\textheight/2]{1ex}{\textheight}}%WIDTH-LIMITED BIG WEDGE
  }{\textheight}% 
}{0.5ex}}%
\stackon[1pt]{#1}{\tmpbox}%
}
\par

\newcommand\+{\mkern2mu} %spacing for \exists -> https://tex.stackexchange.com/questions/101018/should-using-spaces-in-math-mode-be-a-common-thing

\newcommand{\dual}[1]{\widehat{#1}}
\newcommand{\fdot}{{{}\cdot{}}} %spacing of $f(\cdot)$ in not equal to that of $f({}\cdot{})$, also $\lvert\cdot\rvert$ vs $\lvert{}\cdot{}\rvert$ -> good for norm and abs
\newcommand{\gdot}{\cdot} %group action
\newcommand{\gplus}[2]{#1\oplus #2}

\newcommand{\simrg}{\sim} % {\sim_{\R^{d_2},\G}}

\newcommand{\euler}{\mathrm e}
\newcommand{\iu}{\mathrm i}
\newcommand{\C}{{\mathbb C}}
\newcommand{\N}{{\mathbb N}}

\newcommand{\R}{{\mathbb R}}
\newcommand{\Z}{{\mathbb Z}}
\renewcommand{\epsilon}{\varepsilon} %Bernd benutzt varepsilon
\renewcommand{\phi}{\varphi} %Bernd benutzt varphi (cut-off function)

\newcommand{\A}{\mathcal A}

\newcommand{\F}{\mathcal F} %'big' finite group
\newcommand{\G}{\mathcal G}
\newcommand{\Gstar}{{\G_{\text{\textasciitilde}}^*}}
\newcommand{\GstarExample}{{\G\setminus\{\id\}}}
\newcommand{\Gcoset}{{\G_{\text{\textasciitilde}}}}
\newcommand{\Rcoset}{{\RR_{\text{\textasciitilde}}}}
\newcommand{\RRVset}{\RR_V}
\newcommand{\RRVcoset}{\tilde{\RR}_V}
\renewcommand{\H}{\mathcal H}
\newcommand{\id}{{id}} %identity of a group
\newcommand{\LL}{L_\SG} % lattice
\newcommand{\dualL}{L_\SG^*} % dual lattice
\newcommand{\MM}{M_0} % M_0=\set{n\in\N}{\T^m is a normal subgroup of \G}
\renewcommand{\P}{\mathcal P} %'big' point set
\newcommand{\CC}{\mathcal C} % Cuboid
\newcommand{\EE}{\mathcal E} % irreducble representations
\newcommand{\NN}{\mathcal N} % normal subgroup and nearest neighbors
\newcommand{\RR}{\mathcal R} % range (derivative)
\newcommand{\SG}{\mathcal S} %space group in R^{d_1}
\newcommand{\T}{\mathcal T} %'big' translation group
\newcommand{\daff}{{d_{\mathrm{aff}}}} %dimension of the objective structure
\newcommand{\rot}{L} % linear component
\newcommand{\trans}{\tau} % translational component
\newcommand{\xx}{\chi_\G x_0} % constant map x_0
\newcommand{\ff}{{f_V}} % second derivative of the energy function
\newcommand{\gggg}{g_\RR} % the norm function as a convolution
\newcommand{\lambdamin}{\lambda_{\min}}
\newcommand{\newgggg}{g_{\RR,0,0}}
\newcommand{\ee}{e_V}
\newcommand{\eea}{e_{V_{a, \alpha_0}}}

\newcommand{\lambdaa}{\lambda_{\textnormal{a}}} % stability constant 
\newcommand{\lambdanewa}{\lambda_{\textnormal{a},0,0}}
\newcommand{\proj}[1]{\bar{#1}}
\newcommand{\pproj}[1]{\pi(#1)}
\newcommand{\indi}{\chi}

\newcommand{\UIso}{U_{\mathrm{iso},0,0}}
\newcommand{\Uiso}[1]{U_{\mathrm{iso}}(#1)}

\newcommand{\zeroUiso}[1]{U_{\mathrm{iso},0}(#1)}
\newcommand{\Unewiso}[1]{U_{\mathrm{iso},0,0}(#1)}

\newcommand{\Urot}[1]{U_{\mathrm{rot}}(#1)}
\newcommand{\zeroUrot}[1]{U_{\mathrm{rot},0}(#1)}
\newcommand{\Unewrot}[1]{U_{\mathrm{rot},0,0}(#1)}
\newcommand{\UTrans}{U_{\mathrm{trans}}}
\newcommand{\Utrans}[1]{U_{\mathrm{trans}}(#1)}
\newcommand{\UPer}{U_{\mathrm{per}}}
\newcommand{\UPerC}{U_{\mathrm{per},\C}}
\newcommand{\Per}{L_\mathrm{per}^\infty}
\newcommand{\eg}{\mbox{e.\,g.}\xspace}
\newcommand{\ie}{\mbox{i.\,e.}\xspace}

\allowdisplaybreaks

\begin{document}
\begin{center}
\begin{Large}
Stability of Objective Structures: General Criteria and Applications
\end{Large}
\\[0.5cm]
\begin{large}
Bernd Schmidt\footnote{Universität Augsburg, Germany, {\tt bernd.schmidt@math.uni-augsburg.de}} and 
Martin Steinbach\footnote{Universität Augsburg, Germany, {\tt steinbachmartin@gmx.de}} 
\end{large}
\\[0.5cm]
\today
\\[1cm]
\end{center}

\begin{abstract}
We develop a general stability analysis for objective structures, which constitute a far reaching generalization of crystal lattice systems. 
We show that these particle systems, although in general neither periodic nor space filling, allow for the identification of stability constants in terms of representations of the underlying symmetry group and interaction potentials. 
Our main results provide general stability criteria and second order energy bounds for equilibrium configurations. 
In particular, a general computational algorithm to test objective structures for their stability is derived. 
By way of example  we show that our method can be applied to verify the stability of carbon nanotubes with chirality. 
\end{abstract}

2020 {\em Mathematics Subject Classification}. 
70C20, % Mechanics of particles and systems: Statics 
70J25, %   	Stability for problems in linear vibration theory
74Kxx, % 		Thin bodies, structures; in: Mechanics of deformable solids
49K40. %   	Sensitivity, stability, well-posedness

{\em Key words and phrases.} Objective structures,  atomistic systems, stability, nanotubes.  

\tableofcontents

%-------------------------------------------------------------------
%-------------------------------------------------------------------
\section{Introduction}

A fundamental problem in material science is to investigate elastic structures for their stability properties. In classical continuum mechanics stability criteria can be derived from positivity properties of the Hessian of the stored energy function such as the classical Legendre-Hadamard condition. For atomistic systems the continuum approach, however, can imply stability only in a regime in which the Cauchy-Born rule is known to be valid and small scale oscillations are excluded a priori. In such a regime individual atoms follow the macroscopic deformation so that the configurational energy can effectively be described by a macroscopic continuum stored energy function. 

In general, however, stability conditions in the continuum cannot detect small wave-length instabilities when modes at the interatomic length-scale are excited. While systems of only few atoms may be investigated by computational methods, in high dimensional systems with many or even an infinite number of particles a direct and fully discrete numerical approach is not feasible. It is thus a fundamental challenge to identify classes of structures that, on the one hand, are general enough to cover many interesting and physically relevant examples and, on the other hand, are still amenable to a quantitative stability analysis. 

Due to their ubiquity in nature, crystalline atomic systems have been extensively investigated over the last decades by physicists and material scientists, cp.\ \cite{Born1954,Kittel1996,Wallace1998}. More recently, there has been a vital interest also within the applied analysis community to rigorously substantiate the connection of atomistic systems and effective continuum models. A brief review of some key results is provided below. 

The main objective of the current contribution is to go beyond the periodic setting of crystals, and to extend a stability analysis to the class of \emph{objective structures}. These structures, introduced by James in \cite{James2006}, constitute a far reaching generalization of lattice systems. They are relevant in a remarkable number of models ranging from biology (to describe, \eg, parts of viruses) to nanoscience (to model carbon nanotubes), see, \eg,  \cite{FalkJames06,DumitricaJames07,Dayal2010,FengPlucinskyJames19}. Their defining feature is that, up to rigid motions of the surrounding space, any two particles are embedded in an identical environment of other particles. (In a lattice, this would be the case even up to translations.) This entails that objective structures are described by orbits of a single reference point under the action of a general discrete group of Euclidean isometries, cf.\ \cite{James2006,Juestel2014}, and allows to resort to ideas in harmonic analysis. However, the symmetry of objective structures in general is considerably more complex than that of a periodic crystal and the adaption of methods and results for lattices has only been achieved in a few cases so far such as, notably, an algorithm for the Kohn-Sham equations for clusters \cite{BanerjeeElliottJames:15} and the X-ray analysis of helical structures in \cite{FrieseckeJamesJuestel16}.

Our endeavor is to provide an in-depth study of stability properties of general objective structures. As detailed below, we will derive explicit formulae for stability constants which are not only interesting from a theoretical point of view but directly lead to an efficient computational algorithm for these quantities. The strength of our method is demonstrated by showing stability of a carbon nanotube with non-trivial chirality. Yet, the non-periodicity and possible lower dimensionality of objective structures poses severe analytical changes as compared to crystals. 

For a crystal, in the seminal contribution \cite{Friesecke2002} Friesecke and Theil have shown the validity of the Cauchy-Born rule for small strains in a two-dimensional mass spring model. Their results were then extended to arbitrary dimensions and more general atomistic interactions %by Conti, Dolzmann, M{\"u}ller and Krichheim 
in \cite{Conti2006}. We refer to the survey article \cite{Ericksen2008} for a general review on the Cauchy-Born rule. These results are the first step of various variational discrete-to-continuum Gamma-convergence results for energy functionals on lattices in which the limiting model is identified explicitly in terms of the effective Cauchy-Born stored energy function, see \cite{Schmidt2009} and \cite{Braun2013} for linear and nonlinear elastic bulk systems, respectively, and \cite{Schmidt2006,BraunSchmidt:22} for thin films. 

An alternative approach to establish the Cauchy-Born rule for crystals has been set forth by E and Ming in \cite{E2007} where they show that under suitable atomistic stability assumptions solutions of the equations of continuum elasticity theory on a flat torus and subject to smooth body forces are approximated by associated atomistic equilibrium configurations. These results have been generalized to the whole space and only mild regularity assumptions in \cite{Ortner2013} and to boundary value problems in \cite{Braun2016}. Even the dynamic setting has been considered in \cite{E2007,Ortner2013,Braun2017,BuchbergerSchmidt:24}. 

At the core of all the aforementioned contributions lies a stability condition for lattice systems. This seems to have been analyzed in detail for the first time by Hudson and Ortner in \cite{Hudson2011}. Motivated by their results, an explicitly computable stability constant for lattice systems has been derived in \cite{Braun2016} which permits a direct comparison with the corresponding continuum Legendre-Hadamard continuum stability constant in a long wave-length regime. An extension to multilattices has been considered in \cite{OlsonOrtner:17}. The main technical tool which allows for an efficient analysis of discrete and continuum stability constants and their interrelation is the Fourier transform. Due to the periodicity of the underlying lattice, the atomistic stability can be determined on a diverging series of finite boxes with periodic boundary conditions, \ie, larger and larger tori, invading the whole space to define a stability constant for the infinite particle system eventually in the limit.  

A related question has been considered already in the classical contribution \cite{Wigner:30} for the opposite special case of finite systems: Stability assumed, it is shown how the eigenfrequencies of a system of finitely many regularly arranged mass points can be overviewed by considering the group of their covering operations and its representations. 

At this point it becomes apparent why the extension to general objective structures is challenging. First, these structures need not be periodic. While their group theoretic description allows for the definition of a Fourier transform, it is a priori not clear in which sense a large wave-length limit with diverging tori can be performed. Second, these structures may be lower dimensional, both macroscopically and microscopically. This leads to the possibility of buckling modes that might impair stability. A related problem is that novel Korn-type inequalities are necessary in order to control atomistic strain measures in terms of configurational energy expressions. 

In our recent contribution \cite{SchmidtSteinbach:21a} we have provided an efficient and extensive description of the dual space of a general discrete group of Euclidean isometries in terms of a finite union of `wave vector domains' which can be related to a specific `translation type' subgroup of finite index. This subgroup in turn allows to define a notion of periodicity and to construct seminorms that measure the local difference of discrete gradients to the set of rigid motions. These seminorms and, in particular, Korn type inequalities for objective structures in terms of these seminorms are studied in detail in our companion paper \cite{SchmidtSteinbach:21b}. 

The main goal of the present contribution is, departing from \cite{SchmidtSteinbach:21a,SchmidtSteinbach:21b} to obtain a general stability criterion for objective structures which identifies an appropriate stability constant and which leads to a directly implementable numerical algorithm. We achieve this in Theorem~\ref{Theorem:LambdaIrreducible}, which gives an abstract representation result in Fourier space, and Theorem~\ref{Theorem:LambdaInduced}, which provides an explicit formula in terms of the above mentioned wave vector domains.  It turns out interesting to consider two stability constants $\lambdaa$ and $\lambdanewa$ which measure the stability against displacements from the set of rigid motions in two different seminorms $\norm\fdot_{\RR}$ and $\newnorm\fdot{\RR}$. While in the important special cases of finite structures and of space filling structures these seminorms are equivalent, this is not the case for lower dimensional infinite systems. There the distinction between these seminorms allows for a fine analysis of systems at the onset of instabilities, that may be caused by buckling modes. 

As our stability constants are lower bounds on the Hessian operator of the configurational energy, it is worth noticing that we also have results on matching upper bounds: Strong bounds with respect to $\norm\fdot_{\RR}$ are provided in Theorem~\ref{Theorem:BoundedThree} for structures with equilibrized onsite potentials and in Theorem~\ref{Theorem:FiniteSpace} for finite and for space filling structures; bounds with respect to $\newnorm\fdot{\RR}$ are established for general structures in Theorem~\ref{Theorem:dtwodtwo} whose proof turns out rather demanding. 

The algorithm resulting from Theorem~\ref{Theorem:LambdaInduced} is spelled out as Algorithm~\ref{Algorithm:algorithm}. The power of our approach is demonstrated by applying our scheme to the physically most relevant example of a carbon nanotube \cite{Iijima1991,Dresselhaus1992}. These structures have attracted an immense attention in the literature due to their extraordinary mechanical and electronic properties \cite{Treacy1996,Tuukkanen2014,Charlier1998,Cao2007}. A carbon nanotube can be visualized by rolling up a portion of a regular hexagonal lattice along a lattice vector, so that a long hollow cylinder emerges on the surface of which the atoms are bonded in a seamless way. Except for special choices of the winding direction the nanotube will possess a non-trivial chirality. 

In spite of the tremendous boom in the physical and material science literature, rigorous analytical results are comparatively scarce and primarily focused on continuum models such as \cite{Arroyo2005,Bajaj2013,Favata2012}. Notable exceptions are \cite{ElKass2014}, where under the assumption of stability a discrete Saint-Venant principle is established for a general class of nanotubes, and the recent contribution \cite{Friedrich2019}. In \cite{Friedrich2019}, which appears to provide the farthest reaching results on carbon nanotubes to date, the Cauchy-Born rule is established rigorously for an atomistic model for stretched tubes under the assumption that the tubes be achiral. In fact, for tubes without chirality the authors show that the stability of the cell problem can be upscaled to the whole structure. 

Our harmonic analysis based scheme in fact also applies to general carbon nanotubes with non-trivial chirality. By way of example we explicitly apply our algorithm to a so-called $(5,1)$ nanotube and a relaxed version thereof. We investigate the stability both for stretched and natural reference configurations, see Example~\ref{Example:Nanotube}. While in the stretched regime both stability constants $\lambdaa$ and $\lambdanewa$ are indeed positive, we see that at the onset of buckling $\lambdanewa$ vanishes while the weaker constant $\lambdaa$ still remains positive. 

%-------------------------------------------------------------------
\subsection*{Acknowledgements}
This work was partially supported by project 285722765 of the Deutsche
Forschungsgemeinschaft (DFG, German Research Foundation). 

%-------------------------------------------------------------------
\subsection*{Data Availability Statement}
Data sharing is not applicable to this article as no datasets were generated or analyzed during the current study. 

\section{Objective structures}\label{section:objective-structures}

In this section we set the stage by collecting known results on the group theoretic description of objective structures and the quantitative analysis of their deformations in terms of suitable seminorms. In particular, we state results from \cite{SchmidtSteinbach:21a,SchmidtSteinbach:21b} which are of significant relevance for our analysis. Unless explicitly stated otherwise, proofs for those statements in this section that go beyond standard results can all be found in
\cite{SchmidtSteinbach:21a,SchmidtSteinbach:21b}.

%-------------------------------------------------------------------
%-------------------------------------------------------------------
\subsection{Discrete subgroups of the Euclidean group and their orbits}%\label{subsection:structure}
%-------------------------------------------------------------------

The \emph{Euclidean group} $\E(d)$ in dimension $d\in\N$ is the set of all Euclidean distance preserving transformations of $\R^d$ into itself, their elements are called \emph{Euclidean isometries}. It may be described as $\E(d)=\O(d)\ltimes \R^d$, the outer semidirect product of $\R^d$ and the orthogonal group $\O(d)$ in dimension $d$ with group operation given by
\[\iso{A_1}{b_1}\iso{A_2}{b_2}=\iso{A_1A_2}{b_1+A_1b_2}\]
for $\iso{A_1}{b_1},\iso{A_1}{b_2}\in\E(d)$. We call $\rot(\iso Ab):=A$ the \emph{linear component} and $\trans(\iso Ab):=b$ the \emph{translation component} of $\iso Ab$ so that $g=\iso{I_d}{\trans(g)}\iso{\rot(g)}0$ for each $g\in\E(d)$. 
An Euclidean isometry $\iso Ab$ is called a \emph{translation} if $A=I_d$. The set $\{I_d\}\ltimes\R^d$ of translations forms an abelian subgroup of $\E(d)$. 
$\E(d)$ acts on $\R^d$ via 
\[\iso Ab\gdot x := Ax+b\qquad\text{for all }\iso Ab\in \E(d)\text{ and }x\in\R^d.\]
For a group $\G<\E(d)$ the \emph{orbit} of a point $x\in\R^d$ under the action of the group is 
\[\G\gdot x:=\set{g\gdot x}{g\in\G}.\]
In the following we will consider \emph{discrete subgroups} of the Euclidean group, which are those $\G<\E(d)$ for which every orbit $\G\gdot x$, $x\in\R^d$, is discrete.

Particular examples are the so-called \emph{space groups}. These are those discrete groups $\G < \E(d)$ that contain $d$ translations whose translation components are linearly independent. Their subgroup of translations is generated by $d$ such linearly independent translations and forms a normal subgroup of $\G$ which is isomorphic to $\Z^d$. 

In general, discrete subgroups of $\E(d)$ can be characterized as follows. (Also cp.\ \cite[A.4 Theorem 2]{Brown1978}.) Recall that two subgroups $\G_1,\G_2<\E(d)$ are \emph{conjugate} in $\E(d)$ if there exists some $g\in\E(d)$ such that $g^{-1}\G_1 g=\G_2$. (This corresponds to a rigid coordinate transformation in $\R^d$.) 
\begin{Theorem}%\label{Theorem:Browndecomposablediscretegroup}
Let $\G<\E(d)$ be discrete, $d\in\N$. There exist $d_1,d_2\in\N_0$ such that $d=d_1+d_2$, a $d_2$-dimensional space group $\SG$ and a discrete group $\G'<\gplus{\O(d_1)}{\SG}$ such that $\G$ is conjugate under $\E(d)$ to $\G'$ and $\pi(\G')=\SG$, where $\pi$ is the natural epimorphism $\gplus{\O(d_1)}{\E(d_2)}\to\E(d_2)$, $\gplus Ag\mapsto g$.
\end{Theorem}
Here $\gplus{}{}$ is the group homomorphism $\gplus{}{}\colon\O(d_1)\times\E(d_2)\to\E(d_1+d_2)$ with 
\begin{align*}
(A_1,\iso{A_2}{b_2})\mapsto \gplus{A_1}{\iso{A_2}{b_2}}:=\iso[\bigg]{\parens[\bigg]{\begin{matrix}A_1&0\\0&A_2\end{matrix}}}{\parens[\bigg]{\begin{matrix}0\\b_2\end{matrix}}} 
\end{align*}
and $\gplus{\O(d_1)}{\SG}$ is understood to be $\O(d)$ if $d_1=d$ and to be $\SG$ if $d_1=0$. The theorem allows us to assume that $\G$ is of the form $\G'$ which we will do henceforth with no loss of generality.  

Such a group $\G$ can be efficiently described in terms of the space group $\SG$, the kernel $\F$ of $\pi|_\G$ and a section $\T \subset \G$ of the translation group $\T_\SG$ of $\SG$, \ie, a set $\T\subset\G$ such that the map $\T\to\T_\SG$, $g\mapsto \pi(g)$ is bijective. 
The group $\T\F=\pi|_\G^{-1}(\T_\SG)$ is a normal subgroup of $\G$ of finite index. 
We remark that the quantities $d$, $d_1$, $d_2$, $\F$, $\SG$ and $\T_\SG$ are uniquely defined by $\G$. However, in general there is no canonical choice for $\T$, it might not a be group and the elements of $\T$ might not commute. Yet, in \cite{SchmidtSteinbach:21a} it is shown that there is an $m_0\in\N$ such that $\T^N = \set{t^N}{t\in\T}$ is a normal subgroup of $\G$ if and only if $N$ is a multiple of $m_0$: 
\[ \T^N \triangleleft \G \iff N \in \MM := m_0 \N. \] 
For each $N \in \MM$, $\T^N$ is isomorphic to $\Z^{d_2}$ and of finite index in $\G$. 

The set $\T$ allows to introduce a notion of periodicity for functions defined on $\G$. For a set $S$ and $N\in \MM$ we say that a function $u\colon \G \to S$ is \emph{$\T^N$-periodic} if
\[u(g) = u(gt)\qquad\text{for all }g\in\G\text{ and }t\in\T^N.\]
It is called \emph{periodic} if there exists some $N\in \MM$ such that $u$ is $\T^N$-periodic. We also set
\[\Per(\G,\C^{m\times n}):=\set{u\colon\G\to\C^{m\times n}}{u\text{ is periodic}}.\]
(Throughout $\C^{m\times n}$ is equipped with the usual Frobenius inner product and induced norm $\norm\fdot$.) We notice that the above definition of periodicity is independent of the choice of $\T$ and that $\Per(\G,\C^{m\times n})$ is a vector space. In fact, one has 
\[\Per(\G,\C^{m\times n}) = \set[\Big]{\G\to\C^{m\times n},g\mapsto u(g\T^N)}{N\in \MM, u\colon\G/\T^N\to\C^{m\times n}}.\]
For each $N\in \MM$ we now fix a representation set $\CC_N$ of $\G/\T^N$ and we equip $\Per(\G,\C^{m\times n})$ with the inner product $\angles{\fdot,\fdot}$ given by
\[\angles{u,v}:=\frac1{\abs{\CC_N}}\sum_{g\in\CC_N}\angles{u(g),v(g)}\qquad\text{if $u$ and $v$ are $\T^N$-periodic}\]
for all $u,v\in\Per(\G,\C^{m\times n})$.
The induced norm is denoted by $\norm\fdot_2$.

Following James \cite{James2006} we define an \emph{objective (atomic) structure} as a discrete point set $S$ in $\R^d$ such that for any two elements $x_1, x_2 \in S$ there is an isometry $g \in \E(d)$ with $g\cdot S = S$ and $g\cdot x_1 = x_2$. An equivalent characterization is that $S$ be an orbit of a point under the action of a discrete subgroup of $\E(d)$, see, \eg, \cite[Proposition 3.14]{Juestel2014}.

The following elementary observation is proved in \cite[Lemmas~2.12--2.14]{SchmidtSteinbach:21b}. It shows that by changing coordinates we may without loss of generality assume that objective structures $\G \cdot x_0$ lie in $\{0_{d-\daff}\}\times\R^\daff$, where $\daff$ is their affine dimension, and that $\G$ acts trivially on $\R^{d-\daff}\times\{0_{\daff}\}$. We denote by $\aff(A)$ the \emph{affine hull} of a set $A\subset\R^d$ and by $\dim(A):=\dim(\aff (A))$ its \emph{affine dimension}. 

\begin{Lemma}\label{Lemma:WLOGdOStwo}
Let $\G<\E(d)$ be discrete, $x_0\in\R^d$ and $\daff=\dim(\G\cdot x_0)$. 
There exist some $a\in\E(d)$ and a discrete group $\G'<\gplus{\{I_{d-\daff}\}}{\E(\daff)}$ such that with $x_0'=a\cdot x_0$ one has $\G'\cdot x_0'=a\cdot(\G\cdot x_0)$ and 
\[\aff(\G'\cdot x_0')=\{0_{d-\daff}\}\times\R^\daff.\]
\end{Lemma}
\subsubsection*{Fourier analysis}

We denote by $\dual{\G}$ the \emph{dual space} of $\G$, which consists of all equivalence classes of irreducible (unitary) representations of the group $\G$. (Recall that a 
(unitary) representation $\rho$ of dimension $d_\rho$ of $\G$ is a homomorphism from $\G$ to the unitary matrices in $\C^{d_\rho\times d_\rho}$, that two representations $\rho,\rho'$ are said to be equivalent if $d_\rho=d_{\rho'}$ and $T^{\mathsf H}\rho(g)T=\rho'(g)$ for all $g\in\G$ and some unitary $d_\rho\times d_\rho$ matrix $T$ and that $\rho$ is said to be irreducible if the only subspaces of $\C^{d_\rho}$ invariant under $\set{\rho(g)}{g\in\G}$ are $\{0\}$ and $\C^{d_\rho}$.) We further remark that the dimensions $d_\rho$ of irreducible representations $\rho$ of $\G$ are uniformly bounded, cf.\ \cite{Moore1972}. One-dimensional representations will be denoted by the symbol $\chi$ and called \emph{characters}. Their equivalence class is a singleton. If $\G$ is abelian, then every irreducible representation of $\G$ is of dimension one. In particular, $\dual{\T^{m_0}}$ consists of all homomorphisms from $\G$ to the complex unit circle. 

Since $\G$ is almost abelian, the Mackey machine works to describe the dual space $\dual{\G}$, cp.~\cite{Mackey:58,KleppnerLipsman:72,KleppnerLipsman:72b} and also cf.\ \cite{BekkaHarpe20} for a recent account in the present and more general settings. Departing from this, an explicit labeling of representations with wave vectors has been obtained in \cite{SchmidtSteinbach:21a}, which is vital to our stability analysis. 
It is obtained by lifting the characters on $\T_\mathcal{S}$ via $\pi^{-1}$ to $\T\F = \pi^{-1}(\T_\mathcal{S})$ and then considering those representations on $\G$ that are induced by representations of $\T\F$. 
We thus recall here the following definition, cp., \eg, \cite[Section 8.2]{Steinberg2012}. 
\begin{Definition}\label{Definition:IndRep}
Let $\H$ be a subgroup of $\G$ of finite index $n=\abs{\G:\H}$.
Choose a complete set of representatives $\{k_1,\dots,k_n\}$ of the left cosets of $\H$ in $\G$.
If $\rho\colon\H\to\U(d_\rho)$ is a representation of $\H$ and $g\in\G$, we set $\dot\rho(g):=\rho(g)$ if $g\in\H$ and $\dot\rho(g):=0_{d_\rho,d_\rho}$ otherwise. 
The \emph{induced representation} $\Ind_\H^\G\rho\colon\G\to\U(nd_\rho)$ is defined by
\[\Ind_\H^\G\rho(g)=\begin{bmatrix}\dot\rho(k_1^{-1}gk_1)&\cdots&\dot\rho(k_1^{-1}gk_n)\\ \vdots&\ddots&\vdots\\\dot\rho(k_n^{-1}gk_1)&\cdots&\dot\rho(k_n^{-1}gk_n)\end{bmatrix}\qquad\text{for all }g\in\G.\]
The \emph{induced representation} of an equivalence class of representations is the equivalence class of the induced representation of a representative. 
Moreover, let $\Ind_\H^\G(\dual\H)$ denote the set of all induced representations of $\dual\H$.
We also write $\Ind$ instead of $\Ind_\H^\G$ if $\H$ and $\G$ are clear by context.
\end{Definition}

Observe that a representation $\rho$ of $\G$ is $\T^N$-periodic, $N\in \MM$, if and only if $\rho|_{\T^N}=I_{d_\rho}$. We fix a representation set $\EE$ of $\set{\rho\in\dual\G}{\rho\text{ is periodic}}$ and define the Fourier transform as follows. 
\begin{Definition}\label{Definition:FourierPeriodic}
If $u\in\Per(\G,\C^{m\times n})$ and $\rho$ is a periodic representation of $\G$, we set
\[\fourier u(\rho):=\frac1{\abs{\CC_N}}\sum_{g\in\CC_N}u(g)\otimes\rho(g)\in\C^{(md_\rho)\times(nd_\rho)},\]
where $N\in\MM$ is such that $u$ and $\rho$ are $\T^N$-periodic and $\otimes$ denotes the Kronecker product, see \eqref{eq:Kronecker-product}.
\end{Definition}
\begin{Proposition}[The Plancherel formula]\label{Proposition:TFplancherelmatrix} 
The Fourier transformation
\[\fourier\fdot \colon \Per(\G, \C^{m\times n}) \to\bigoplus_{\rho\in\EE} \C^{(md_\rho)\times(nd_\rho)}, \quad u \mapsto (\fourier u(\rho))_{\rho\in\EE}\]
is well-defined and bijective. Moreover, the Plancherel formula
\[\angles{u,v}=\sum_{\rho\in\EE}d_\rho\angles{\fourier u(\rho),\fourier v(\rho)}\qquad\text{for all }u,v\in\Per(\G,\C^{m\times n})\]
holds true. 
\end{Proposition}
We remark that for all $u\colon\G\to\C^{m\times n}$ and $N\in \MM$ such that $u$ is $\T^N$-periodic, one gets 
\[\set{\rho \in \EE}{\fourier u(\rho) \neq 0} \subset \set{\rho\in\EE}{\rho\text{ is $\T^N$-periodic}}.\]

\begin{Definition}\label{Definition:FourierLOne}
For all $u\in L^1(\G,\C^{m\times n})$ and all representations $\rho$ of $\G$ we define
\[\fourier u(\rho):=\sum_{g\in\G}u(g)\otimes \rho(g).\]
\end{Definition}
Note that if $\G$ is finite and $u\in L^1(\G,\C^{m\times n})=\Per(\G,\C^{m\times n})$, then the Definitions~\ref{Definition:FourierPeriodic} and \ref{Definition:FourierLOne} for $\fourier u$ differ by the multiplicative constant $\abs\G$, but it will always be clear from the context which of the two is meant.
If $\G$ is infinite, there is no ambiguity as then $L^1(\G,\C^{m\times n})\cap\Per(\G,\C^{m\times n})=\{0\}$.
\begin{Lemma}\label{Lemma:Convolution}
Let $u\in L^1(\G,\C^{l\times m})$, $v\in\Per(\G,\C^{m\times n})$ and consider their convolution $u*v\in\Per(\G,\C^{l\times n})$ given by
\[u*v(g):=\sum_{h\in\G}u(h)v(h^{-1}g)\qquad\text{for all }g\in\G.\]
Suppose $\rho$ be a periodic representation of $\G$.
Then
\begin{enumerate}
\item the convolution $u*v$ is $\T^N$-periodic if $v$ is $\T^N$-periodic and
\item we have $\fourier{u*v}(\rho)=\fourier u(\rho)\fourier v(\rho)$.
\end{enumerate}
\end{Lemma}
%-------------------------------------------------------------------
\subsubsection*{Wave vector characterization of the dual space}
An essential feature of our stability results in Theorem~\ref{Theorem:LambdaInduced} and the induced Algorithm~\ref{Algorithm:algorithm} below is a computationally amenable characterization of the representation set $\EE$. To this end we recall from \cite{SchmidtSteinbach:21a} that $\dual\G$ is effectively described in terms of a finite set $R$ of representations of $\T\F$ and corresponding wave vector domains $K_\rho\subset\R^{d_2}$, $\rho\in R$, which are the fundamental domains of certain space groups $\G_\rho$ acting on $\R^{d_2}$. The remaining part of this section shows how to construct $R$ and $\G_\rho$, $\rho\in R$. (Readers might prefer to postpone reading this until it is needed in Theorem~\ref{Theorem:LambdaInduced} and Algorithm~\ref{Algorithm:algorithm}.)

We define an equivalence relation $\simrg$ on $\dual{\T\F}$. As $\T\F$ is a normal subgroup of $\G$, $\G$ acts on the set of irreducible representations of $\T\F$ via $g \cdot \rho (h) = \rho(g^{-1} h g)$ for all $h \in \T\F$ for any $g\in\G$ and irreducible representation $\rho$ of $\T\F$. This induces an action of $\G$ on $\dual{\T\F}$. The characters of $\T\F$ act on $\dual{\T\F}$ by multiplication. 
\begin{Definition}\label{Definition:ChiK-L}
\begin{enumerate}
\item\label{Definition:ChiK} For all $k\in\R^{d_2}$ we define the character $\chi_k\in\dual{\T\F}$ by
\[\chi_k(g):=\exp(2\pi\iu\scalar k{\trans(\pi(g))})\qquad \text{for all }g\in\T\F,\]
where $\pi\colon\T\F\to\T_\SG$ is the natural epimorphism.
\item\label{Definition:LLst}
$\LL<\R^{d_2}$ is the lattice of translational components of $\T_\SG$ and $\dualL$ its dual lattice: 
\begin{align*}
\LL:=\trans(\T_\SG), \qquad 
\dualL:=\set{x\in\R^{d_2}}{\scalar xy\in\Z\text{ for all }y\in \LL}.
\end{align*}
\item\label{Definition:RelationTF}
We define the relation $\simrg$ on $\dual{\T\F}$ by
\[(\rho\simrg\rho')\;:\Longleftrightarrow\;(\exists\,g\in\G\,\exists\,k\in\R^{d_2} : g\cdot\rho=\chi_k\rho').\]
\end{enumerate}
\end{Definition}
\begin{Remark}\label{Remark:VNK}
One has $\dualL/n = \set{k\in\R^{d_2}}{\chi_k|_{\T^n}=1}$ for all $n\in\N$. 
\end{Remark}
The following provides an algorithm for the determination of a representation set $R$ of $\dual{\T\F}/{\simrg}$.
\begin{Lemma}\label{Lemma:RepSys}
\begin{enumerate}
\item\label{item:aaa} 
Every representation set of $\set{\rho\in\dual{\T\F}}{\rho|_{\T^{m_0}}=I_{d_\rho}}/{\simrg}$ is a representation set of $\dual{\T\F}/{\simrg}$.
\item%\label{item:bbb} 
The map
\[\dual{(\T\F)_{m_0}}\to\set{\rho\in\dual{\T\F}}{\rho|_{\T^{m_0}}=I_{d_\rho}},\quad\rho\mapsto\rho\circ\pi\]
where $\pi\colon\T\F\to(\T\F)_{m_0}$ is the natural epimorphism, is bijective.
In particular, the set $\set{\rho\in\dual{\T\F}}{\rho|_{\T^{m_0}}=I_{d_\rho}}$ is finite.
\item%\label{item:ccc} 
Let $K$ be a representation set of $(\dualL/{m_0})/\dualL$ and $\P$ be a representation set of $\G/(\T\F)$.
Then, for all $\rho,\rho'\in\set{\tilde\rho\in\dual{\T\F}}{\tilde\rho|_{\T^{m_0}}=I_{d_{\tilde\rho}}}$ it holds
\[(\rho\simrg\rho')\iff(\exists\,g\in\P\,\exists\,k\in K: g\cdot\rho=\chi_k\rho').\]
\end{enumerate} 
\end{Lemma}
In particular, the set $\dual{\T\F}/{\simrg}$ is finite. We now associate a special space group acting on $\R^{d_2}$ to any $\rho\in\dual{\T\F}$. 
\begin{Definition}\label{Definition:Grho}
For all $\rho\in\dual{\T\F}$ we define the set
\[\G_\rho:=\set[\Big]{\iso{\rot(\pi(g))}k}{g\in\G,k\in\R^{d_2}:g\cdot\rho=\chi_k\rho}\subset\E(d_2),\]
where $\pi\colon\G\to\SG$ is the natural epimorphism. We also set $\G_{\rho'}:=\G_\rho$ if $\rho'\in\rho$. 
\end{Definition}
\begin{Proposition}\label{Proposition:SpaceGroupLattice}
For each $\rho\in\dual{\T\F}$ the set $\G_\rho$ is a $d_2$-dimensional space group. It holds
\[\dualL\le\set[\big]{k\in\R^{d_2}}{\iso{I_{d_2}}k\in\G_\rho}\le\dualL/m_0.\]
Moreover, if $\rho|_{\T^N}=I_{d_\rho}$ with $N\in \MM$, then the set $\dualL/N$ is invariant under $\G_\rho$. 
\end{Proposition}
The following theorem summarizes results from \cite{SchmidtSteinbach:21a} on the structure of $\dual\G$ that will be needed in the proof of Theorem~\ref{Theorem:LambdaInduced}. As they are not needed for its mere statement, Theorem~\ref{Theorem:MainRepDual} can be safely skipped on first reading. Here the quotient space $\R^{d_2}/\G_\rho$ denotes the set of orbits of $\R^{d_2}$ under the action of $\G_\rho$, cf.\ Proposition~\ref{Proposition:SpaceGroupLattice}, and we write $\bigsqcup$ for a disjoint union. 
Also recall that the sets $R$ in the following theorem are finite.
\begin{Theorem}\label{Theorem:MainRepDual}
\begin{enumerate}
\item\label{Theorem-item:RepSet}
Let $R$ be a representation set of $\dual{\T\F}/{\simrg}$.
Then, the following map is bijective:
\begin{align*}
\bigsqcup_{\rho\in R}\R^{d_2}/\G_\rho\to\Ind_{\T\F}^\G(\dual{\T\F}),\qquad
(\G_\rho\cdot k,\rho)\mapsto\Ind_{\T\F}^\G(\chi_k\rho).
\end{align*}
\item\label{Theorem-item:RepSys}
Let $R$ be a representation set of $\set{\rho\in\dual{\T\F}}{\rho|_{\T^{m_0}}=I_{d_\rho}}/{\simrg}$ and $N\in \MM$.
Then the following map is bijective: 
\begin{align*}
&\bigsqcup_{\rho\in R}(\dualL/N)/\G_\rho\to\Ind_{\T\F}^\G(\set{\rho\in\dual{\T\F}}{\rho|_{\T^N}=I_{d_\rho}}), \qquad 
(\G_\rho\cdot k,\rho)\mapsto\Ind_{\T\F}^\G(\chi_k\rho).
\end{align*}
\item\label{Corollary-item:Subreps-induced}
Let $R$ be as in \ref{Theorem-item:RepSet}. For every $\sigma\in\dual\G$ there exists a $\rho\in R$ and a $k\in\R^{d_2}$ such that $\sigma$ is a subrepresentation of $\Ind_{\T\F}^\G(\chi_k\rho)$. If moreover $R$ is as in \ref{Theorem-item:RepSys} and $\rho|_{\T^N}=I_{d_\rho}$ for an $N \in M_0$, then $k$ can be chosen in $\dualL/N$. 
\end{enumerate}
\end{Theorem}
%

%-------------------------------------------------------------------
\subsection{Deformations and local rigidity seminorms}%\label{subsection:deformation-seminorms}
%-------------------------------------------------------------------
%
We consider deformation mappings $y\colon\G\cdot x_0\to\R^d$ of the structure $\G\cdot x_0$. One can describe such a mapping by the induced `deformation' $v\colon\Gcoset\to\R^d$ on the set $\Gcoset:=\G/\G_{x_0}$ of left cosets  of the stabilizer subgroup, given by $v(g)=y(g\cdot x_0)$.
In order to describe the action of a deformation at $g\cdot x_0$ in relation to its position within the whole structure $\G\cdot x_0$ in its environment we define an associated `group displacement mapping' $u\colon\G\to\R^d$ such that
\begin{equation}\label{eq:group-defo}
v(g)=\frac1{\abs{\G_{x_0}}}\sum_{g'\in g}g'\cdot(x_0+u(g'))\qquad\text{for all }g\in\Gcoset,
\end{equation}
\eg by setting $u(g')=\rot(g')^T(v(g'\G_{x_0})-g'\cdot x_0)$. 
More generally, if $\RR\subset\G$ is such that $\RR\G_{x_0}=\RR$ we write $\Rcoset$ for $\RR/\G_{x_0}$ and for any mapping $u\colon\RR\to\R^d$ we define the averaged mapping $\proj{u}\colon\RR\to\R^d$ by  
\[\proj{u}(g)=\frac1{\abs{\G_{x_0}}}\sum_{h'\in \G_{x_0}}\rot(h')u(gh')\qquad\text{for all }g\in\RR.\]
So by \eqref{eq:group-defo} $\rot(g)\proj{u}(g)=v(g)-g\cdot x_0$ only depends on $g\G_{x_0}\in\Rcoset$. 
(Hence, the expression $\rot(g)\proj{u}(g)$ is unambiguously understood for $g\in\Rcoset$.) 
In particular, $v$ is the translation $v(g)=g\cdot x_0+a$ for all $g\in\Gcoset$ and an $a\in\R^d$ if and only if $\rot(g)\proj{u}(g)=a$ for all $g\in\G$ and $v$ is the rotation $v(g)=R(g\cdot x_0)$ for all $g\in\Gcoset$ and an $R\in\SO(d)$ if and only if $\rot(g)\proj{u}(g)=(R-I_d)(g\gdot x_0)$ for all $g\in\G$. 
We note that this defines a projection 
\begin{align}\label{eq:pi-proj}
\pi\colon\{u\colon\RR\to\R^d\}\to\{u\colon\RR\to\R^d\}, \qquad 
u\mapsto\proj{u}. 
\end{align}
(In case $\G_{x_0}=\{\id\}$, $\pi$ is the identity mapping.) 
For brevity we introduce the notation 
\begin{align*}
\UPerC&:=\Per(\G,\C^{d\times 1})=\set{u\colon\G\to\C^d}{u\text{ is periodic}}, \\ 
\UPer&:=\set{u\colon\G\to\R^d}{u\text{ is periodic}}\subset\UPerC.
\end{align*}

We will measure the deviation of a displacement from the set of rigid motions or a specific subset thereof locally for a given neighborhood range  $\RR$ in terms of certain seminorms on $\UPer$. In particular, $\norm \fdot_\RR$ will measure the local distances from the set of all infinitesimal rigid motions. While the linear component of a general rigid motion is a generic skew symmetric matrix $S \in \Skew(d)$, stronger seminorms are obtained by restricting to specific subsets.

For $S\in\Skew(d)$ we write $S\in\Skew_{0}(d)$ if its lower right $d_2\times d_2$ block vanishes and $S\in\Skew_{0,0}(d)$ if all entries outside its upper left $d_1\times d_1$ block vanish. 
The seminorm $\norm \fdot_{\RR,0}$ will measure the local distances to those rigid motions that fix $\{0_{d_1}\}\times\R^{d_2}$ intrinsically, corresponding to $S \in\Skew_0(d)$, while $\norm \fdot_{\RR,0,0}$ will measure the local distances to those rigid motions that fix $\{0_{d_1}\}\times\R^{d_2}$ in $\R^{d}$, corresponding to $S\in\Skew_{0,0}(d)$. 
\begin{Definition}\label{Definition:UtransUrot}
For all $\RR\subset\G$ such that $\RR\G_{x_0}=\RR$ we define the vector spaces 
\begin{align*}
\Utrans{\RR}&:=\set[\bigg]{u\colon\RR\to\R^d}{\exists\+ a\in\R^d\; \forall g\in\RR:\rot(g)\proj{u}(g)=a},\\
\Urot{\RR}&:=\set[\bigg]{u\colon\RR\to\R^d}{\exists\+ S\in\Skew(d)\; \forall g\in\RR:\rot(g)\proj{u}(g)=S(g\gdot x_0-x_0)},\\
\Uiso{\RR}&:=\Utrans{\RR}+\Urot{\RR}.
\end{align*}
Analogously we define $\zeroUrot{\RR}$ and $\zeroUiso{\RR}$ as well as $\Unewrot{\RR}$ and $ \Unewiso{\RR}$ if in the above definition we replace $\Skew(d)$ by $\Skew_{0}(d)$, respectively, $\Skew_{0,0}(d)$. 
\end{Definition}
By construction, all these sets are invariant under $\pi$. 
Clearly, $\Unewrot{\RR}\subset\zeroUrot{\RR}\subset\Urot\RR$ and it is easy to see that $\pi(U_{\mathrm{iso}(,0,0)}(\RR))=\pi(\Utrans\RR)\oplus \pi(U_{\mathrm{rot}(,0,0)}(\RR))$. 
\begin{Proposition}\label{Proposition:UtransUrotbig}
Suppose that $\RR\subset\G$ is such that  $\RR\G_{x_0}=\RR$, $\id\in\RR$ and $\aff(\RR\cdot x_0)=\aff(\G\cdot x_0)$. Let $\pi$ be as in \eqref{eq:pi-proj}.
Then
\begin{align*}
&\begin{aligned}\phi_1\colon&\R^d\to\pi({\Utrans\RR})\\
&a\mapsto\parens[\big]{\RR\to\R^d,g\mapsto\rot(g)^{\mathsf T}a},\end{aligned}\\
&\begin{aligned}\phi_2\colon&\R^{d_3\times\daff}\times\Skew(\daff)\to\pi(\Urot\RR)\\
&(A_1,A_2)\mapsto \parens[\Big]{\RR\to\R^d,g\mapsto\rot(g)^{\mathsf T}\parens[\Big]{\begin{smallmatrix}0&A_1\\-A_1^{\mathsf T}&A_2\end{smallmatrix}}(g\gdot x_0-x_0)},\end{aligned}\\
&\begin{aligned}\phi_3\colon&\R^{d_3\times d_4}\times\R^{d_3\times d_2}\times\Skew(d_4)\times\R^{d_4\times d_2}\to\pi(\zeroUrot\RR)\\
&(A_1,A_2,A_3,A_4)\mapsto \parens[\bigg]{\RR\to\R^d,g\mapsto\rot(g)^{\mathsf T}\parens[\bigg]{\begin{smallmatrix}0&A_1&A_2\\-A_1^{\mathsf T}&A_3&A_4\\-A_2^{\mathsf T}&-A_4^{\mathsf T}&0\end{smallmatrix}}(g\gdot x_0-x_0)},\quad\text{and}\end{aligned}\\
%\shortintertext{and}
&\begin{aligned}\phi_4\colon&\R^{d_3\times d_4}\times\Skew(d_4)\to\pi(\Unewrot\RR)\\
&(A_1,A_2)\mapsto \parens[\Big]{\RR\to\R^d,g\mapsto\rot(g)^{\mathsf T}\parens[\Big]{\parens[\Big]{\begin{smallmatrix}0&A_1\\-A_1^{\mathsf T}&A_2\end{smallmatrix}}\oplus0_{d_2,d_2}}(g\gdot x_0-x_0)}\end{aligned}
\end{align*}
are isomorphisms, where $d_3=d-\daff$ and $d_4=\daff-d_2$.
\end{Proposition}
\begin{Remark}
In Proposition~\ref{Proposition:Manifold} we show that indeed $\Uiso\RR$ is the set of \emph{infinitesimally rigid displacements} of $\RR$ which is the tangent space at the identity mapping to the space of finite rigid deformations.
\end{Remark}
\begin{Definition}
Let $\RR\subset\G$ be a finite set. We denote by $\norm\fdot$ the Euclidean norm on $(\R^d)^{\RR}$, \ie, $\norm{u}^2=\sum_{g\in\RR}\norm{u(g)}^2$, and by $\pi_{\Uiso\RR}$, $\pi_{\zeroUiso\RR}$, $\pi_{\Unewiso\RR}$ the orthogonal projections on $\{u\colon\RR\to\R^d\}$ with respect to the norm $\norm\fdot$ with kernels $\Uiso\RR$, $\zeroUiso\RR$ and $\Unewiso\RR$, respectively. 

We define three seminorms $\norm\fdot_{\RR}\le\zeronorm\fdot{\RR}\le\newnorm\fdot{\RR}$ on $\UPer$ by setting 
\begin{align*}
\norm{u}_{\RR(,0,0)}&=\parens[\Big]{\frac1{\abs{\CC_N}}\sum_{g\in\CC_N}\norm{\pi_{U_{\mathrm{iso}(,0,0)}(\RR}(u(g\fdot)|_\RR)}^2}^{\frac12}
\end{align*}
(with the obviously matching choices of indices $0$), whenever $u$ is $\T^N$-periodic.  
\end{Definition}
Due to the discrete nature of the underlying objective structure these seminorms can equivalently be described as seminorms on (discrete) gradients. 
\begin{Definition}
For all $u\in\UPer$ and finite sets $\RR\subset\G$ such that $\RR\G_{x_0}=\RR$ we define the \emph{discrete derivative}
\begin{align*}
\nabla_\RR u\colon&\G\to\{v\colon\RR\to\R^d\}\\
&g\mapsto(\nabla_\RR u(g)\colon\RR\to\R^d,h\mapsto \proj{u}(gh)-\rot(h)^{\mathsf T}\proj{u}(g)).
\end{align*}
\end{Definition}
We remark that, if $u\in\UPer$ is $\T^N$-periodic for some $N\in\MM$ and $\RR\subset \G$ is finite, then also the discrete derivative $\nabla_{\RR}u$ is $\T^N$-periodic.
\begin{Definition}
For each finite set $\RR\subset\G$ we define the four seminorms $\nablanorm\fdot\RR\le\nablazeronorm\fdot\RR\le\newnablanorm\fdot\RR\le\norm{\nabla_{\RR}\fdot}_2$ on $\UPer$ by first setting 
\begin{align*}
\norm{u}_{\RR,\nabla(,0,0)}&=\parens[\Big]{\frac1{\abs{\CC_N}}\sum_{g\in\CC_N}\norm{\pi_{U_{\mathrm{rot}(,0,0)}(\RR)}(\nabla_\RR u(g))}^2}^{\frac12}
\end{align*}
(with matching choices of indices $0$), whenever $u$ is $\T^N$-periodic. Here $\pi_{U_{\mathrm{rot}(,0,0)}(\RR)}$ are the orthogonal projections on $\{u\colon\RR\to\R^d\}$ with respect to the norm $\norm\fdot$ with respective kernels $U_{\mathrm{rot}(,0,0)}(\RR)$. 
The fourth seminorm is given by
\begin{align*}
\norm{\nabla_{\RR} u}_2&=\parens[\Big]{\frac1{\abs{\CC_N}}\sum_{g\in\CC_N}\norm{\nabla_{\RR} u(g)}^2}^{\frac12}.
\end{align*}
\end{Definition}
We note that all of the above introduced seminorms are independent of the choice of $\CC_N$.
The main result of \cite{SchmidtSteinbach:21b} is summarized in the following Theorem~\ref{Theorem:StrongerEquivalence} on the equivalence of seminorms whenever the range $\RR$ is rich enough. For this we first introduce the following notation. 
\begin{Definition}\label{Definition:Property}
$\RR\subset\G$ is an \emph{admissible} neighborhood range of $\id$ if $\RR$ is finite, $\RR\G_{x_0}=\RR$ and there exist two sets $\RR',\RR''\subset \G$ with $\RR'\RR''\subset\RR$ such that $\id\in\RR'\cap\RR''$, $\RR'$ generates $\G$ and 
\[\aff(\RR''\cdot x_0)=\aff(\G\cdot x_0).\]
\end{Definition}
The highly non-trivial part of this theorem is that for an admissible $\RR\subset\G$ the two seminorms $\norm\fdot_{\RR}$ and $\zeronorm\fdot{\RR}$ are equivalent. This is a \emph{discrete Korn inequality} for objective structures. 
Examples show that in general $\norm\fdot_{\RR}$ and $\newnorm\fdot{\RR}$ are not equivalent.
\begin{Theorem}\label{Theorem:StrongerEquivalence}
Suppose $\RR,\RR'\subset\G$ are admissible $\id$-neighborhoods and $\RR''\subset\G$ is a finite generating set for $\G$ such that $\RR''\G_{x_0}=\RR''$. Then 
\begin{enumerate}
\item\label{Theorem:EquivalenceAll} 
$\norm\fdot_{\RR}$, $\norm\fdot_{\RR'}$, $\zeronorm\fdot{\RR}$, $\nablanorm\fdot{\RR}$ and $\nablazeronorm\fdot{\RR}$ are equivalent and their kernel is $\UIso\cap\UPer$, 
\item\label{Theorem:NewEquivalenceAll} 
$\newnorm\fdot{\RR}$, $\newnorm\fdot{\RR'}$, and $\newnablanorm\fdot{\RR}$ are equivalent and their kernel is $\UIso\cap\UPer$, 
\item\label{Theorem:nablaequivalent} 
$\norm{\nabla_{\RR}\fdot}_2$ and $\norm{\nabla_{\RR''}\fdot}_2$ are equivalent and their kernel is $\UTrans\cap\UPer$.
\item\label{Theorem:seminormequivalence} 
If $\G$ is a space group, then $\norm\fdot_{\RR}$, $\newnorm\fdot{\RR}$ and $\norm{\nabla_{\RR}\fdot}_2$ are equivalent.
\end{enumerate}
\end{Theorem}
%

%-------------------------------------------------------------------
\section{Energy, criticality and stability}%\label{section:potential}
%
%-------------------------------------------------------------------
\subsection{Configurational energy and stability constants}
We assume that the configurational energy of a deformed objective structure is given as a sum of site potentials that describe the interaction of any single atom with all other atoms. Although we will consider bounded perturbations of the identity eventually, it will be convenient to define the interaction potential on all of  
$(\R^d)^\Gstar$, cf.\ Remark~\ref{Remark:VFrechet}\ref{item:RemarkVDomain} below,  where $\Gstar = \Gcoset\setminus\G_{x_0}$ consists of the nontrivial left cosets of $\G_{x_0}$. 
\begin{Definition}\label{Definition:potential}
Let $V\colon(\R^d)^\Gstar\to\R$ be the \emph{interaction potential}.
We assume that $V$ has the following properties:
\begin{enumerate}[label=(H\arabic*)]
\item\label{item:Rotation} (\emph{Invariance under $\O(d)$}) For all $R\in\O(d)$ and $y\colon\Gstar\to\R^d$ we have
\[V(Ry)=V(y).\]
\item\label{item:Vfrechet}(\emph{Smoothness}) For all $y\colon\Gstar\to\R^d$ the function
\begin{align*}
L^\infty(\Gstar,\R^d)\to\R, \qquad 
z\mapsto V(y+z)
\end{align*}
is two times continuously Fr\'echet differentiable, where $L^\infty(\Gstar,\R^d)$ is the space of all bounded functions from $\Gstar$ to $\R^d$ equipped with the uniform norm $\norm\fdot_\infty$.
\end{enumerate}
For all $y\colon\Gstar\to\R^d$ and $g,h\in\Gstar$ we define the partial Jacobian row vector $\partial_gV(y)\in\R^d$ and the partial Hessian matrix $\partial_g\partial_hV(y)\in\R^{d\times d}$ by
\begin{align*}
\parens{\partial_gV(y)}_i:=V'(y)(\delta_ge_i)
\qquad\text{and}\qquad
\parens{\partial_g\partial_hV(y)}_{ij}:=V''(y)(\delta_ge_i,\delta_he_j)
\end{align*}
for $i,j\in\{1,\dots,d\}$, where $\delta_k\colon\Gstar\to\{0,1\}$, $l\mapsto\delta_{k,l}$ for each $k\in\Gcoset$.
\begin{enumerate}[label=(H\arabic*)]\setcounter{enumi}{2}
\item\label{item:infty}(Summability) For all $y\colon\Gstar\to\R^d$ we have
\[\sum_{g\in\Gstar}\norm{\partial_g V(y)}<\infty\quad\text{and}\quad \sum_{g,h\in\Gstar}\norm{\partial_g\partial_hV(y)}<\infty.\]
\end{enumerate}
We say a set $\RRVcoset\subset\Gstar$ is an \emph{interaction range} of $V$ if for all $y\colon\Gstar\to\R^d$ we have $V(y)=V(\chi_{\RRVcoset}y)$, where $\chi_{\RRVcoset}$ is the indicator function.
We denote $y_0=(g\gdot x_0-x_0)_{g\in\Gstar}\in(\R^d)^\Gstar$.
If $V$ has finite interaction range $\RRVcoset$, then we extend the domain of $V'(y_0)$ and $V''(y_0)$ to $\{z\colon\Gstar\to\R^d\}$ and $\{z\colon\Gstar\to\R^d\}^2$, respectively, by
\begin{align*}
V'(y_0)z_1:=V'(y_0)(\chi_{\RRVcoset}z_1)
\qquad{and}\qquad
V''(y_0)(z_1,z_2):=V''(y_0)(\chi_{\RRVcoset}z_1,\chi_{\RRVcoset}z_2)
\end{align*}
for all $z_1,z_2\in\{z\colon\Gstar\to\R^d\}\setminus L^\infty(\Gstar,\R^d)$.
\end{Definition} % vgl. stability Hudson Ortner
\begin{Remark}\label{Remark:VFrechet}
\begin{enumerate}
\item\label{item:RemarkVFrechet} For all $y\colon\Gstar\to\R^d$ and $z,z_1,z_2\in L^\infty(\Gstar,\R^d)$ we have
\begin{align*}
V'(y)z&=\sum_{g\in\Gstar}\partial_g V (y)z(g)
\shortintertext{and}
V''(y)(z_1,z_2)&=\sum_{g,h\in\Gstar}z_1(g)^{\mathsf T}\partial_g\partial_hV(y)z_2(h).
\end{align*}
\item If $V$ has finite interaction range, then \ref{item:Vfrechet} implies \ref{item:infty}.
\item\label{item:RemarkVDomain} For simplicity we assume that the domain of $V$ is the whole space $(\R^d)^\Gstar$.
It would be sufficient if $V$ is defined only on $\O(d)y_0+U$, where $U$ is a small neighbourhood of $0\in(\R^d)^{\Gstar}$ with respect to the uniform norm.
\item\label{item:RemarkGTrivial} If $\G_{x_0}$ is trivial, the domain $\Gstar$ of $V$ may be identified with $\GstarExample$.
\end{enumerate}
\end{Remark}
\begin{Example}
An example of an interaction potential is given in terms of the Lennard-Jones pair interaction potential as 
\[V\colon(\R^d)^\Gstar\to\R,\quad y\mapsto\sum_{g\in\Gstar}\parens[\big]{\norm{y(g)}^{-12}-\norm{y(g)}^{-6}}.\]
\end{Example}
Recall from \eqref{eq:group-defo} that, for a given displacement $u\colon\G\to\R^d$ the physical particles are at the points $v(g\G_{x_0})=g\cdot x_0+\rot(g)\proj{u}(g)$, $g\in\G$, and in particular $u=0$ corresponds to the identity mapping.
It will be convenient to write the energy as a functional acting on $w=\xx+u$ so that $v(g)=\frac1{\abs{\G_{x_0}}}\sum_{g'\in g}g'\cdot w(g')$ for $g\in\Gcoset$ and the constant function $w=\xx$ corresponds to the identity deformation.
\begin{Definition}%\label{Definition:configurationalenergy}
The \emph{configurational energy} with interaction potential $V$ is 
\begin{align*}
\begin{split}
E\colon&\UPer\to\R\\
&w\mapsto\frac1{\abs{\CC_N}}\sum_{g\in\CC_N}V\parens[\Big]{\parens[\Big]{\frac1{\abs{\G_{x_0}}}\sum_{h'\in h}(gh')\gdot w(gh')-\frac1{\abs{\G_{x_0}}}\sum_{h'\in\G_{x_0}}(gh')\gdot w(gh')}_{h\in\Gstar}}
\end{split}
\end{align*}
if $w$ is $\T^N$-periodic and $N\in\MM$.
\end{Definition}
\begin{Remark}
The function $E$ is well-defined and independent of the choice of the representation set $\CC_N$ for all $N\in \MM$.
\end{Remark}
\begin{Lemma}\label{Lemma:FrechetE}
The function $E$ is two times continuously Fr\'echet differentiable with respect to the uniform norm $\norm\fdot_\infty$.
We have
\begin{align*}
&E(\xx)=V(y_0),\\
&E'(\xx)u=\frac1{\abs{\CC_N}}\sum_{g\in\CC_N}V'(y_0)\parens[\big]{\rot(h)\proj{u}(gh)-\proj{u}(g)}_{h\in\Gstar},
\quad \text{and}\\
&E''(\xx)(u,v)\\
&\qquad=\frac1{\abs{\CC_N}}\sum_{g\in\CC_N}V''(y_0)\parens[\Big]{\parens[\big]{\rot(h)\proj{u}(gh)-\proj{u}(g)}_{h\in\Gstar},\parens[\big]{\rot(h)\proj{v}(gh)-\proj{v}(g)}_{h\in\Gstar}}
\end{align*}
for all $u,v\in\UPer$ and $N\in \MM$ such that $u$ and $v$ are $\T^N$-periodic.
\end{Lemma}
(Note that, for $h\in\G$, $\rot(h)\proj{u}(gh)=\rot(h)\pproj{u(g\fdot)}(h)$ only depends on $h\G_{x_0}$.) The proof of Lemma~\ref{Lemma:FrechetE} is postponed to Section~\ref{subsec:prelim}. 
\begin{Remark}
\begin{enumerate}
\item If the map in \ref{item:Vfrechet} is $n$ times (continuously) Fr\'echet differentiable for some natural number $n$, then also $E$ is $n$ times (continuously) Fr\'echet differentiable with respect to the uniform norm $\norm\fdot_\infty$.
The proof is analogous.
\item The function $E$ need not be continuous with respect to the norm $\norm\fdot_2$.
In particular $E$ is not two times Fr\'echet differentiable with respect to $\norm{\nabla_\RR\fdot}_2$ although in other models a similar proposition is true, see, \eg, \cite[Theorem 1]{Ortner2013}.
\end{enumerate}
\end{Remark}
We fix an admissible neighborhood range $\RR\subset\G$ of $\id$. Furthermore we assume that $\G\cdot x_0$ is not the trivial structure $\{x_0\}$ such that $\lambdaa<\infty$ and $\lambdanewa<\infty$ in the following.  
\begin{Definition}\label{Definition:Stable}
We say that $w\in\UPer$ is a \emph{critical point} of $E$ if $E'(w)=0$.
We say that $(\G,x_0,V)$ is \emph{stable (in the atomistic model)} with respect to $\norm\fdot_\RR$ (resp.\ $\newnorm\fdot\RR$) if $\xx$ is a critical point of $E$ and the bilinear form $E''(\xx)$ is coercive with respect to $\norm\fdot_\RR$ (resp.\ $\newnorm\fdot\RR$), \ie there exists a constant $c>0$ such that, respectively, 
\[c\norm u_{\RR(,0,0)}^2\le E''(\xx)(u,u)\qquad\text{for all }u\in\UPer.\]
We define the corresponding constants $\lambdaa, \lambdanewa \in\R\cup\{-\infty\}$ by 
\begin{align*}
&\lambda_{\textnormal{a}(,0,0)}
:=\sup\set[\big]{c\in\R}{\forall\+ u\in\UPer:c\norm u_{\RR(,0,0)}^2\le E''(\xx)(u,u)}\in\R\cup\{-\infty\}.
\end{align*}
\end{Definition}
\begin{Remark}\label{Remark:BilinearFormLambda}
\begin{enumerate}
\item\label{item:BilinearFormLambda} The bilinear form $E''(\xx)$ is coercive with respect to the seminorm $\norm\fdot_\RR$ (resp.\ $\newnorm\fdot\RR$) if and only if $\lambdaa>0$ (resp.\ $\lambdanewa>0$).
\item If $(\G,x_0,V)$ is stable with respect to $\newnorm\fdot\RR$, then $(\G,x_0,V)$ is also stable with respect to $\norm\fdot_\RR$, since $\norm\fdot_{\RR}\le\newnorm\fdot{\RR}$.
\item The above definition of the stability and the constant $\lambdaa$ generalizes the definition in \cite{Hudson2011,Braun2016} where these terms are defined for lattices.
For lattices we have $\lambdaa=\lambdanewa$ since then $\norm\fdot_\RR=\newnorm\fdot\RR$.
\item By Theorem~\ref{Theorem:StrongerEquivalence} the stability of $(\G,x_0,V)$ is independent of the choice of $\RR$.
\item\label{item:lambdanotfinite} The constants $\lambdaa$ and $\lambdanewa$ need not be finite, see Example~\ref{Example:DEUiso} and Example~\ref{Example:counterV}.
In Theorems~\ref{Theorem:BoundedThree}%, \ref{Theorem:FiniteSpace} and 
--\ref{Theorem:dtwodtwo} we present sufficient conditions for both $\lambdaa\in\R$ and $\lambdanewa\in\R$.
\end{enumerate}
\end{Remark}
The following elementary proposition states a characterization of $\lambdaa$ and $\lambdanewa$ by means of the dual problem. For completeness we include its short proof in Section~\ref{subsec:prelim}.
\begin{Proposition}\label{Proposition:LambdaDual}
We have (with matching choices of indices $0$)
\begin{align*}
\lambda_{\textnormal{a}(,0,0)}=\inf\set{E''(\xx)(u,u)}{u\in\UPer,\norm u_{\RR(,0,0)}=1}.
\end{align*}
\end{Proposition}

\subsubsection*{Computationally amenable representations}
In our main stability criteria and the corresponding stability algorithm we will need to characterize the first and second derivatives of $E$ as well as $\norm\fdot_{\RR(,0,0)}$ in a computationally amenable way. To do so, we introduce the quantities $\ee, \ff$ and $g_{\RR(,0,0)}$ that will allow us to write $E'$, $E''$ and $\norm\fdot_{\RR(,0,0)}$ in terms of convolution operators, cf.\ Lemma~\ref{Lemma:energyderivativeY}, \eqref{eq:Azu1} and \eqref{eq:E-and-norm-as-conv} below. We begin with $\ee$ and $\ff$.
\begin{Definition}\label{Definition:eeff}
We define the row vector  
\[\ee:=\sum_{g\in\Gstar}\partial_g V(y_0)\parens[\big]{L_g-L_{\G_{x_0}}}\in\R^d, \]
where $L_g:=\frac1{\abs{\G_{x_0}}}\sum_{g'\in g}\rot(g')$ for $g\in\Gcoset$, and the function $\ff\in L^1(\G,\R^{d\times d})$ by
\begin{align*}
\begin{split}
\begin{aligned}
\ff(g):=\sum_{h_1,h_2\in\Gstar}\frac1{\abs{\G_{x_0}}^2}&\sum_{h_1',h_2'\in\G}\delta_{g,h_2'^{-1}h_1'}\rot(h_2')^{\mathsf T}\partial_{h_2}\partial_{h_1}V(y_0)\rot(h_1')\\
&\quad\cdot(\chi_{h_1}(h_1')-\chi_{\G_{x_0}}(h_1'))(\chi_{h_2}(h_2')-\chi_{\G_{x_0}}(h_2'))\qquad\text{for all }g\in\G.
\end{aligned}
\end{split}
\end{align*}
\end{Definition}
\begin{Remark}\label{Remark:ee}
\begin{enumerate}
\item By \ref{item:infty} the function $\ff$ is well-defined and we have
\begin{align*} 
\sum_{g\in\G}\ff(g)&=\sum_{h_1,h_2\in\Gstar}L_{\G_{x_0}}^{\mathsf T}\parens{\rot(h_2)-I_d}^{\mathsf T}\partial_{h_2}\partial_{h_1}V(y_0)\parens{\rot(h_1)-I_d}L_{\G_{x_0}}.
\end{align*}
\item\label{item:eefinite}If $\RRVset/\G_{x_0}\subset\Gstar$ is an interaction range of $V$, then we have
\[\supp\ff\subset\RRVset^{-1}\RRVset\cup\RRVset^{-1}\cup\RRVset.\]
In particular, if $V$ has finite interaction range, then the support of $\ff$ is finite.
\end{enumerate}
\end{Remark}
\begin{Definition}\label{Definition:PartialHessian}
For all $N\in\MM$ and $g,h\in\G/\T^N$ we define the partial Jacobian row vector $\partial_gE(\xx)\in\R^d$ and the partial Hessian matrix $\partial_g\partial_hE(\xx)\in\R^{d\times d}$ by
\begin{align*}
\parens{\partial_gE(\xx)}_i:=E'(\xx)(\indi_ge_i)
\qquad\text{and}\qquad
\parens{\partial_g\partial_hE(\xx)}_{ij}:=E''(\xx)(\indi_ge_i,\indi_he_j)
\end{align*}
for all $i,j\in\{1,\dots,d\}$.
\end{Definition}
The following lemma characterizes these derivatives and, in particular, shows that $\partial_g E(\xx)$ is independent of $g$ and $\partial_{g_2}\partial_{g_1} E(\xx)$ only depends on $g_2^{-1}g_1$.
\begin{Lemma}\label{Lemma:energyderivativeY}
Let $N\in \MM$.
For all $g,g_1,g_2\in\G/\T^N$ we have
\begin{align*}
\partial_g E(\xx)=\frac1{\abs{\CC_N}}\ee
\qquad\text{and}\qquad
\partial_{g_2}\partial_{g_1} E(\xx)=\frac1{\abs{\CC_N}}\sum_{g\in g_2^{-1}g_1}\ff(g).
\end{align*}

In particular, $\xx$ is a critical point of $E$ if and only if $\ee=0$. 
\end{Lemma}
Thus, by Lemma~\ref{Lemma:energyderivativeY} and Remark~\ref{Remark:BilinearFormLambda}\ref{item:BilinearFormLambda} the triple $(\G,x_0,V)$ is stable with respect to $\norm\fdot_\RR$ (resp.\ $\newnorm\fdot\RR$) if and only if $\ee=0$ and $\lambdaa>0$ (resp.\ $\lambdanewa>0$). We prove this lemma in Section~\ref{subsec:prelim}. 

Now we construct $g_{\RR(,0,0)}$. We fix a bijection $\phi\colon\RR\to\{0,\dots,\abs\RR-1\}$ and define an isomorphism between $\C^{(m\abs\RR)\times n}$ and $(\C^{m\times n})^\RR$ by
\[(a_{i,j})_{i\in\{1,\dots,m\abs\RR\};j\in\{1,\dots,n\}}\mapsto\parens[\big]{(a_{i+m\phi(g),j})_{i\in\{1,\dots,m\};j\in\{1,\dots,n\}}}_{g\in\RR}.\]
\begin{Definition}\label{Definition:gggg}
We define the functions $\gggg,\newgggg\in L^1(\G,\R^{(d\abs\RR)\times d})$ by
\begin{align*}
\gggg(g)=P\parens[\big]{\delta_{g,h}I_d}_{h\in\RR}
\qquad \text{and}\qquad 
\newgggg(g)=P_0\parens[\big]{\delta_{g,h}I_d}_{h\in\RR}\end{align*}
for all $g\in\G$,
where $P$ and $P_0$ are square matrices of order $d\abs\RR$ such that the map $\R^{d\abs\RR}\to\R^{d\abs\RR}$ with $x\mapsto Px$, respectively, $x\mapsto P_0x$, is the orthogonal projection with respect to the norm $\norm\fdot$ with kernel $\Uiso\RR$, respectively, $\Unewiso\RR$.
\end{Definition}
We then have
\begin{equation}\label{eq:Azu1}
\norm u_{\RR(,0,0)}^2=\frac1{\abs{\CC_N}}\sum_{g\in\CC_N}\norm[\big]{P_{(0)}(u(gh))_{h\in \RR}}^2.
\end{equation}

\begin{Remark}\label{Remark:newee}
\begin{enumerate}
\item\label{item:abdefd} The configurational energy is left-translation invariant, \ie for all $w\in\UPer$ and $g\in\G$ it holds $E(w)=E(w(g\fdot))$.
This implies that also $E'(\xx)$ and $E''(\xx)$ are left-translation invariant, \ie $E'(\xx)u=E'(\xx)u(g\fdot)$ and $E''(\xx)(u,v)=E(\xx)(u(g\fdot),v(g\fdot))$ for all $u\in\UPer$ and $g\in\G$.
This directly shows that $\partial_{g_1}E(\xx)=\partial_\id E(\xx)$ and $\partial_{g_2}\partial_{g_1}E(\xx)=\partial_\id\partial_{g_2^{-1}g_1} E(\xx)$ for all $N\in \MM$ and $g_1,g_2\in\G/\T^N$.

Also the seminorms $\norm\fdot_\RR$ and $\newnorm\fdot\RR$ are left-translation invariant.
\item\label{item:newee}If $V$ has finite interaction range $\RRVset/\G_{x_0}$ so that by Remark~\ref{Remark:ee}\ref{item:eefinite} $\supp\ff\subset\RRVset^{-1}\RRVset\cup\RRVset^{-1}\cup\RRVset=:\RR_\ff$, then the above lemma shows that %
\[\ff(g)
=\abs{\CC_N}\partial_\id\partial_{g\T^N} E(\xx) \quad\text{for all }g\in\RR_\ff\]
for all $N\in\MM$ large enough, precisely for $N\in\MM$ such that $\T^N\cap\RR_\ff^{-1}\RR_\ff\subset\{\id\}$. 
\item\label{Remark-item:gggg}
The support of both $\gggg$ and $\newgggg$ is equal to $\RR$.
We have
\begin{align*}
\gggg(g)=p_{\phi(g)}
\qquad\text{and}\qquad
\newgggg(g)=p_{0,\phi(g)}
\end{align*}
for all $g\in\RR$ where $P$ and $P_0$ are as above and $p_0,\dots,p_{\abs\RR-1},p_{0,0},\dots,p_{0,\abs\RR-1}\in\R^{(d\abs\RR)\times d}$ such that $P=(p_0,\dots,p_{\abs\RR-1})$ and $P_0=(p_{0,0},\dots,p_{0,\abs\RR-1})$.
\end{enumerate}
\end{Remark}
\subsection{Main results: Stability criteria and second order bounds}\label{Section:Stability-and-Boundedness}
Our main results establish upper bounds and criteria for matching lower bounds on $E''(\xx)$ in terms of $\norm\fdot_\RR$ and $\newnorm\fdot\RR$. We thus provide energy bounds and stability criteria for general objective structures. Their proofs are given in Section~\ref{section:proofs}
\subsubsection*{Energy bounds}
We recall that a bilinear form $B$ on a real vector space $W$ is said to be bounded with respect to a seminorm $\norm\fdot$ on $W$ if there exists a constant $C>0$ such that
\[\abs{B(v,w)}\le C\norm v\norm w\qquad\text{for all }v,w\in W.\]
If $B$ is symmetric this is the case if and only if there exists a constant $C>0$ such that
\[\abs{B(v,v)}\le C\norm v^2\qquad\text{for all }v\in W.\]

In the following four results we assume that $V$ has finite interaction range. 
Indeed, bounds  with respect to the strong seminorm $\norm{\nabla_\RR\fdot}_2$ turn out rather straightforward.
\begin{Proposition}\label{Proposition:BoundedOne}
The bilinear form $E''(\xx)$ is bounded with respect to $\norm{\nabla_\RR\fdot}_2$.
\end{Proposition}

In view of our stability criteria to be discussed, the question arises, if such bounds can be established for the weaker seminorms $\norm\fdot_{\RR(,0,0)}$. 
Our first observation is that, under the assumption that the structure is an equilibrium configuration of each individual onsite potential given by $V$, \ie $V'(y_0)=0$ (recall that $y_0=(g\gdot x_0-x_0)_{g\in\Gstar}$), then $E''(\xx)$ is bounded even with respect to $\norm\fdot_\RR$. 
\begin{Theorem}\label{Theorem:BoundedThree}
Suppose that $V'(y_0)=0$.
Then $E''(\xx)$ is bounded with respect to $\norm\fdot_\RR$.
In particular we have $\lambdaa\in\R$ and $\lambdanewa\in\R$.
\end{Theorem}
We proceed to consider general critical points, first when $\G$ is finite or a space group. In these cases we obtain boundedness with respect to the weak seminorm $\norm\fdot_\RR$. 
\begin{Theorem}\label{Theorem:FiniteSpace}
\begin{enumerate}
\item\label{Theorem-item:Finite} 
Suppose that $\G$ is finite and $E'(\xx)=0$.
Then $E''(\xx)$ is bounded with respect to $\norm\fdot_\RR=\newnorm\fdot\RR$.
In particular, we have $\lambdaa=\lambdanewa\in\R$.
\item\label{Theorem-item:Space}
Suppose that $\G$ is a space group.
Then $E''(\xx)$ is bounded with respect to both $\norm\fdot_\RR$ and $\newnorm\fdot\RR$.
In particular, we have $\lambdaa,\lambdanewa\in\R$.
\end{enumerate}
\end{Theorem}

If $\G$ is neither finite nor a space group, the boundedness of $E''(\xx)$ is non-trivial in general. While an estimate with respect to $\newnorm\fdot\RR$ is straightforward if $d_1=1$, the case $d_1\ge2$ is rather demanding. 

\begin{Theorem}\label{Theorem:dtwodtwo}
Suppose that $E'(\xx)=0$ and $E''(\xx)$ is positive semidefinite.
Then $E''(\xx)$ is bounded with respect to $\newnorm\fdot\RR$. In particular, we have $\lambdanewa\in\R$.
\end{Theorem}
\begin{Remark}\label{Remark:Hessianbounds}
\begin{enumerate}
\item\label{Eprimezerofinite}Example~\ref{Example:DEUiso} below shows that the condition $E'(\xx)=0$ in Theorem~\ref{Theorem:FiniteSpace}\ref{Theorem-item:Finite} cannot be dropped. 
\item\label{Remarkitem:donedtwo}
The proof of Theorem~\ref{Theorem:dtwodtwo} shows that the assumptions $E'(\xx)=0$ and that $E''(\xx)$ is positive semidefinite can be dropped if $d=1+d_2$. 
Example~\ref{Example:counterV} below shows that here the assumption of a finite interaction range is essential. 
\end{enumerate}
\end{Remark}

\subsubsection*{Representation formulae for stability constants} 

We recall the definition of $\ff$ and $\gggg,\newgggg$ from Definition~\ref{Definition:eeff} and~\ref{Definition:gggg}, respectively. Since $E''(\xx)$, $\norm\fdot_\RR$ and $\newnorm\fdot\RR$ are left-translation invariant, see Remark~\ref{Remark:newee}\ref{item:abdefd}, they can be represented by means of multiplier operators as 
\begin{align}\label{eq:E-and-norm-as-conv}
E''(\xx)(u,v)=\angles{\ff*v_0,u_0} 
\qquad\text{and}\qquad 
\norm u_{\RR(,0,0)}=\norm{{g_{\RR(,0,0)}}*u_0}_2,
\end{align}
where $u_0=u(\fdot^{-1})$ and $v_0=v(\fdot^{-1})$, see Lemma~\ref{Lemma:onetwoL} below. This directly leads to a first characterization of the stability constants $\lambdaa$ and $\lambdanewa$ in the Fourier transform domain. 
Recall that $\EE$ is a representation set of $\set{\rho\in\dual\G}{\rho\text{ is periodic}}$.
We write 
\begin{align}\label{eq:lambdamindef}
 \lambdamin(A,B):=\sup\set[\big]{c\in\R}{cB^{\mathsf H}B\le A}\in\R\cup\{\pm\infty\}
\end{align}
for all Hermitian matrices $A\in\C^{n\times n}$ and matrices $B\in\C^{m\times n}$. Here $\le$ denotes the \emph{Loewner order} on the Hermitian $n\times n$ matrices defined by $A\ge A'$ if $A-A'$ is positive semidefinite.
\begin{Theorem}\label{Theorem:LambdaIrreducible}
We have
\begin{align*}
&\lambdaa=\inf\set[\Big]{\lambdamin\parens[\Big]{\fourier\ff(\rho),\fourier\gggg(\rho)}}{\rho\in\EE}
\shortintertext{and}
&\lambdanewa=\inf\set[\Big]{\lambdamin\parens[\Big]{\fourier\ff(\rho),\fourier\newgggg(\rho)}}{\rho\in\EE}.
\end{align*}
\end{Theorem}
Theorem~\ref{Theorem:LambdaIrreducible} is an abstract representation result in the sense that $\EE$ might not be easily accessible in concrete examples. For a meaningful and applicable generalization of the representation results in \cite{Hudson2011,Braun2016} from lattices to general objective structures, it is essential to observe that one may reduce to representations that are induced by representations $\rho$ of $\T\F$, which in turn are characterized in terms of suitable wave vector domains $K_\rho$. More precisely, we first fix a complete set of representatives of the cosets of $\T\F$ in $\G$ such that $\Ind\rho$ is well-defined for all $\rho$ by Definition~\ref{Definition:IndRep}. We recall the equivalence relation $\simrg$ on $\dual{\T\F}$ from  Definition~\ref{Definition:ChiK-L}\ref{Definition:RelationTF}, choose a finite representation set of $\dual{\T\F}/{\simrg}$, \eg with the aid of the finite group $(\T\F)_{m_0}$ and Lemma~\ref{Lemma:RepSys}, and recall the $d_2$-dimensional space groups $\G_\rho$ for $\rho\in\dual{\T\F}$ from Definition~\ref{Definition:Grho} and Proposition~\ref{Proposition:SpaceGroupLattice}. 
\begin{Theorem}\label{Theorem:LambdaInduced}
Let $R$ be a representation set of a representation set of $\dual{\T\F}/{\simrg}$.
For all $\rho\in R$ let $K_\rho\subset\R^{d_2}$ such that its closure contains a representation set of $\R^{d_2}/\G_\rho$.
Then 
\begin{align*}
&\lambdaa=\inf\set[\Big]{\lambdamin\parens[\Big]{\fourier\ff(\Ind(\chi_k\rho)),\fourier\gggg(\Ind(\chi_k\rho))}}{\rho\in R,k\in K_\rho}
\shortintertext{and}
&\lambdanewa=\inf\set[\Big]{\lambdamin\parens[\Big]{\fourier\ff(\Ind(\chi_k\rho)),\fourier\newgggg(\Ind(\chi_k\rho))}}{\rho\in R,k\in K_\rho}.
\end{align*}
\end{Theorem}
\begin{Remark}\label{Remark:Python}
By means of the dual problem we have
\begin{align*}
\lambdamin(A,B)=\inf\set[\big]{x^{\mathsf H}Ax}{x\in\C^n,\norm{Bx}=1}
\end{align*}
for all Hermitian matrices $A\in\C^{n\times n}$ and matrices $B\in\C^{m\times n}\setminus\{0\}$.
Suppose that $B$ has in addition rank $n$ and consider the \emph{generalized} eigenvalue problem $Av=\lambda B^{\mathsf H}Bv$, \ie the problem of finding the \emph{eigenvalues} of the \emph{matrix pencil} $A-\lambda B^{\mathsf H}B$.
Then the eigenvalues of the generalized eigenvalue problem are real and $\lambdamin(A,B)$ is equal to the smallest one, see \cite[Chapter X, Theorem 11]{Gantmacher1998}.
The eigenvalues of the generalized eigenvalue problem are equal to the eigenvalues of the matrix $A(B^{\mathsf H}B)^{-1}$, see \cite[Proposition 6.1.1]{Watkins2007}, but the eigenvalues of $A(B^{\mathsf H}B)^{-1}$ are ill-conditioned. % spectrum(AB)=spectrum(BA) (Sylvester's determinant identity)
There exist many numerically stable algorithms, see, \eg, \cite[Chapter 5]{Bai2000}, and thus many programming languages have a function for this problem; \eg for Python the subpackage linalg of the package SciPy has the function eigvalsh.
\end{Remark}
We finally state the following theorem which gives a sufficient condition for a stable structure to be a minimum point of $E$ in case $\G$ is finite, a space group, or $d_1=1$.
\begin{Theorem}\label{Theorem:local-min}
Suppose that $d_1\in\{0,1,d\}$, $V$ has finite interaction range, $\ee=0$ and $\lambdanewa>0$.
Then $E$ has a local minimum point at $\xx$ with respect to $\norm\fdot_\infty$, \ie there exists a neighborhood $U\subset\UPer$ of $0$ with respect to $\norm\fdot_\infty$ such that
\[E(\xx+u)\ge E(\xx)\qquad\text{for all }u\in U.\]
\end{Theorem}
\begin{Remark}
The proof will show that in case $d_1\in\{0,1\}$ there even exists a neighborhood $U\subset\UPer$ of $\xx$ with respect to $\norm\fdot_\infty$ such that
\[E(\xx+u)\ge E(\xx)+\tfrac\lambdanewa{4}\newnorm u\RR^2\qquad\text{for all }u\in U.\]
\end{Remark}
%
%%%%%%%%%%%%%%%%%%%%%
%
\section{A stability algorithm and applications}%\label{Section:Algorithm}
\subsection{The general algorithm}
In view of our main results we can now give an algorithm which checks if $(\G, x_0, V)$ is stable with respect to $\norm\fdot_\RR$, see Definition~\ref{Definition:Stable}.
The algorithm for the stability with respect to $\newnorm\fdot\RR$ is analogous.
\begin{Algorithm}\label{Algorithm:algorithm}
Given is a discrete group $\G<\E(d)$ and its associated groups $\F$, $\SG$ and set $\T$, some point $x_0\in\R^d$, the stabilizer subgroup $\G_{x_0}$, and an interaction potential $V$, see Definition~\ref{Definition:potential}.
Since the algorithm is numeric and by \ref{item:infty}, we may assume that $V$ has finite interaction range $\RRVset/\G_{x_0}$.
\begin{enumerate}
\item\label{item:A1} Check if $\xx$ is a critical point of the configurational energy $E$, \eg by computing the derivative $\partial_g V(y_0)$ for all $g\in\supp V$, see Definition~\ref{Definition:potential}, the vector $\ee$, see Definition~\ref{Definition:eeff}, and checking if $\ee=0$, see Lemma~\ref{Lemma:energyderivativeY}.
\item\label{item:A2} Determine the derivative $\partial_g\partial_h V(y_0)$ for all $g,h\in\supp V$, see Definition~\ref{Definition:potential}.
Then compute the function $f_V$ by computing $f_V(g)$ for all $g\in\RRVset^{-1}\RRVset\cup\RRVset^{-1}\cup\RRVset$, see Definition~\ref{Definition:eeff} and Remark~\ref{Remark:ee}\ref{item:eefinite}.
\item\label{item:A3} Determine an admissible $\id$-neighborhood $\RR$, see Definition~\ref{Definition:Property}.
Fix a bijection $\phi\colon\RR\to\{0,\dots,\abs{\RR}-1\}$.
Thus the map
\[\psi\colon\Uiso\RR\hookrightarrow\R^{d\abs\RR},\quad u\mapsto (u(\phi^{-1}(0)),\dots,u(\phi^{-1}(\abs\RR-1)))^{\mathsf T},\]
which maps a function to a column vector, is an embedding, where $\Uiso\RR$ is defined in Definition~\ref{Definition:UtransUrot}.
By Proposition~\ref{Proposition:UtransUrotbig} and the Gram-Schmidt process, we can determine an orthonormal basis $\{b_1,\dots,b_n\}$ of $\psi(\Uiso\RR)$, where $n=\dim(\Uiso\RR)$.
Let $B$ be the $d\abs\RR\times n$ matrix $(b_1,\dots,b_n)$.
The matrix $I_{d\abs\RR}-BB^{\mathsf T}$ is the orthogonal projection matrix with kernel $\psi(\Uiso\RR)$.
Now we can determine the function $\gggg$, \ie the matrix $\gggg(g)$ for all $g\in\RR$, see Definition~\ref{Definition:gggg} and Remark~\ref{Remark:newee}\ref{Remark-item:gggg}.
\item\label{item:A4} Determine a representation set $R$ of $\dual{\T\F}/{\simrg}$, \eg with Lemma~\ref{Lemma:RepSys}, where $\simrg$ is the equivalence relation defined in Definition~\ref{Definition:ChiK-L}\ref{Definition:RelationTF}.
For all $\rho\in R$ determine the space group $\G_\rho$, see Definition~\ref{Definition:Grho}, with, \eg, Proposition~\ref{Proposition:SpaceGroupLattice}, and determine $K_\rho\subset\R^{d_2}$ as a fundamental domain (or merely a set whose closure contains a representation set) of $\R^{d_2}/\G_\rho$.
\item\label{item:A5} Fix a complete set of representatives of the cosets of $\T\F$ in $\G$.
Thus the induced representation $\Ind(\chi_k\rho)$ is well-defined for all $\rho\in R$ and $k\in K_\rho$, see Definitions~\ref{Definition:ChiK-L}\ref{Definition:ChiK} and~\ref{Definition:IndRep}.
For all $\rho\in R$ and $k\in K_\rho$ the matrices $\fourier\ff(\Ind(\chi_k\rho))$ and $\fourier\gggg(\Ind(\chi_k\rho))$, see Definition~\ref{Definition:FourierLOne}, and the extended real number $\lambdamin(\fourier\ff(\Ind(\chi_k\rho)),\fourier\gggg(\Ind(\chi_k\rho)))$, see~\eqref{eq:lambdamindef}, can be determined.
Theorem~\ref{Theorem:LambdaInduced} then allows to compute the extended real number $\lambdaa$.

If $\G_{x_0}$ is trivial, the following is helpful for the computation of $\lambdaa$: Then the dimension of the kernel of $\norm\fdot_\RR$ is finite and, if the $K_\rho$ are fundamental domains, for all $\rho\in R$ and all but finitely many $k\in K_\rho$, the matrix $\fourier\gggg(\Ind(\chi_k\rho))$ has full rank and thus $\lambdamin(\fourier\ff(\Ind(\chi_k\rho)),\fourier\gggg(\Ind(\chi_k\rho)))$ is even a real number and can easily be computed, see Remark~\ref{Remark:Python}.
\item\label{item:A6} The triple $(\G,x_0, V)$ is stable with respect to $\norm\fdot_\RR$ if and only if $\xx$ is a critical point of $E$ and $\lambdaa>0$, see Definition~\ref{Definition:Stable}.
\end{enumerate}
\end{Algorithm}
In the following two examples, we investigate the stability of a triple $(\G_i,x_i,V_i)$ for all $i\in I$, where $I$ is a suitable index set.
The figures are generated with the programming language Python, see \url{https://github.com/Toymodel-Nanotube/} for the source code.
\subsection{Example: a chain of atoms}
\begin{Example}\label{Example:ToyModel}
A suitable toy model for the investigation of stability is an atom chain.

Let $a>0$ be the scale factor, $t=t_a=\iso{I_2}{ae_2}\in\E(2)$ and $\G=\G_a=\angles t<\E(2)$.
We define the interaction potential $V = V_a$, see Definition~\ref{Definition:potential} and Remark~\ref{Remark:VFrechet}\ref{item:RemarkVDomain}-\ref{item:RemarkGTrivial}, by
\[V_a(y)=v_1(\norm{y(t_a)})+v_2(\norm{y(t_a^2)}),\]
where
\begin{align*}
&v_1\colon(0,\infty)\to\R,\quad r\mapsto r^{-12}-r^{-6}
\intertext{is the Lennard-Jones potential and}
&v_2\colon(0,\infty)\to\R,\quad r\mapsto 8r^{-6}.
\end{align*}
Let $x_0=0_2$.
The stabilizer subgroup $\G_{x_0}$ is trivial.
By Lemma~\ref{Lemma:FrechetE} for all $a>0$ we have
\[E(\xx)=V(y_0)=a^{-12}-\frac78a^{-6},\]
where $E=E_a$ is the configurational energy and $y_0=y_{0,a}=(g\cdot x_0-x_0)_{g\in\G_a}$.
We define
\[a^* :=\argmin_{a\in(0,\infty)}E(\xx)=\sqrt[6]{\frac{16}7}\approx 1.1477.\]
Thus the structure $\G\cdot x_0$ is stretched (resp.\ compressed) if $a>a^*$ (resp.\ $a<a^*$).
Now we investigate its stability numerically with Algorithm~\ref{Algorithm:algorithm}.
\begin{enumerate}
\item[\ref{item:A1}] Since $\G<\{I_d\}\ltimes\R^d$ we have $\ee=0$ and thus $\xx$ is a critical point of $E$ for all $a>0$ by Lemma~\ref{Lemma:energyderivativeY}.
\item[\ref{item:A2}] We have
\begin{align*}
\partial_g\partial_h V(y_0)=\begin{cases}
6a^{-8}\parens[\bigg]{\begin{matrix}-2a^{-6}+1&0\\0&26a^{-6}-7\end{matrix}} & \text{if }g=h=t\\
2^{-4}3a^{-8}\parens[\bigg]{\begin{matrix}-1&0\\ 0&7\end{matrix}}& \text{if }g=h=t^2\\
0_{2,2}& \text{else.}
\end{cases}
\end{align*}
We have $(\{\id\}\cup\supp V)^{-1}(\{\id\}\cup\supp V)=\{t^{-2},\dots,t^2\}$ and
\begin{align*}
\ff(g)=\begin{cases}
a^{-8}\parens[\bigg]{\begin{matrix}-24a^{-6}+93/8&0\\ 0&312a^{-6}-651/8\end{matrix}}& \text{if }g=\id\\
6a^{-8}\parens[\bigg]{\begin{matrix}2a^{-6}-1&0\\0&-26a^{-6}+7\end{matrix}} & \text{if }g\in\{t^{-1},t\}\\
2^{-4}3a^{-8}\parens[\bigg]{\begin{matrix}1&0\\ 0&-7\end{matrix}}& \text{if }g\in\{t^{-2},t^2\}\\
0_{2,2}& \text{else.}
\end{cases}
\end{align*}
\item[\ref{item:A3}] Since $\aff(\{\id, t\}\cdot x_0)=\{0\}\times\R=\aff(\G\cdot x_0)$ and $\{t\}$ generates $\G$, the set $\RR=\{\id, t, t^2\}$ is an admissible $\id$-neighborhood.
We define the functions
\begin{align*}
&b_i\colon\RR\to\R^2,\quad g\mapsto e_i&&\text{for all }i\in\{1, 2\}
\shortintertext{and}
&b_3\colon\RR\to\R^2,\quad g\mapsto \parens[\Big]{\begin{smallmatrix}0&-1\\1&0\end{smallmatrix}}(g\cdot x_0-x_0).
\end{align*}
By Proposition~\ref{Proposition:UtransUrotbig} the sets $\{b_1,b_2,b_3\}$ and $\{b_1,b_2\}$ are bases of $\Uiso\RR$ and of $\Unewiso\RR$, respectively.
We define the bijection $\phi\colon\RR\to\{0,1,2\}$ by $t^n\mapsto n$ for all $n\in\{0,1,2\}$.
Let $\psi$ be the embedding 
\[\Uiso\RR\hookrightarrow\R^6,\quad u\mapsto(u(\phi^{-1}(0)),\dots,u(\phi^{-1}(2))).\]
A computation shows that the orthogonal projection matrices of $\R^6$ with kernels $\psi(\Uiso\RR)$ and $\psi(\Unewiso\RR)$ are
\[\frac16\left(\begin{matrix}
1&0&-2&0&1&0\\
0&4&0&-2&0&-2\\
-2&0&4&0&-2&0\\
0&-2&0&4&0&-2\\
1&0&-2&0&1&0\\
0&-2&0&-2&0&4
\end{matrix}\right)\quad \text{and}\quad\frac13\left(\begin{matrix}
2&0&-1&0&-1&0\\
0&2&0&-1&0&-1\\
-1&0&2&0&-1&0\\
0&-1&0&2&0&-1\\
-1&0&-1&0&2&0\\
0&-1&0&-1&0&2
\end{matrix}\right),\]
respectively.
Thus the functions $\gggg$ and $\newgggg$ of Definition~\ref{Definition:gggg} are given by
\begin{alignat*}{3}
&\supp\gggg=\supp\newgggg=\RR,\\
&\gggg(\id)=\frac16\left(\begin{matrix}
1&0\\
0&4\\
-2&0\\
0&-2\\
1&0\\
0&-2
\end{matrix}\right),
&&\gggg(t)=\frac16\left(\begin{matrix}
-2&0\\
0&-2\\
4&0\\
0&4\\
-2&0\\
0&-2
\end{matrix}\right),
&&\gggg(t^2)=\frac16\left(\begin{matrix}
1&0\\
0&-2\\
-2&0\\
0&-2\\
1&0\\
0&4
\end{matrix}\right)
\intertext{and}
&\newgggg(\id)=\frac13\left(\begin{matrix}
2&0\\
0&2\\
-1&0\\
0&-1\\
-1&0\\
0&-1
\end{matrix}\right),\;\,
&&\newgggg(t)=\frac13\left(\begin{matrix}
-1&0\\
0&-1\\
2&0\\
0&2\\
-1&0\\
0&-1
\end{matrix}\right),\;\,
&&
\newgggg(t^2)=\frac13\left(\begin{matrix}
-1&0\\
0&-1\\
-1&0\\
0&-1\\
2&0\\
0&2
\end{matrix}\right).
\end{alignat*}
\item[\ref{item:A4}] We have $\G=\T\F=\angles t$, $\MM=\N$ and $\{\id\}$ is a representation set of $\dual{\T\F}/{\simrg}$ by Lemma~\ref{Lemma:RepSys}\ref{item:aaa}.
We have $\SG=\angles{\iso{I_1}a}$, $\LL=\angles a$ and $\dualL=\angles{a^{-1}}$, see Definition~\ref{Definition:ChiK-L}\ref{Definition:LLst}.
By Proposition~\ref{Proposition:SpaceGroupLattice} we have $\set{k\in\R}{\iso{I_1}k\in\G_\id}=\angles{a^{-1}}$ and thus $\G_\id=\angles{\iso{I_1}{a^{-1}}}$.
The interval $K_\id=[0,a^{-1})$ is a representation set of $\R/\G_\id$.
\item[\ref{item:A5}] For all $k\in K_\id$ we have $\Ind_{\T\F}^\G \chi_k = \chi_k$.
Both $\gggg$ and $\newgggg$ satisfy
\begin{align*}
\set{k\in K_\id}{\fourier{g_{\RR(,0,0)}}(\chi_k)\text{ has full rank}}=K_\id\setminus\{0\}.
\end{align*}
The numbers 
$\lambdamin\parens{\fourier\ff(\chi_k),\fourier{g_{\RR(,0,0)}}(\chi_k)}$ can then be computed for all $k\in K_\id\setminus\{0\}$.
In particular, we can compute $\lambdaa=\lambdaa(a)$ and $\lambdanewa=\lambdanewa(a)$ numerically, see Figure~\ref{fig:ToyModel}.
\begin{figure}
    \centering
    \begin{subfigure}{0.48\textwidth}
        \includegraphics[width=\textwidth]{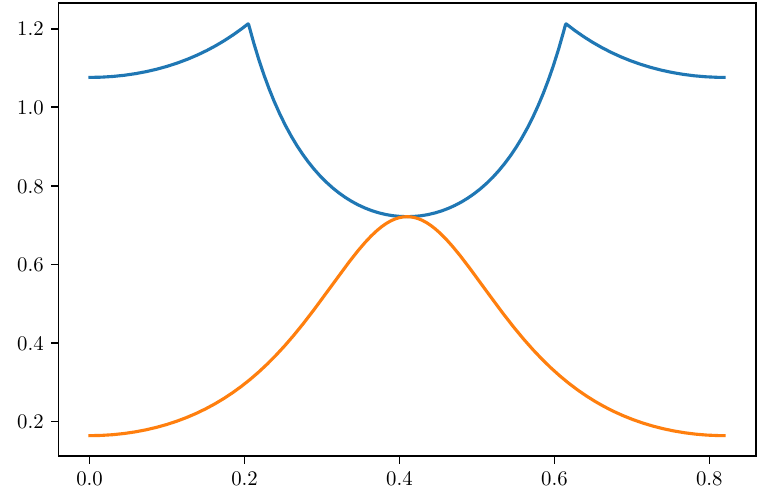}
    \end{subfigure}
    ~ 
    \begin{subfigure}{0.48\textwidth}
        \includegraphics[width=\textwidth]{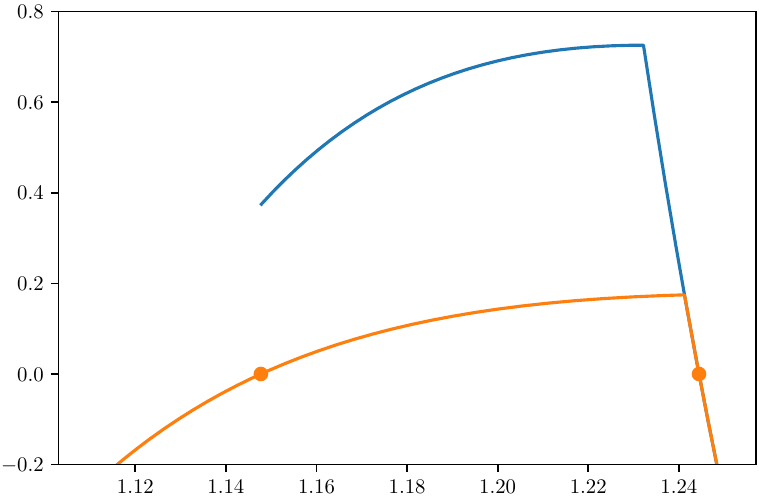}
    \end{subfigure}
    \caption{For the toy model as described in Example~\ref{Example:ToyModel}, the graphs of the numbers $\lambdamin\parens{\widehat{\ff}(\chi_k),\widehat{\gggg}(\chi_k)}$ (blue) and $\lambdamin\parens{\widehat{\ff}(\chi_k),\widehat{\newgggg}(\chi_k)}$ (orange) dependent on $k\in K_\id\setminus\{0\}$ are plotted on the left for the choice $a=1.22$. The points $(a^*,0)$ and $(a^{**},0)$  and the graphs of $\lambdaa$ (blue) and $\lambdanewa$ (orange) dependent on the scale factor are plotted on the right.}\label{fig:ToyModel}
\end{figure}
\item[\ref{item:A6}] In the compressed case $a\in(0,a^*)$ we have $\lambdaa=-\infty$ and $\lambdanewa\in(-\infty,0)$ and thus $(\G,x_0,V)$ is not stable with respect to both $\norm\fdot_\RR$ and $\newnorm\fdot\RR$.
Now we investigate the stretched case, \ie $a>a^*$.
We consider the `period doubling mode' $u=\chi_{\T^2}e_2$ and let $a^{**}>0$ such that $E''(\xx)(u,u)=0$, \ie $a^{**}=\sqrt[6]{26/7}\approx1.244455$.
Indeed for all $a\in(a^*,a^{**})$ we have $\lambdaa>0$ and $\lambdanewa>0$ and thus $(\G,x_0,V)$ is stable with respect to both $\norm\fdot_\RR$ and $\newnorm\fdot\RR$.
For all $a>a^{**}$ we obtain $E''(\xx)(u, u)<0$ and we have $\lambdaa<0$ and $\lambdanewa<0$ and thus $(\G,x_0,V)$ is not stable with respect to both $\norm\fdot_\RR$ and $\newnorm\fdot\RR$.
In particular, the loss of stability beyond $a^{**}$ is seen to result from a period doubling deformation mode. 
\end{enumerate}
Notice that in the stretched case $a\in(a^*,a^{**})$, the appropriate seminorm for the stability is $\newnorm\fdot\RR$. 
For the equilibrium case $a\approx a^*$, the weaker seminorm $\norm\fdot_\RR$ is appropriate since $\lim_{a\to a^*}\lambdanewa=0$ and $\lim_{a\searrow a^*}\lambdaa>0$.
\end{Example}
\subsection{Example: a carbon nanotube}
\begin{Example}\label{Example:Nanotube}
We now consider a carbon nanotube with non-trivial chirality. Single-walled carbon nanotubes are classified by an integer pair $(n,m)$ depending on the winding direction if the tube is visualized as a rolled-up graphene sheet. 
While there is a considerable amount of literature on the stability of the achiral zigzag (of type $(n,0)$) and armchair (of type $(n,n)$) variants, see in particular \cite{Friedrich2019} and the references therein, general chiral tubes are much less understood. To the best of our knowledge, the following analysis provides the first rigorous approach for an achiral case.

Since for any pair $(n,m)$ the nanotube is the orbit of some point in $\R^3$ under the action of a discrete subgroup of $\E(3)$, our stability analysis applies. 
By way of example we will now investigate the stability of a $(5,1)$ nanotube, see Figure~\ref{fig:(5,1)_Nanotube} with Algorithm~\ref{Algorithm:algorithm}.
\begin{figure}
    \centering
        \includegraphics[width=0.7\textwidth]{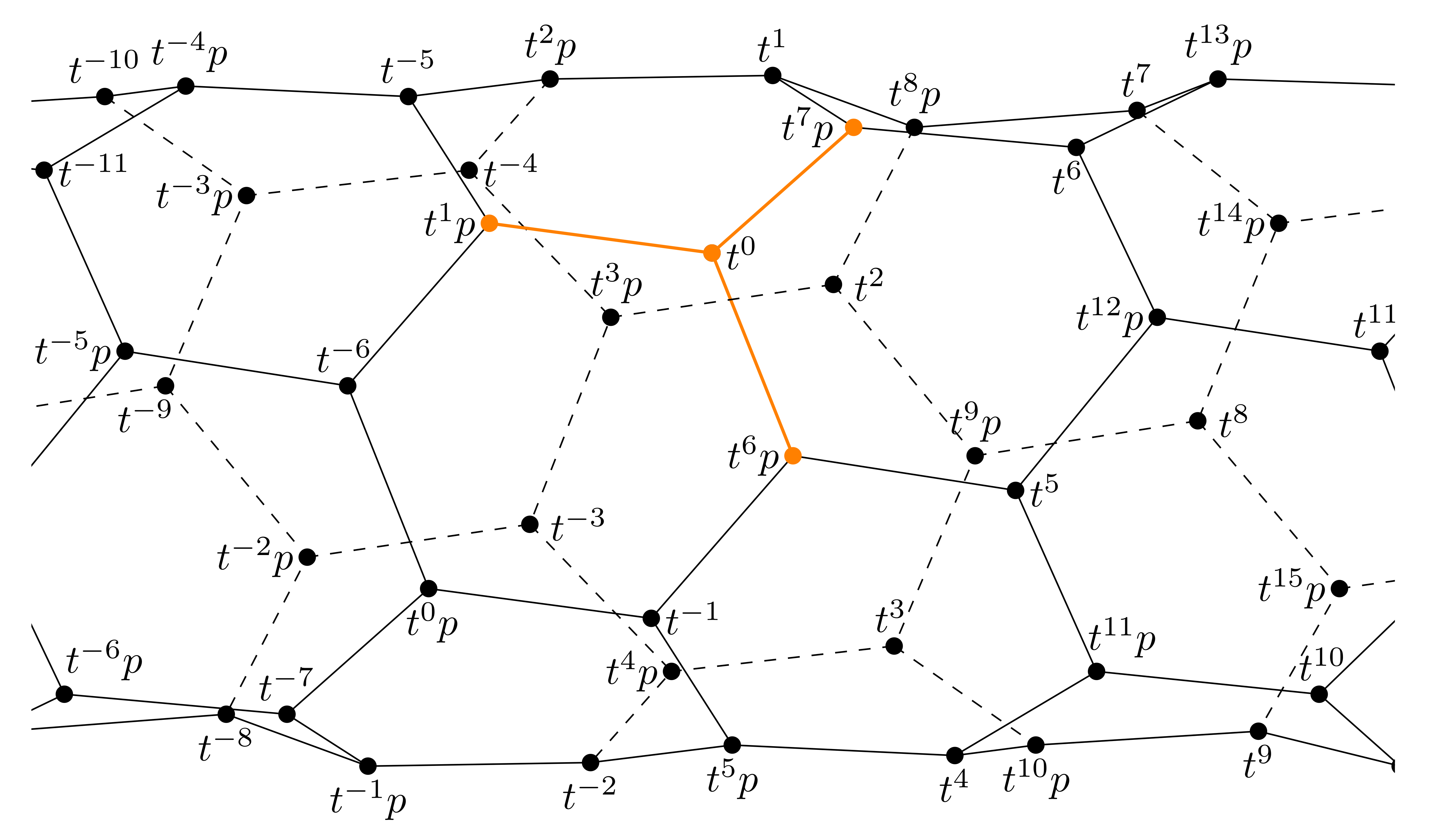}
    \caption{As described in Example~\ref{Example:Nanotube}, the orbit of the point $x_{a_0}$ under the action of the group $\G_{a_0,\alpha_0}$ is a $(5,1)$ nanotube. We have a natural bijection between the group elements and the atoms. $x_{a_0}$ and its nearest neighbor bonds to atoms in $\NN\cdot x_{a_0}$ are highlighted.}\label{fig:(5,1)_Nanotube}
\end{figure}
For all scale factors $a>0$ and angles $\alpha\in(0,\pi)$ we define:
Let $R(\alpha)\in\O(2)$ be the rotation matrix $R(\alpha):=\parens[\bigg]{\begin{matrix}\mbox{\footnotesize $\cos(\alpha)$}&\mbox{\footnotesize $-\sin(\alpha)$}\\ \mbox{\footnotesize $\sin(\alpha)$}&\mbox{\footnotesize $\cos(\alpha)$}\end{matrix}}$, $t=t_{a, \alpha}=\iso{R(\alpha)\oplus I_1}{ae_3}\in\E(3)$, $p=\iso{I_1\oplus(-I_2)}{0_3}\in\E(3)$ and $\G=\G_{a, \alpha}$ be the discrete group $\angles{t, p}<\E(3)$, \ie $\G=\set{t^mp^q}{m\in\Z, q\in\{0,1\}}$.
For all $x\in\R^3$ we have $\G\cdot x\subset C_x$, where $C_x$ is the cylinder $\set{y\in\R^3}{y_1^2+y_2^2=x_1^2+x_2^2}$.

Let $\NN=\NN_{a,\alpha}=\{tp, t^6p, t^7p\}$.
Let $U_{a, \alpha}\subset\R^3$ be the set of all points $x\in\R^3$ for which the stabilizer subgroup $\G_x$ is trivial and the three nearest neighbors of $x$ in $\G\cdot x$ are the points $\NN\cdot x$, \ie
\[ \sup\set[\Big]{\norm{g\cdot x - x}}{g\in\NN}<\inf\set[\Big]{\norm{g\cdot x - x}}{g\in\G\setminus(\NN\cup\{\id\})}.\]
Let
\[W:=\set[\big]{(a,\alpha,x)}{a>0,\alpha\in(0,\pi),x\in U_{a,\alpha}}.\]

Analogously to \cite{Friedrich2019} we define the interaction potential $V = V_{a, \alpha}$, see Definition~\ref{Definition:potential} and Remark~\ref{Remark:VFrechet}\ref{item:RemarkVDomain}-\ref{item:RemarkGTrivial}, by
\[V(y)=\frac12\sum_{g\in\NN}v_1(\norm{y(g)}) + \frac12\sum_{g,h\in\NN}v_2(y(g), y(h)),\]
where
\begin{align*}
&v_1\colon(0,\infty)\to\R,\quad r \mapsto (r-1)^2
\intertext{is a \emph{two-body potential} and}
&v_2\colon\set{(x,y)}{x,y\in\R^3\setminus\{0\}}\to\R,\quad (x, y)\mapsto\parens[\bigg]{\frac{\scalar xy}{\norm x\norm y}+\frac12}^2
\end{align*}
is a \emph{three-body potential}.
Thus the \emph{bonded} points of $\G\cdot x$ tend to have distance 1 and the \emph{bond angles} tend to form $2\pi/3$ angles.
By Lemma~\ref{Lemma:FrechetE} for all $(a,\alpha,x)\in W$ we have $E(\chi_\G x)=V(y_0)$, where $E=E_{a,\alpha}$ is the configurational energy and $y_0=y_{0,a,\alpha,x}=(g\cdot x-x)_{g\in\G_{a,\alpha}}$.

First we consider the $(5,1)$ nanotube.
We define
\begin{align*}
&\alpha_0:=11\pi/31\approx 1.115
\shortintertext{and}
&x_a:=a(r\cos(\beta), r\sin(\beta), 7/3)\in\R^3&&\text{for all }a>0,
\end{align*}
where $r=31/(\pi\sqrt3)$ and $\beta=5\pi/31$.
In the strict sense of \cite{Dresselhaus1995}, for all $(a, \alpha, x)\in W$ the set $\G\cdot x$ is a so-called $(5,1)$ nanotube if and only if $\alpha=\alpha_0$ and $x=x_a$.
The bond length of the unrolled $(5,1)$ nanotube $\G_{a,\alpha_0}\cdot x_a$, \ie the distance of two neighboring points of $\G_{a,\alpha_0}\cdot x_a$ with respect to the induced metric of the manifold $C_x$, is equal to 1 if and only if $a=a_0$, where
\[a_0:=3/(2\sqrt{31})\approx 0.269.\]
Now we investigate numerically with Algorithm~\ref{Algorithm:algorithm} the stability of the $(5,1)$ nanotube, more precisely of $(\G_{a,\alpha_0},x_a,V_{a,\alpha_0})$.
\begin{enumerate}
\item[\ref{item:A1}] For all $a>0$ we have $\eea\neq0$, see Figure~\ref{fig:Nanotube_1}, and thus $\chi_{\G_{a, \alpha_0}}x_a$ is not a critical point of $E_{a,\alpha_0}$.
Thus we can proceed with \ref{item:A6}.
\begin{figure}
    \centering
        \includegraphics[width=0.7\textwidth]{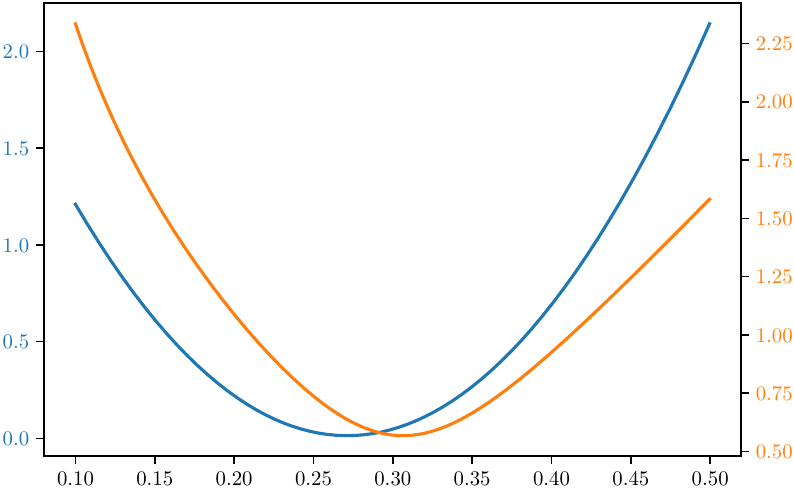}
    \caption{For the $(5,1)$ nanotube as described in Example~\ref{Example:Nanotube}, the graphs of the energy $E(\chi_{\G_{a,\alpha_0}}x_a)$ and the norm of $e_{V_{a,\alpha_0}}$ dependent on the scale factor $a$ are plotted in blue and orange, respectively. For all $a$, we have $e_{V_{a,\alpha_0}}\neq0$ and thus the $(5,1)$ nanotube is not stable and in fact not even a critical point.}\label{fig:Nanotube_1}
\end{figure}
\item[\ref{item:A6}] By \ref{item:A1} for all $a>0$ the triple $(\G_{a,\alpha_0}, x_a, V_{a,\alpha_0})$ is not stable with respect to both $\norm\fdot_\RR$ and $\newnorm\fdot\RR$.
\end{enumerate}
This failure suggests that we should relax our model to allow for nanotubes that correspond to critical points of the energy and still have the same neighborhood structure. 
We define
\begin{align*}
&(a^*,\alpha^*,x^*):=\argmin_{(a,\alpha,x)\in W}E_{a,\alpha}(\chi_{\G_{a,\alpha}} x)\approx(0.263, 1.117, (1.388, 0.776, 0.626))
\shortintertext{and}
&x_a^*:=\argmin_{x\in U_{a,\alpha^*}}E(\chi_\G x)\qquad\text{for all }a\approx a^*.
\end{align*}
In particular we have $x^*=x_{a^*}^*$.
We have $(a^*,\alpha^*,x^*)\approx(a_0,\alpha_0,x_{a_0})$ and thus the nanotube $\G_{a^*,\alpha^*}\cdot x^*$ is approximately equal to the $(5,1)$ nanotube $\G_{a_0,\alpha_0}\cdot x_{a_0}$.
Now for all $a\approx a^*$ we check the stability of $(\G_{a,\alpha^*},x_a^*,V_{a,\alpha^*})$ numerically with Algorithm~\ref{Algorithm:algorithm}.
\begin{enumerate}
\item[\ref{item:A1}] For all $a\approx a^*$ the function $\chi_\G x_a^*$ is a critical point of $E$ since Lemma~\ref{Lemma:energyderivativeY} implies $\ee=\parens[\big]{E'(\chi_\G x_a^*)(\chi_\G e_i))}_{i\in\{1,\dots,d\}}=0$ and hence that $\chi_\G x_a^*$ is critical. 
\item[\ref{item:A2}] We have
\begin{align*}
&\supp V =\{tp, t^6p, t^7p\}
\shortintertext{and}
&\supp\ff=\{t^{-6}, t^{-5}, t^{-1}, \id, t, t^5, t^6, tp, t^6p, t^7p\}
\end{align*}
by Remark~\ref{Remark:newee}\ref{item:newee} and the relations $(tp)^{-1}=pt^{-1}=tp$.
The first and second derivative of $V$ can be computed, \eg, with the Python library SymPy and $\ff$ can be computed numerically by Definition~\ref{Definition:eeff}.
\item[\ref{item:A3}] Since $\aff(\{t^{-1}, \id, t, p\}\cdot x_0)=\R^3=\aff(\G\cdot x_0)$ and $\{t, p\}$ generates $\G$, by Definition~\ref{Definition:Property} the set
\[\RR=\RR_a:=\{t^{-1},\id, t, t^2, t^{-1}p, p, tp\}\]
is an admissible $\id$-neighborhood. 
We define the bijection $\phi$ between $\RR$ and $\{0,\dots,6\}$ by $\phi(t^m)=m+1$ for all $m\in\{-1,0,1,2\}$ and $\phi(t^mp)=m+5$ for all $m\in\{-1,0,1\}$.
For all $a\approx a^*$ we define the six functions 
\begin{align*}
b_i=b_{i,a}\colon\RR\to\R^3,\qquad g\mapsto
\begin{cases}
\rot(g)^{\mathsf T}e_i&\text{for }i\in\{1, 2, 3\},\\
\rot(g)^{\mathsf T}A_i(g\cdot x_a^*-x_a^*)&\text{for }i\in\{4, 5, 6\},
\end{cases}
\end{align*}
where
\[A_4=\parens[\bigg]{\begin{smallmatrix}0&-1&0\\1&0&0\\0&0&0\end{smallmatrix}}, A_5=\parens[\bigg]{\begin{smallmatrix}0&0&-1\\0&0&0\\1&0&0\end{smallmatrix}}\text{ and } A_6=\parens[\bigg]{\begin{smallmatrix}0&0&0\\0&0&-1\\0&1&0\end{smallmatrix}}.\]
By Proposition~\ref{Proposition:UtransUrotbig} the sets $\{b_1,\dots,b_6\}$ and $\{b_1,\dots,b_4\}$ are bases of $\Uiso\RR$ and $\Unewiso\RR$, respectively.
With, \eg, the Gram-Schmidt process we can determine functions $b_1',\dots,b_6'\colon\RR\to\R^3$ such that $\{b_1',\dots,b_6'\}$ and $\{b_1',\dots,b_4'\}$ are orthonormal bases of $\Uiso\RR$ and $\Unewiso\RR$, respectively.
A bijection between $\{u\colon\RR\to\C^3\}$ and $\C^{21}$ is given by $u\mapsto(u(\phi^{-1}(0)),\dots,u(\phi^{-1}(6)))$.
Let $B=(b_1',\dots,b_6')\in\R^{21\times6}$ and $B_0=(b_1',\dots,b_4')\in\R^{21\times4}$.
The matrices $P=I_{21}-BB^{\mathsf T}$ and $P_0=I_{21}-B_0B_0^{\mathsf T}$ are orthogonal projection matrices with kernels $\Uiso\RR$ and $\Unewiso\RR$, respectively.
Let $p_0,\dots,p_6,p_{0,0},\dots,p_{0,6}\in\R^{21\times3}$ such that $P=(p_0,\dots,p_6)$ and $P_0=(p_{0,0},\dots,p_{0,6})$.
For the functions $\gggg$ and $\newgggg$ of Definition~\ref{Definition:gggg} we have
\begin{align*}
&\supp\gggg=\supp\newgggg=\RR\text{ and}\\
&\gggg(g)=p_{\phi(g)},~
\newgggg(g)=p_{0,\phi(g)}\quad\text{for all }g\in\RR.
\end{align*}
\item[\ref{item:A4}] We have $\T\F=\T=\angles t$, $\MM=\N$ and $\{\id\}$ is a representation set of $\dual{\T\F}/{\simrg}$ by Lemma~\ref{Lemma:RepSys}\ref{item:aaa}.
We have $\LL=\angles a$ and $\dualL=\angles{a^{-1}}$, see Definition~\ref{Definition:ChiK-L}\ref{Definition:LLst}.
By Proposition~\ref{Proposition:SpaceGroupLattice} we have $\set{k\in\R}{\iso{I_1}k\in\G_\id}=\angles{a^{-1}}$ and thus $\G_\id=\set{\iso{(-I_1)^q}{ma^{-1}}}{m\in\Z,q\in\{0,1\}}$.
The interval $K_\id=[0,1/(2a))$ is a representation set of $\R/\G_\id$.
\item[\ref{item:A5}] The set $\{\id,p\}$ is a complete set of representatives of the cosets of $\T\F$ in $\G$.
For all $k\in K_\id$ and $g\in\G$ we have
\[\Ind_{\T\F}^\G \chi_k(g)=\begin{cases}
\parens[\bigg]{\begin{matrix}\chi_k(g)&0\\0&\chi_k(p^{-1}gp)\end{matrix}} & \text{if }g\in\T\F\\
\parens[\bigg]{\begin{matrix}0&\chi_k(gp)\\ \chi_k(p^{-1}g)&0\end{matrix}}& \text{else.}
\end{cases}\]
Now for all $k\in K_\id$, it is easy to compute the complex $6\times 6$ matrices $\fourier\ff(\Ind\chi_k)$, $\fourier\gggg(\Ind\chi_k)$ and $\fourier\newgggg(\Ind\chi_k)$.
For both $\gggg$ and $\newgggg$ we have
\begin{align*}
\set{k\in K_\id}{\fourier{g_{\RR(,0,0)}}(\Ind\chi_k)\text{ has full rank}}=K_\id\setminus\{0,\alpha^*/(2\pi a)\}.
\end{align*}
For all $k\in K_\id\setminus\{0,\alpha^*/(2\pi a)\}$ we can compute $\lambdamin\parens{\fourier\ff(\Ind\chi_k),\fourier\gggg(\Ind\chi_k)}$ and $\lambdamin\parens{\fourier\ff(\Ind\chi_k),\allowbreak\fourier\newgggg(\Ind\chi_k)}$.
In particular we can compute both $\lambdaa(a,\alpha^*)$ and $\lambdanewa(a,\alpha^*)$ numerically, see Figure~\ref{fig:Nanotube_2}.
\item[\ref{item:A6}] In the stretched case $a>a^*$, we have $\lambdaa(a,\alpha^*)>0$ and $\lambdanewa(a,\alpha^*)>0$ and thus $(\G_{a,\alpha^*},x_{a,\alpha^*},V_{a,\alpha^*})$ is stable with respect to both $\norm\fdot_\RR$ and $\newnorm\fdot\RR$.
In the compressed case $a\in(0,a^*)$ we have $\lambdaa(a,\alpha^*)=-\infty$ and $\lambdanewa(a,\alpha^*)<0$ and thus $(\G_{a,\alpha^*},x_{a,\alpha^*},V_{a,\alpha^*})$ is not stable with respect to both $\norm\fdot_\RR$ and $\newnorm\fdot\RR$.
\begin{figure}
    \centering
    \begin{subfigure}{0.48\textwidth}
        \includegraphics[width=\textwidth]{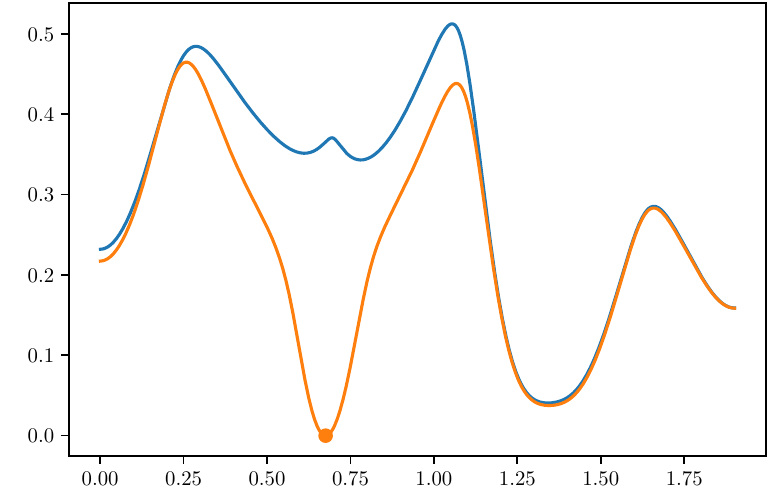}
    \end{subfigure}
    ~ 
    \begin{subfigure}{0.48\textwidth}
        \includegraphics[width=\textwidth]{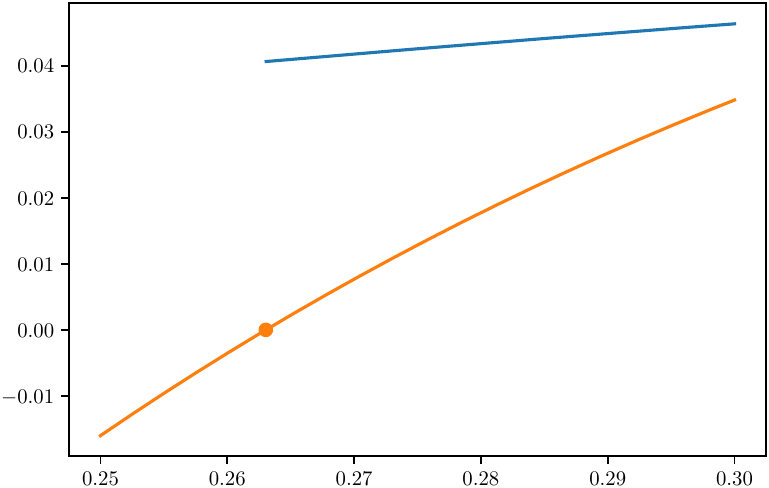}
    \end{subfigure}
    \caption{For the nanotube as described in Example~\ref{Example:Nanotube}, the point $(\alpha^*/(2\pi a^*),0)$ and the graphs of $\lambdamin\parens{\widehat{\ff}(\chi_k),\widehat{\gggg}(\chi_k)}$ (blue) and $\lambdamin\parens{\widehat{\ff}(\chi_k),\widehat{\newgggg}(\chi_k)}$ (orange) dependent on $k\in K_\id\setminus\{0,\alpha^*/(2\pi a^*)\}$ are plotted on the left for the choice $a=a^*$. The point $(a^*,0)$ and and the graphs of $\lambdaa$ (blue) and $\lambdanewa$ (orange) dependent on the scale factor are plotted on the right.}\label{fig:Nanotube_2}
\end{figure}
\end{enumerate}
Notice that in the stretched case $a>a^*$, the appropriate seminorm for the stability is $\newnorm\fdot\RR$.
For the equilibrium case $a\approx a^*$, the weaker seminorm $\norm\fdot_\RR$ is appropriate since $\lim_{a\searrow a^*}\lambdanewa(a,\alpha^*)=0$ and $\lim_{a\searrow a^*}\lambdaa(a,\alpha^*)>0$.

For all $a\approx a^*$ and $\alpha\approx\alpha^*$ we can compute $\lambdaa(a,\alpha)$ and $\lambdanewa(a,\alpha)$ analogously.
For $\alpha\approx\alpha^*$ the graphs of $\lambdaa(\fdot,\alpha)$ and $\lambdanewa(\fdot, \alpha)$ are similar to the graphs of $\lambdaa(\fdot, \alpha^*)$ and $\lambdanewa(\fdot,\alpha^*)$.
As an example, we consider 
\[\alpha_a:=\argmin_{\alpha\in(0,\pi)}E(\chi_\G x_{a,\alpha})\qquad\text{for all }a\approx a^*,\]
see Figure~\ref{fig:Nanotube_3}.
In that figure, the graphs of 
\begin{align*}
&a\mapsto\reldiff\parens[\big]{\lambdaa(a,\alpha^*),\lambdaa(a,\alpha_a)}
\shortintertext{and}
&a\mapsto\reldiff\parens[\big]{\lambdanewa(a,\alpha^*),\lambdanewa(a,\alpha_a)}
\end{align*}
are plotted, where $\reldiff(x,y):=\abs{x-y}/\max\{\abs x,\abs y\}$ for all $x,y\in\R.$
\begin{figure}
    \centering
    \begin{subfigure}{0.48\textwidth}
        \includegraphics[width=\textwidth]{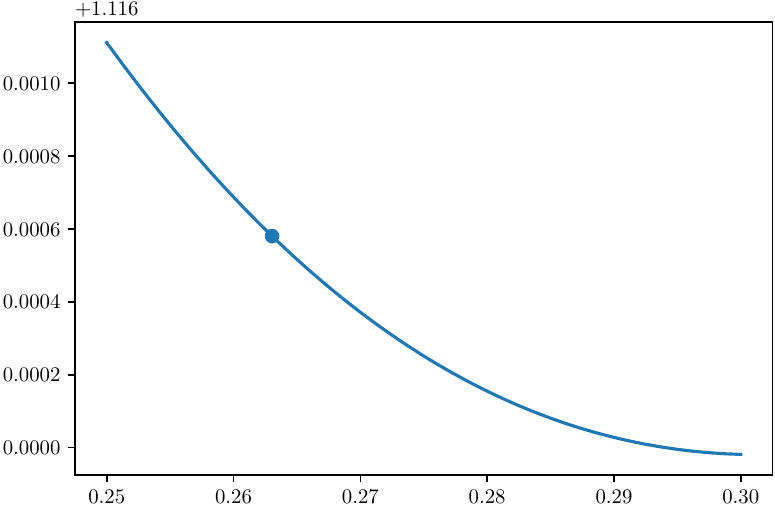}
    \end{subfigure}
    ~ 
    \begin{subfigure}{0.48\textwidth}
        \includegraphics[width=\textwidth]{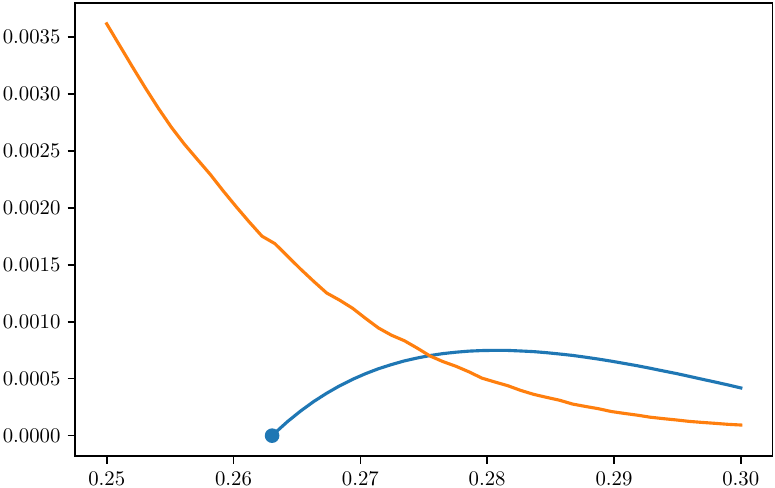}
    \end{subfigure}
    \caption{For the nanotube as described in Example~\ref{Example:Nanotube}, the point $(a^*,\alpha(a^*))$ and the graph of the angle $\alpha(a)$ dependent on the scale factor $a$ are plotted on the left. The point $(a^*,0)$ and the graphs of $\reldiff\parens[\big]{\lambdaa(a,\alpha^*),\allowbreak\lambdaa(a,\alpha_a)}$ (blue) and $\reldiff\parens[\big]{\lambdanewa(a,\alpha^*),\allowbreak\lambdanewa(a,\alpha_a)}$ (orange) dependent on the scale factor $a$ are plotted on the right.}\label{fig:Nanotube_3}
\end{figure}
\end{Example}
%
%%%%%%%%%%%%%%%%%%%%%%%%
%
\section{Proofs}\label{section:proofs}
%
%%%%%%%%%%%%%%%%%%%%%%%%
%
\subsection{Derivatives of the energy}\label{subsec:prelim}
We prove Lemma~\ref{Lemma:FrechetE}, Proposition~\ref{Proposition:LambdaDual} and Lemma~\ref{Lemma:energyderivativeY}.  
\begin{proof}[Proof of Lemma~\ref{Lemma:FrechetE}]
By \ref{item:Rotation} we have
\begin{equation}\label{eq:beach}
E(u)=\frac1{\abs{\CC_N}}\sum_{g\in\CC_N}V\parens[\Big]{\parens[\Big]{\frac1{\abs{\G_{x_0}}}\sum_{h'\in h}h'\gdot u(gh')-\frac1{\abs{\G_{x_0}}}\sum_{h'\in\G_{x_0}}h'\gdot u(gh')}_{h\in\Gstar}}
\end{equation}
for all $u\in\UPer$ and $N\in \MM$ such that $u$ is $\T^N$-periodic.
By \ref{item:Vfrechet} the function $V$ is two times Fr\'echet differentiable.
We define the vector space
\[W=\set[\Big]{w\colon\G\to L^\infty(\Gstar,\R^d)}{w\text{ is periodic}}\]
and equip $\UPer$ and $W$ each with the uniform norm $\norm\fdot_\infty$.
The linear map
\begin{align*}
\phi_1\colon\UPer\to W,\quad u\mapsto \parens[\Big]{\G\to L^\infty(\Gstar,\R^d),g\mapsto
\parens[\big]{\rot(h)\proj{u}(gh)-\proj{u}(g)}_{h\in\Gstar}}
\end{align*}
is bounded and thus two times continuously Fr\'echet differentiable.
Let us denote $\tau_{\G_{x_0}}:=\frac1{\abs{\G_{x_0}}}\sum_{h\in \G_{x_0}}\trans(h)$.
The first and second derivative of the function 
\begin{align*}
\phi_2\colon&W\to\R,\quad 
w\mapsto\frac1{\abs{\CC_N}}\sum_{g\in\CC_N}V\parens[\big]{\parens{h\cdot\tau_{\G_{x_0}}-\tau_{\G_{x_0}}}_{h\in\Gstar}+w(g)}
\end{align*}
if $w$ is $\T^N$-periodic, is given by
\begin{align*}
&\phi_2'(w)w_1 = \frac1{\abs{\CC_N}}\sum_{g\in\CC_N}V'\parens[\big]{\parens{h\cdot\tau_{\G_{x_0}}-\tau_{\G_{x_0}}}_{h\in\Gstar}+w(g)}w_1(g)
\shortintertext{and}
&\phi_2''(w)(w_1,w_2)
= \frac1{\abs{\CC_N}}\sum_{g\in\CC_N}V''\parens[\big]{\parens{h\cdot\tau_{\G_{x_0}}-\tau_{\G_{x_0}}}_{h\in\Gstar}+w(g)}(w_1(g),w_2(g))
\end{align*}
for all $w,w_1,w_2\in W$ and $N\in\MM$ such that $w$, $w_1$ and $w_2$ are $\T^N$-periodic.
Thus $\phi_2$ is two times continuously Fr\'echet differentiable.
Since $E=\phi_2\circ\phi_1$, also the function $E$ is two times continuously Fr\'echet differentiable.

Equation~\eqref{eq:beach} also implies the representations of $E(\xx)$, $E'(\xx)$ and $E''(\xx)$.
\end{proof}

\begin{proof}[Proof of Proposition~\ref{Proposition:LambdaDual}.]
We only consider $\lambdaa$. 
The characterization of $\lambdanewa$ is analogous. 
Let us denote $\text{RHS}=\inf\set{E''(\xx)(u,u)}{u\in\UPer,\norm u_\RR=1}$.
It is clear that $\lambdaa\le\text{RHS}$.
Let $c\in\R$ be such that $c>\lambdaa$.
There exists some $u\in\UPer$ such that $c\norm u_\RR^2> E''(\xx)(u,u)$.
By Theorem~\ref{Theorem:StrongerEquivalence}\ref{Theorem:EquivalenceAll}, Proposition~\ref{Proposition:UtransUrotbig} and since $\G\cdot x_0$ is not equal to $\{x_0\}$, we have $\ker(\norm\fdot_\RR)\neq\UPer$.
Thus and since $\norm\fdot_\RR\le\sqrt{\abs\RR}\norm\fdot_\infty$, we may assume that $\norm u_\RR=1$.
So we have $\text{RHS}\le c$.
Since $c$ was arbitrary, we have $\lambdaa\ge\text{RHS}$.
\end{proof}

\begin{proof}[Proof of Lemma~\ref{Lemma:energyderivativeY}.]
Let $N\in\MM$ and $g_1,g_2\in\G$. Since $\T^N$ is a normal subgroup of $\G$, we have
\begin{equation}\label{eq:rxot}
\sum_{g\in\CC_N}\indi_{g_1\T^N}(gh)=\sum_{g\in\CC_N}\sum_{t\in\T^N}\delta_{g_1h^{-1},gt}=1\qquad\text{for all }h\in\G.
\end{equation}
Using Lemma~\ref{Lemma:FrechetE}, Remark~\ref{Remark:VFrechet}\ref{item:RemarkVFrechet} and \eqref{eq:rxot}, we have
\begin{align*}
&\partial_{g_1\T^N}E(\xx)=\parens[\big]{E'(\xx)(\indi_{g_1\T^N}e_i)}_{i\in\{1,\dots,d\}}\nonumber\\
&\quad=\frac1{\abs{\CC_N}\abs{\G_{x_0}}}\sum_{g\in\CC_N}\sum_{h\in \Gstar}\partial_hV(y_0)\parens[\Big]{\sum_{h'\in h}\indi_{g_1\T^N}(gh')\rot(h')-\sum_{h'\in\G_{x_0}}\indi_{g_1\T^N}(gh')\rot(h')}\nonumber\\
&\quad=\frac1{\abs{\CC_N}}\sum_{h\in\Gstar}\partial_hV(y_0)\parens{L_h-L_{\G_{x_0}}}
=\frac1{\abs{\CC_N}}\ee.
\end{align*}
Since $\T^N$ is a normal subgroup of $\G$, for all $h_1,h_2\in\G$ we have
\begin{align}
\sum_{g\in\CC_N}\indi_{g_2\T^N}(gh_2)\indi_{g_1\T^N}(gh_1)&=\sum_{g\in\CC_N}\sum_{t,s\in\T^N}\delta_{g_2h_2^{-1},gs}\delta_{g_1h_1^{-1}t,gs}\nonumber\\
&=\sum_{t\in\T^N}\delta_{g_2h_2^{-1},g_1h_1^{-1}t}
=\sum_{t\in\T^N}\delta_{h_2^{-1}h_1,g_2^{-1}g_1t}.\label{eq:starstarstar}
\end{align}
Using Lemma~\ref{Lemma:FrechetE}, Remark~\ref{Remark:VFrechet}\ref{item:RemarkVFrechet} and \eqref{eq:starstarstar}, we have
\begin{align}
&\partial_{g_2\T^N}\partial_{g_1\T^N}E(\xx)=\parens[\big]{E''(\xx)(\indi_{g_2\T^N}e_i,\indi_{g_1\T^N}e_j)}_{i,j\in\{1,\dots,d\}}\nonumber\\
\begin{split}\nonumber
&\quad=\frac1{\abs{\CC_N}\abs{\G_{x_0}}^2}\sum_{g\in\CC_N}\sum_{h_1,h_2\in\Gstar}\sum_{h_1',h_2'\in\G}\parens{\chi_{h_2}(h_2')-\chi_{\G_{x_0}}(h_2')}\indi_{g_2\T^N}(gh_2')\rot(h_2')^{\mathsf T}\\
&\quad\quad\quad\quad\quad\quad\quad\quad\quad\quad\quad\quad\quad\quad\partial_{h_2}\partial_{h_1}V(y_0)\parens{\chi_{h_1}(h_1')-\chi_{\G_{x_0}}(h_1')}\indi_{g_1\T^N}(gh_1')\rot(h_1')\end{split}\\
\begin{split}\nonumber
&\quad=\frac1{\abs{\CC_N}\abs{\G_{x_0}}^2}\sum_{t\in\T^N}\sum_{h_1,h_2\in\Gstar}\sum_{h_1',h_2'\in\G}\delta_{h_2'^{-1}h_1',g_2^{-1}g_1t}\rot(h_2')^{\mathsf T}\partial_{h_2}\partial_{h_1}V(y_0)\rot(h_1')\\
&\quad\quad\quad\quad\quad\quad\quad\quad\quad\quad\quad\quad\quad\quad\parens{\chi_{h_2}(h_2')-\chi_{\G_{x_0}}(h_2')}\indi_{g_2\T^N}(gh_2')\rot(h_2')^{\mathsf T}\end{split}\\
\nonumber
&\quad=\frac1{\abs{\CC_N}}\sum_{t\in\T^N}\ff(g_2^{-1}g_1t).\qedhere
\end{align}
\end{proof}

\subsection{A sufficient condition for a local minimum}\label{Section:Minimum}
We prove Theorem~\ref{Theorem:local-min}. 
In the following $\RR\subset\G$ is an admissible neighborhood range of $\id$.  
We first notice that the space $\Uiso\RR$ of infinitesimally rigid displacements of $\RR$ is indeed the tangent space at the identity mapping to the space of nonlinearly rigid deformations.
\begin{Proposition}\label{Proposition:Manifold}
There exists an open neighborhood $U\subset\E(d)$ of $\id$ such that the set
\[\set[\Big]{u\colon\RR\to\R^d}{\exists\+a\in U\+\forall g\in\RR: g\cdot(x_0+\proj{u}(g))=a\cdot(g\cdot x_0)}\]
is a manifold and $\Uiso\RR$ is its tangent space at the point $0$.
\end{Proposition}
\begin{proof}
For a representation set $\mathcal D$ of $\Rcoset$ we define the embedding $e\colon(\R^d)^{\RR\setminus\mathcal D}\to(\R^d)^\RR$ by
\begin{align*}
(e(x))_g=
\begin{cases}
-x_g&\text{if }g\in\RR\setminus\mathcal D,\\
\rot(g)^{\mathsf T}\sum_{g'\in g(\G_{x_0}\setminus\{\id\})}\rot(g')x_{g'}&\text{if }g\in\mathcal D.
\end{cases}
\end{align*}
Let $B=\set{S\in\Skew(d)}{\norm S<c}$ with $c>0$ so small that the matrix exponential $\exp\colon B\to\exp(B)$ is a diffeomorphism onto a neighborhood $\exp(B)$ of $I_d$ in $\SO(d)$.
Let $\log$ be its inverse map.
Let $U\subset\Skew((d-\daff)+\daff)$ be a neighborhood of $0$ such that the map
\begin{align*}
f\colon&U\to\Skew(d),\quad 
\parens[\Big]{\begin{smallmatrix}S_1&A\\-A^{\mathsf T}&S_2\end{smallmatrix}}\mapsto\log\parens[\Big]{\exp\parens[\Big]{\begin{smallmatrix}0&A\\-A^{\mathsf T}&S_2\end{smallmatrix}}\exp\parens[\Big]{\begin{smallmatrix}S_1&0\\0&0\end{smallmatrix}}}
\end{align*}
is well-defined.
By the inverse function theorem there exists an open neighborhood $V\subset U$ of 0 such that $W:=f(V)$ is an open neighborhood of $0$ and the map $f|_V\colon V\to W$ is a diffeomorphism.
Without loss of generality we may assume that
\[V=\set[\Big]{\parens[\Big]{\begin{smallmatrix}S_1&A\\-A^{\mathsf T}&S_2\end{smallmatrix}}}{S_1\in V_1,(A,S_2)\in V_2},\]
where $V_1\subset\Skew(d-\daff)$ is an open neighborhood of $0$ and $V_2\subset\R^{(d-\daff)\times\daff}\times\Skew(\daff)$ is an open neighborhood of $0$.
The set $X:=\set{\iso{\exp(T)}b}{T\in W,b\in\R^d}\subset\E(d)$ is an open neighborhood of $\id$.
Setting $Y=\R^d\times V_2\times(\R^d)^{\RR\setminus\mathcal D}$ we have
\begin{align*}
M:=&\set[\Big]{u\colon\RR\to\R^d}{\exists\+a\in X\+\forall g\in\RR : g\cdot x_0+\rot(g)\proj{u}(g)=a\cdot(g\cdot x_0)}\\
=&\set[\big]{(\rot(g)^{\mathsf T}(b+(\exp(T)-I_d)(g\cdot x_0)))_{g\in\RR}+e(x)}{b\in\R^d,T\in W,x\in(\R^d)^{\RR\setminus\mathcal D}}\\
=&\set[\Big]{\parens[\Big]{\rot(g)^{\mathsf T}\parens[\Big]{b+\parens[\Big]{\exp\parens[\Big]{\begin{smallmatrix}0&A\\-A^{\mathsf T}&S\end{smallmatrix}}-I_d}(g\cdot x_0-x_0)}}_{g\in\RR}+e(x)}{(b,(A,S),x)\in Y}
\end{align*}
since $g\cdot x_0-x_0\in\{0_{d-\daff}\}\times\R^\daff$ for all $g\in\RR$.
Thus the map $h\colon Y\to M$ given by 
\begin{align*}
h(b,A,S,x)=\parens[\Big]{\rot(g)^{\mathsf T}\parens[\Big]{b+\parens[\Big]{\exp\parens[\Big]{\begin{smallmatrix}0&A\\-A^{\mathsf T}&S\end{smallmatrix}}-I_d}(g\cdot x_0-x_0)}}_{g\in\RR}+e(x)
\end{align*}
is surjective.
Since $\RR$ is admissible and thus $\aff(\RR\cdot x_0)=\aff(\G\cdot x_0)$, there exists some $C=(c_g)_{g\in\RR}\in\R^{\daff\times\abs\RR}$ of rank $\daff$ such that $(g\cdot x_0-x_0)_{g\in\RR}=(\begin{smallmatrix}0\\C\end{smallmatrix})$.
We have
\begin{align*}
h'(0)\colon Y\to(\R^d)^\RR,\qquad 
(b,A,S,x)\mapsto\parens[\Big]{\rot(g)^{\mathsf T}\parens[\Big]{b+\parens[\Big]{\begin{smallmatrix}Ac_g\\Sc_g\end{smallmatrix}}}}_{g\in\RR}+e(x).
\end{align*}
Since $\id\in\RR$ and the rank of $C$ is equal to the number of its rows, the map $h'(0)$ is injective.
Thus there exist an open neighborhood $Y_0\subset Y$ of $0$ and an open neighborhood $M_0\subset M$ of $0$ such that $h|_{Y_0}\colon Y_0\to M_0$ is a diffeomorphism.
In particular $M$ is a manifold and $\Uiso\RR$ is its tangent space at $0$.
\end{proof}
\begin{Remark}
A chart of the manifold of the above theorem is given in the proof.
\end{Remark}

We can now prove Theorem~\ref{Theorem:local-min}. 
\begin{proof}[Proof of Theorem~\ref{Theorem:local-min}]
First we assume that $d_1\in\{0,1\}$.
Let $\RRVset/\G_{x_0}\subset\Gstar$ be a finite interaction range of $V$.
Since $\ee=0$, by Lemma~\ref{Lemma:energyderivativeY} we have $E'(\xx)=0$.
By Theorem~\ref{Theorem:StrongerEquivalence}\ref{Theorem:NewEquivalenceAll} there exists a constant $c_1$ such that $\newnorm\fdot\RR\ge c_1\newnablanorm\fdot{\RR\cup\RRVset}$.
Let $c_2=c_1^2\lambdanewa/2>0$.
We have
\begin{align}
E''(\xx)(u,u)&\ge\lambdanewa\newnorm u\RR^2
\ge\tfrac\lambdanewa2\newnorm u\RR^2+ c_2\newnablanorm u{\RR\cup\RRVset}^2\nonumber\\
&\ge\tfrac\lambdanewa2\newnorm u\RR^2+c_2\newnablanorm u{\RRVset}^2
=\tfrac\lambdanewa2\newnorm u\RR^2+c_2\norm{\nabla_{\RRVset}u}_2^2\label{eq:econe}
\end{align}
for all $u\in\UPer$.
In the last step we used that $\newnablanorm\fdot{\RRVset}=\norm{\nabla_{\RRVset}\fdot}_2$ since $d_1\in\{0,1\}$ implies $\Unewrot{\RRVset}=\set{u\colon\RRVset\to\R^d}{\forall g\in\RRVset\colon\rot(g)\proj{u}(g)=0}$.
Since $\RRVset/\G_{x_0}$ is a finite interaction range of $V$, by Taylor's theorem there exists some $\epsilon>0$ such that for all $u\colon\Gstar\to\R^d$ with $\norm u_\infty<\epsilon$ we have
\begin{equation}\label{eq:ectwo}
V(y_0+u)\ge V(y_0)+V'(y_0)u+\tfrac12V''(y_0)(u,u)-c_2\norm{u|_{\RRVset}}^2.
\end{equation}
For all $u\in\UPer$ with $\norm u_\infty<\epsilon/2$ we have (with $w=\xx+u$)
\begin{align*}
&E(\xx+u)=\frac1{\abs{\CC_N}}\sum_{g\in\CC_N}V\parens[\Big]{\parens[\Big]{\frac1{\abs{\G_{x_0}}}\sum_{h'\in h}h'\cdot w(gh')-\frac1{\abs{\G_{x_0}}}\sum_{h'\in\G_{x_0}}h'\cdot w(gh')}_{h\in\Gstar}}\\
\begin{split}
&\quad\ge\frac1{\abs{\CC_N}}\sum_{g\in\CC_N}\parens[\Big]{V(y_0)+V'(y_0)\parens[\big]{\rot(h)\proj{u}(gh)-\proj{u}(g)}_{h\in\Gstar}\\
&\quad\quad+\frac12 V''(y_0)\parens[\Big]{\parens[\big]{\rot(h)\proj{u}(gh)-\proj{u}(g)}_{h\in\Gstar}^2}-c_2\norm{\nabla_{\RRVset} u(g)}^2}
\end{split}\\
&\quad=E(\xx)+E'(\xx)u+\frac12E''(\xx)(u,u)-c_2\norm{\nabla_{\RRVset} u}_2^2
\ge E(\xx)+\tfrac\lambdanewa{4}\newnorm u\RR^2,
\end{align*}
where $N\in\MM$ such that $u$ is $\T^N$-periodic and we used \ref{item:Rotation} in the first, \eqref{eq:ectwo} in the second, Lemma~\ref{Lemma:FrechetE} in the third and \eqref{eq:econe} in the last step.

Now we assume that $d_1=d$, \ie $\G$ is finite.
Thus we have $\Uiso\RR=\Unewiso\RR$.
By Proposition~\ref{Proposition:Manifold} there exists a neighborhood $U\subset\E(d)$ of $\id$ such that the set
\[M:=\set[\Big]{u\in\UPer}{\exists\+ a\in U\+\forall g\in\G : g\cdot x_0+\rot(g)\proj{u}(g)=a\cdot (g\cdot x_0)}\]
is a manifold and $\UIso$ is its tangent space at $0$.
For all $u\in M$ and $v\in\UPer$ we have
\begin{align}
\begin{split}
&E(\xx+u+v)
=\frac1{\abs\G}\sum_{g\in\G}V\parens[\Big]{\parens[\big]{(gh)\cdot x_0+\rot(gh)\proj{u}(gh)+\rot(gh)\proj{v}(gh)\\
&\quad\quad\quad\quad\quad\quad\quad\quad\quad\quad\quad\quad-g\cdot x_0-\rot(g)\proj{u}(g)-\rot(g)\proj{v}(g)}_{h\in\Gstar}}
\end{split}\nonumber\\
&\quad=\frac1{\abs\G}\sum_{g\in\G}V\parens[\Big]{\parens[\big]{(agh)\cdot x_0+\rot(gh)\proj{v}(gh)-(ag)\cdot x_0-\rot(g)\proj{v}(g)}_{h\in\Gstar}}\nonumber\\
&\quad=\frac1{\abs\G}\sum_{g\in\G}V\parens[\Big]{\parens[\big]{A\parens[\big]{(gh)\cdot x_0+\rot(gh)\proj{w}(gh)-g\cdot x_0-\rot(g)\proj{w}(g)}}_{h\in\Gstar}}\nonumber\\
&\quad=\frac1{\abs\G}\sum_{g\in\G}V\parens[\Big]{\parens[\big]{(gh)\cdot x_0+\rot(gh)\proj{w}(gh)-g\cdot x_0-\rot(g)\proj{w}(g)}_{h\in\Gstar}}\nonumber\\
&\quad=E\parens[\big]{\xx+w},
\label{eq:manifold}
\end{align}
where $a=\iso Ab\in U$ such that $g\cdot x_0+\rot(g)\proj{u}(g)=(ag)\cdot x_0$ for all $g\in\G$, the function $w\colon\G\to\R^d$ is defined by $g\mapsto \rot(g)^{\mathsf T}A^{\mathsf T}\rot(g)v(g)$, and we used \ref{item:Rotation} in the second to last step.
In particular we have
\begin{equation}\label{eq:manifoldE}
E(\xx+u)=E(\xx)\qquad\text{for all }u\in M.
\end{equation}
Since $\ee=0$, by \eqref{eq:manifold} and Lemma~\ref{Lemma:energyderivativeY} for all $u\in M$ and $v\in\UPer$ we have
\begin{align}
E'(\xx+u)v&=\lim_{t\to0}\frac{E(\xx+u+tv)-E(\xx+u)}t\nonumber\\
&=\lim_{t\to0}\frac{E\parens[\big]{\xx+tw}-E(\xx)}t
=E'(\xx)w
=0,\label{eq:manifoldDEhelp}
\end{align}
where $w$ is defined as above.
By \eqref{eq:manifoldDEhelp} we have
\begin{equation}\label{eq:manifoldDE}
E'(\xx+u)=0\qquad\text{for all }u\in M.
\end{equation}
In the following, $c>0$ denotes a sufficiently small constant, which may vary from line to line.
Since $\lambdanewa>0$, we have
\[E''(\xx)(u,u)\ge c\newnorm u\RR^2\qquad\text{for all }u\in\UPer.\]
Let $\UIso^\bot$ be the orthogonal complement of $\UIso$ with respect to $\norm\fdot_2$.
By Theorem~\ref{Theorem:StrongerEquivalence}\ref{Theorem:NewEquivalenceAll} the seminorm $\newnorm\fdot\RR|_{\UIso^\bot}$ is a norm and thus we have
\[E''(\xx)(u,u)\ge c\norm u_\infty^2\qquad\text{for all }u\in\UIso^\bot.\]
Since $E''$ is continuous in $(\UPer,\norm\fdot_\infty)$, without loss of generality we may assume that $M$ is such that
\begin{equation}\label{eq:manifoldTaylor}
E''(\xx+u)(v,v)\ge c\norm v_\infty^2\qquad\text{for all }u\in M\text{ and }v\in\UIso^\bot.
\end{equation}
Without loss of generality let $M$ be such that by \eqref{eq:manifoldDE}, \eqref{eq:manifoldTaylor}, Taylor's theorem and \eqref{eq:manifoldE} there exists a neighborhood $V\subset\UIso^\bot$ of $0$ such that
\[E(\xx+u+v)\ge E(\xx+u)=E(\xx)\qquad\text{for all }u\in M\text{ and }v\in V.\]
Since $M+V\subset\UPer$ is a neighborhood of $0$, the assertion is proven.
\end{proof}
%
%%%%%%%%%%%%%%%%
%
\subsection{Second order bounds on the energy}\label{Section:Boudedness}
In this section we prove Proposition~\ref{Proposition:BoundedOne},  Theorems~\ref{Theorem:BoundedThree}, \ref{Theorem:FiniteSpace} and \ref{Theorem:dtwodtwo}. We also discuss the two Examples~\ref{Example:DEUiso} and~\ref{Example:counterV} alluded to in Remarks~\ref{Remark:BilinearFormLambda}\ref{item:lambdanotfinite} and~\ref{Remark:Hessianbounds}. 
Throughout we assume that $V$ has finite interaction range. 

We begin with the general, yet weak estimate provided by Proposition~\ref{Proposition:BoundedOne}.
\begin{proof}[Proof of Proposition~\ref{Proposition:BoundedOne}]
Let $\RRVset/\G_{x_0}\subset\Gstar$ be a finite interaction range of $V$.
By Theorem~\ref{Theorem:StrongerEquivalence}\ref{Theorem:nablaequivalent} we may assume that $\RRVset\subset\RR$.
There exists a constant $C>0$ such that
\begin{equation}\label{eq:eob}
\abs[\big]{V''(y_0)(z,z)}\le C\norm{z|_{\RRVset}}^2\qquad\text{for all }z\in L^\infty(\Gstar,\R^d).
\end{equation}
Let $u\in\UPer$ and $N\in\MM$ such that $u$ is $\T^N$-periodic.
We have
\begin{align}
\abs[\big]{E''(\xx)(u,u)}
&=\abs[\Big]{\frac1{\abs{\CC_N}}\sum_{g\in\CC_N}V''(y_0)\parens[\big]{\parens[\big]{\rot(h)\proj{u}(gh)-\proj{u}(g)}_{h\in\Gstar}^2}} \nonumber\\
&\le\frac C{\abs{\CC_N}}\sum_{g\in\CC_N}\norm{\nabla_\RR u(g)}^2
=C\norm{\nabla_\RR u}_2^2,\nonumber%\label{eq:BoundedOne}
\end{align}
where we used Lemma~\ref{Lemma:FrechetE} in the first step and \eqref{eq:eob} in the second step.
Since $E''(\xx)$ is a symmetric bilinear form, the assertion follows.
\end{proof}
\subsubsection{Structures with equilibrized onsite potentials}
Note first that the property~\ref{item:Rotation} of $V$ implies the following lemma.
\begin{Lemma}\label{Lemma:PropertiesV}
For all $S\in\Skew(d)$ and $z\colon\Gstar\to\R^d$ we have
\[V''(y_0)(Sy_0,z)=-V'(y_0)(Sz).\]
\end{Lemma}
\begin{proof}
By \ref{item:Rotation} for all $z\colon\Gstar\to\R^d$ and $A\in\SO(d)$ we have
\begin{equation}\label{eq:Vtwofinite}
V'(Ay_0)(Az)=\lim_{t\to 0}\frac{V(Ay_0+tAz)-V(Ay_0)}t=\lim_{t\to0}\frac{V(y_0+tz)-V(y_0)}t=V'(y_0)z.
\end{equation}
For all $S\in\Skew(d)$ and $z\colon\Gstar\to\R^d$ we have
\begin{align*}
V''(y_0)(Sy_0,z)&=\lim_{t\to0}\frac{V'(y_0+tSy_0)z-V'(y_0)z}t\\
&=\lim_{t\to0}\frac{V'(\euler^{-tS}(y_0+tSy_0))(\euler^{-tS}z)-V'(y_0)z}t\\
&=\lim_{t\to0}\frac{V'(y_0)((I_d-tS)z)-V'(y_0)z}t
=-V'(y_0)(Sz),
\end{align*}
where we used \eqref{eq:Vtwofinite} in the second step and Taylor's theorem in the third step.
\end{proof}
\begin{Remark}\label{Remark:DVRot}
If $V$ does not have finite interaction range, then for all $S\in\Skew(d_1)\oplus\{0_{d_2,d_2}\}$ and $z\in L^\infty(\Gstar,\R^d)$ we have
\[V''(y_0)(Sy_0,z)=-V'(y_0)(Sz).\]
The proof is analogous since $S(g\cdot x_0)=S\rot(g)x_0$ for all $g\in\G$ and thus $Sy_0\in L^\infty(\Gstar,\R^d)$ for all $S\in\Skew(d_1)\oplus\{0_{d_2,d_2}\}$.
\end{Remark}
We can now prove Theorem~\ref{Theorem:BoundedThree}.
\begin{proof}[Proof of Theorem~\ref{Theorem:BoundedThree}] 
Suppose that $V'(y_0)=0$.
Let $u\in\UPer$ and $N\in\MM$ such that $u$ is $\T^N$-periodic.
Since $\rot(h')(\nabla_\RR u(g))(h')=\rot(h'')(\nabla_\RR u(g))(h'')$ for $g\in\CC_N$ and $h',h''\in\RR$ with $h'\G_{x_0}=h''\G_{x_0}$, there exists some $\T^N$-periodic $S\in\Per(\G,\Skew(d))$ such that
\[\nabla_\RR u(g)=\pi_{\Urot\RR}(\nabla_\RR u(g))+\parens[\big]{\rot(h)^{\mathsf T}S(g)(h\cdot x_0-x_0)}_{h\in\RR}\qquad\text{for all }g\in\CC_N,\]
where $\pi_{\Urot\RR}$ is the orthogonal projection on $\{v\colon\RR\to\R^d\}$ with respect to the norm $\norm\fdot$ with kernel $\Urot\RR$.
In the following, $C>0$ denotes a sufficiently large constant, which is independent of $u$, and may vary from line to line.
Let $\RRVset/\G_{x_0}\subset\Gstar$ be a finite interaction range of $V$.
By Theorem~\ref{Theorem:StrongerEquivalence}\ref{Theorem:EquivalenceAll} we may assume that $\RRVset\subset\RR$.
Note that
\begin{equation}\label{eq:eobc}
\abs[\big]{V''(y_0)(z,z)}\le C\norm{z|_{\RRVset}}^2\le C\norm{z|_{\RR}}^2\qquad\text{for all }z\colon\Gstar\to\R^d.
\end{equation}
We have
\begin{align*}
&\abs[\big]{E''(\xx)(u,u)}
=\abs[\Big]{\frac1{\abs{\CC_N}}\sum_{g\in\CC_N}V''(y_0)\parens[\big]{\parens[\big]{\rot(h)\proj{u}(gh)-\proj{u}(g)}_{h\in\Gstar}^2}} \nonumber\\
\begin{split}
&\quad=\abs[\Big]{\frac1{\abs{\CC_N}}\sum_{g\in\CC_N}V''(y_0)\parens[\big]{\parens[\big]{\rot(h)\proj{u}(gh)-u(g)-S(g)(h\cdot x_0-x_0)}_{h\in\Gstar}^2}} 
\end{split}\nonumber\\
&\quad\le\frac C{\abs{\CC_N}}\sum_{g\in\CC_N}\norm[\big]{\pi_{\Urot\RR}(\nabla_\RR u(g))}^2
=C\nablanorm u\RR^2
\le C\norm u_\RR^2,\nonumber%\label{eq:rtl}
\end{align*}
where we used Lemma~\ref{Lemma:FrechetE} in the first step, Lemma~\ref{Lemma:PropertiesV} in the second step, \eqref{eq:eobc} in the third step and Theorem~\ref{Theorem:StrongerEquivalence}\ref{Theorem:EquivalenceAll} in the last step.
As $E''(\xx)$ is symmetric, the assertion follows.
\end{proof}
\subsubsection{Strong estimates for finite and space filling structures}
We now turn to the proof of Theorem~\ref{Theorem:FiniteSpace}. 
\begin{Proposition}\label{Proposition:EUiso}
Suppose that $E'(\xx)=0$.
Then we have
\[E''(\xx)(u,v)=0\qquad\text{for all }u\in\UIso\cap\UPer\text{ and }v\in\UPer.\]
\end{Proposition}
\begin{proof}
Let $u\in\UIso\cap\UPer$ and $v\in\UPer$.
By Proposition~\ref{Proposition:UtransUrotbig} there exist some $a\in\R^d$, $A_1\in\R^{(d-\daff)\times(\daff-d_2)}$ and $A_2\in\Skew(\daff-d_2)$ such that
\begin{equation}\label{eq:EUiso1}
\rot(g)\proj{u}(g)=a+S(g\cdot x_0-x_0)\qquad\text{for all }g\in\G,
\end{equation}
where $S=\parens[\bigg]{\begin{matrix}\mbox{\footnotesize $0$}&\mbox{\footnotesize $A_3$}\\\mbox{\footnotesize $-A_3^{\mathsf T}$}&\mbox{\footnotesize $A_4$}\end{matrix}}\in\Skew(d_1)\oplus\{0_{d_2,d_2}\}$, $A_3=(\begin{matrix} A_1 & 0_{d-\daff,d_2} \end{matrix})$ and $A_4=A_2\oplus0_{d_2,d_2}$. 
Let $N\in\MM$ such that $u$ and $v$ are $\T^N$-periodic. 
Since $\RR$ is admissible and in particular $\aff(\RR\cdot x_0)=\aff(\G\cdot x_0)$, the matrix $C\in\R^{\daff\times\abs{\Rcoset}}$ defined by
\[ \parens[\bigg]{\begin{matrix}0\\C\end{matrix}}=(h\cdot x_0-x_0)_{h\in\Rcoset} \]
has rank $\daff$. 
For all $g\in\G$ and $t\in\T^N$ we have
\begin{align}\label{eq:EUiso2}
&\parens[\big]{\rot(h)\proj{u}(gh)-\proj{u}(g)}_{h\in\Rcoset}
=\parens[\big]{\rot(h)\proj{u}(gth)-\proj{u}(gt)}_{h\in\Rcoset}\nonumber\\
&\qquad=\parens[\big]{\rot(gt)^{\mathsf T}S\rot(gt)(h\cdot x_0-x_0)}_{h\in\Rcoset}
=\parens[\bigg]{\begin{matrix}
A_3B_1B_2C\\
B_2^{\mathsf T}B_1^{\mathsf T}A_4B_1B_2C
\end{matrix}},
\end{align}
where we used the $\T^N$-periodicity of $u$ in the first, \eqref{eq:EUiso1} in the second step and $B_1,B_2\in\O(\daff)$ such that $\rot(g)=I_{d-\daff}\oplus B_1$ and $\rot(t)=I_{d-\daff}\oplus B_2$.
Since the left hand side of \eqref{eq:EUiso2} is independent of $t$ and $C$ has full rank, \eqref{eq:EUiso2} implies
\begin{equation}\label{eq:EUiso3}
\rot(gt)^{\mathsf T}S\rot(gt)=\rot(g)^{\mathsf T}S\rot(g)\qquad\text{for all }g\in\G\text{ and }t\in\T^N.
\end{equation}
We have
\begin{align*}
E''(\xx)(u,v)=\frac1{\abs{\CC_N}}\sum_{g\in\CC_N}V''(y_0)\parens[\big]{\parens[\big]{\rot(h)\proj{u}(gh)-\proj{u}(g)}_{h\in\Gstar},
\parens[\big]{\rot(h)\proj{v}(gh)-\proj{v}(g)}_{h\in\Gstar}}.
\end{align*}
We now split $\parens[\big]{\rot(h)\proj{v}(gh)-\proj{v}(g)}_{h\in\Gstar}=\parens[\big]{\rot(h)\proj{v}(gh)}_{h\in\Gstar}-\parens[\big]{\proj{v}(g)}_{h\in\Gstar}$. To further compute the first term for each $g\in\CC_N$ and $h'\in\G$ we fix $g_{g,h'}'\in g\T^N$ to be specified later and notice that, due to the periodicity of $u$ and $v$, $u(gh)=u(g_{g,h'}'h)$ and $v(gh)=v(g_{g,h'}'h)$ for all $g\in\CC_N$ and $h\in\G$. Then
\begin{align*}
&\frac1{\abs{\CC_N}}\sum_{g\in\CC_N}V''(y_0)\parens[\Big]{\parens[\Big]{\rot(h)\proj{u}(gh)-\proj{u}(g)}_{h\in\Gstar},\parens[\big]{\rot(h)\proj{v}(gh)}_{h\in\Gstar}}\\
&\quad=\frac1{\abs{\CC_N}}\sum_{g\in\CC_N}V''(y_0)\parens[\Big]{\parens[\big]{\rot(g)^{\mathsf T}S\rot(g)(h\cdot x_0-x_0)}_{h\in\Gstar},\parens[\Big]{\frac1{\abs{\G_{x_0}}}\sum_{h'\in h}\rot(h')v(gh')}_{h\in\Gstar}}\\
&\quad=-\frac1{\abs{\CC_N}}\sum_{g\in\CC_N}V'(y_0)\parens[\Big]{\frac1{\abs{\G_{x_0}}}\sum_{h'\in h}\rot(g_{g,h'}')^{\mathsf T}S\rot(g_{g,h'}')\rot(h')v(g_{g,h'}'h')}_{h\in\Gstar}\\
&\quad=-\frac1{\abs{\CC_N}}V'(y_0)\parens[\Big]{\frac1{\abs{\G_{x_0}}}\sum_{h'\in h}\rot(h')\sum_{g\in\CC_N}\rot(g)^{\mathsf T}S\rot(g)v(g)}_{h\in\Gstar},
\end{align*}
where we used Lemma~\ref{Lemma:PropertiesV} and \eqref{eq:EUiso3} in the second step and
for each $h'\in\G$ we have chosen $(g_{g,h'}')_{g\in\CC_N}$ in such a way that $\set{g_{g,h'}'}{g\in\CC_N}=\CC_Nh'^{-1}$ so that by the periodicity of $v$ 
\begin{align*}
\sum_{g\in\CC_N}\rot(g_{g,h'}')^{\mathsf T}S\rot(g_{g,h'}')\rot(h')v(h')&=\sum_{g\in\CC_N}\rot(gh'^{-1})^{\mathsf T}S\rot(gh'^{-1})\rot(h')v(g)\\
&=\sum_{g\in\CC_N}\rot(h')\rot(g)^{\mathsf T}S\rot(g)v(g).
\end{align*}
For the second contribution we proceed likewise to arrive at
\begin{align*}
E''(\xx)(u,v)&=-\frac1{\abs{\CC_N}}V'(y_0)\parens[\Big]{\parens[\big]{L_h-L_{\G_{x_0}}}\sum_{g\in\CC_N}\rot(g)^{\mathsf T}S\rot(g)v(g)}_{h\in\Gstar}
=0
\end{align*}
with the help of Lemma~\ref{Lemma:energyderivativeY}.
\end{proof}
\begin{Remark}\label{Remark:Assumption}
\begin{enumerate}
\item In the above proposition the assumption $E'(\xx)=0$ is essential, see Example~\ref{Example:DEUiso}.
\item\label{item:Assumption} In the above proposition the assumption that $V$ has finite interaction range is not necessary.
Using Remark~\ref{Remark:DVRot} instead of Lemma~\ref{Lemma:PropertiesV}, the proof is analogous.
\end{enumerate}
\end{Remark}
We can now prove Theorem~\ref{Theorem:FiniteSpace}. 
\begin{proof}[Proof of Theorem~\ref{Theorem:FiniteSpace}]
\ref{Theorem-item:Finite} We have $\lambdaa=\lambdanewa$ since $\G$ being finite entails $\norm\fdot_\RR=\newnorm\fdot\RR$.
Let $U$ be a subspace of $\UPer$ such that $\UPer=\UIso\oplus U$.
By Theorem~\ref{Theorem:StrongerEquivalence}\ref{Theorem:EquivalenceAll} the seminorm $\norm\fdot_\RR$ is a norm on $U$ and thus there exists a constant $C>0$ such that $\norm\fdot_\infty\le C\norm\fdot_\RR$ on $U$.
We have
\begin{align*}
\sup&\set[\big]{\abs{E''(\xx)(u,u)}}{u\in\UPer, \norm u_\RR\le1}=\sup\set[\big]{\abs{E''(\xx)(u,u)}}{u\in U,\norm u_\RR\le1}\\
&\le \sup\set[\big]{\abs{E''(\xx)(u,u)}}{u\in U,\norm u_\infty\le C}
<\infty,
\end{align*}
where we used Proposition~\ref{Proposition:EUiso} and Theorem~\ref{Theorem:StrongerEquivalence}\ref{Theorem:EquivalenceAll} in the first step and in the last step that $E''(\xx)$ is bounded with respect to $\norm\fdot_\infty$ by Lemma~\ref{Lemma:FrechetE}.
Since $E''(\xx)$ is symmetric, the assertion follows.

\ref{Theorem-item:Space} This is clear by Proposition~\ref{Proposition:BoundedOne} and Theorem~\ref{Theorem:StrongerEquivalence}\ref{Theorem:seminormequivalence}.
\end{proof}

\begin{Example}\label{Example:DEUiso}
We present an example such that $E''(\xx)(u,u)<0$ for some $u\in\UIso\cap\UPer$.
In particular we have $\lambdaa=\lambdanewa=-\infty$, $E''(\xx)$ is not bounded with respect to $\newnorm\fdot\RR$, and in Proposition~\ref{Proposition:EUiso} and Theorem~\ref{Theorem:FiniteSpace}\ref{Theorem-item:Finite} the condition $E'(\xx)=0$ cannot be dropped.

Let $d=d_2=2$, $p=\iso{-I_2}0\in\E(2)$, $\G=\{\id,p\}<\E(2)$, $x_0=e_1\in\R^2$ and
\[V\colon\R^2\to\R,\quad x\mapsto-\norm x^2.\]
We define the function $u\in\UIso$ by $\rot(g)u(g)=\parens[\Big]{\begin{matrix}\mbox{\footnotesize $0$}&\mbox{\footnotesize $1$}\\\mbox{\footnotesize $-1$}&\mbox{\footnotesize $0$}\end{matrix}}(g\gdot x_0-x_0)$ for all $g\in\G$.
We have $\G_{x_0}=\{\id\}$, $y_0=p\gdot x_0-x_0=-2e_1$ and by Lemma~\ref{Lemma:FrechetE}
\begin{align*}
E''(\xx)(u,u)&=\frac1{\abs\G}\sum_{g\in\G}V''(y_0)(-u(gp)-u(g),-u(gp)-u(g))\\
&=V''(y_0)(u(\id)+u(p),u(\id)+u(p))
=-2\norm{u(\id)+u(p)}^2
=-8,
\end{align*}
while $\norm u_\RR=\newnorm u\RR=0$, from which the above assertions follow. 
\end{Example}

\subsubsection{Estimates for lower dimensional infinite structures}
We now give the proof of Theorem~\ref{Theorem:dtwodtwo} which turns out rather demanding. 
\begin{Definition}
For all $u\in\UPer$ we define the function
\[S_u\in L^\infty\parens[\bigg]{\G,\set[\bigg]{\parens[\bigg]{\begin{smallmatrix}0&A_1&0\\-A_1^{\mathsf T}&A_2&0\\0&0&0\end{smallmatrix}}}{A_1\in\R^{(d-\daff)\times(\daff-d_2)},A_2\in\Skew(\daff-d_2)}}\]
by the condition
\[\nabla_\RR u(g)=\pi_{\Unewrot\RR}(\nabla_\RR u(g))+\parens[\big]{\rot(gh)^{\mathsf T}S_u(g)\rot(g)(h\cdot x_0-x_0)}_{h\in\RR}\qquad\text{for all }g\in\G,\]
where $\pi_{\Unewrot\RR}$ is the orthogonal projection on $\{v\colon\RR\to\R^d\}$ with respect to the norm $\norm\fdot$ with kernel $\Unewrot\RR$.
\end{Definition}
\begin{Remark}
For all $u\in\UPer$ the function $S_u$ is well-defined:
Let $g\in\G$. 
By Lemma~\ref{Lemma:WLOGdOStwo} there exist $B_1\in\O(\daff-d_2)$ and $B_2\in\O(d_2)$ such that $\rot(g)=I_{d-\daff}\oplus B_1\oplus B_2$. 
Since $\pi(\nabla_\RR u(g))=\nabla_\RR u(g)$, setting $T=\R^{(d-\daff)\times(\daff-d_2)}\times\Skew(\daff-d_2)$ and writing $A=(A_1,A_2)$ for $A\in T$, by Proposition~\ref{Proposition:UtransUrotbig} we have
\begin{align*}
&\pi(\Unewrot\RR)=\set[\Big]{\RR\to\R^d,h\mapsto\rot(h)^{\mathsf T}\parens[\Big]{\parens[\Big]{\begin{smallmatrix}0&A_1\\-A_1^{\mathsf T}&A_2\end{smallmatrix}}\oplus0_{d_2,d_2}}(h\cdot x_0-x_0)}{A\in T}\\
&\quad=\set[\Big]{\RR\to\R^d,h\mapsto\rot(h)^{\mathsf T}\parens[\Big]{\parens[\Big]{\begin{smallmatrix}0&B_1^{\mathsf T}A_1B_2\\-B_2^{\mathsf T}A_1^{\mathsf T}B_1&B_2^{\mathsf T}A_2B_2\end{smallmatrix}}\oplus0_{d_2,d_2}}(h\cdot x_0-x_0)}{A\in T}\\
&\quad=\set[\Big]{\RR\to\R^d,h\mapsto\rot(gh)^{\mathsf T}\parens[\Big]{\parens[\Big]{\begin{smallmatrix}0&A_1\\-A_1^{\mathsf T}&A_2\end{smallmatrix}}\oplus0_{d_2,d_2}}\rot(g)(h\cdot x_0-x_0)}{A\in T}.
\end{align*}
\end{Remark}
\begin{Lemma}\label{Lemma:Three}
For all $g_0\in\G$ there exists a constant $C>0$ such that
\[\frac1{\abs{\CC_N}}\sum_{g\in\CC_N}\norm{S_u(gg_0)-S_u(g)}^2\le C\newnorm u\RR^2\]
for all $u\in\UPer$ and $N\in\MM$ such that $u$ is $\T^N$-periodic.
\end{Lemma}
\begin{proof}
Let $g_0\in\G$.
Since $\RR$ is admissible and so $\aff(\RR\cdot x_0)=\aff(\G\cdot x_0)$, there exists some $\RR'\subset\RR$ and $A\in\GL(\daff)$ such that
\[(g\cdot x_0-x_0)_{g\in\RR'}=\parens[\bigg]{\begin{matrix}0\\A\end{matrix}}.\]
By Theorem~\ref{Theorem:StrongerEquivalence}\ref{Theorem:NewEquivalenceAll} without loss of generality, we may assume that $g_0\G_{x_0}\cup g_0\RR'\G_{x_0}\subset\RR$.
Let $u\in\UPer$ and $N\in\MM$ such that $u$ is $\T^N$-periodic.
Using that $g_0\in\RR$ we have
\begin{align}
\newnablanorm u\RR^2&=\frac1{\abs{\CC_N}}\sum_{g\in\CC_N}\norm[\Big]{\nabla_\RR u(g)-\parens[\big]{\rot(gh)^{\mathsf T}S_u(g)\rot(g)(h\cdot x_0-x_0)}_{h\in\RR}}^2\nonumber\\
&\ge\frac1{\abs{\CC_N}}\sum_{g\in\CC_N}\norm[\big]{\rot(g_0)\proj{u}(gg_0)-\proj{u}(g)-\rot(g)^{\mathsf T}S_u(g)\rot(g)(g_0\cdot x_0-x_0)}^2.\label{eq:ThreeOne}
\end{align}
As $g_0\RR'\G_{x_0}\subset\RR$, we also have
\begin{align}
&\newnablanorm u\RR^2=\frac1{\abs{\CC_N}}\sum_{g\in\CC_N}\norm[\Big]{\nabla_\RR u(g)-\parens[\big]{\rot(gh)^{\mathsf T}S_u(g)\rot(g)(h\cdot x_0-x_0)}_{h\in\RR}}^2\nonumber\\
&\quad\ge\frac1{\abs{\CC_N}}\sum_{g\in\CC_N}\sum_{h\in\RR'}\norm[\big]{\rot(g_0h)\proj{u}(gg_0h)-\proj{u}(g)-\rot(g)^{\mathsf T}S_u(g)\rot(g)((g_0h)\cdot x_0-x_0)}^2.\label{eq:ThreeTwo}
\end{align}
Since $\CC_Ng_0$ is a representation set of $\G/\T^N$ and $\RR'\G_{x_0}\subset\RR$, we furthermore have
\begin{align}
&\newnablanorm u\RR^2=\frac1{\abs{\CC_N}}\sum_{g\in\CC_N}\norm[\Big]{\nabla_\RR u(gg_0)-\parens[\big]{\rot(gg_0h)^{\mathsf T}S_u(gg_0)\rot(gg_0)(h\cdot x_0-x_0)}_{h\in\RR}}^2\nonumber\\
&\quad\ge\frac1{\abs{\CC_N}}\sum_{g\in\CC_N}\sum_{h\in\RR'}\norm[\big]{\rot(g_0h)\proj{u}(gg_0h)-\rot(g_0)\proj{u}(gg_0)
-\rot(g)^{\mathsf T}S_u(gg_0)\rot(gg_0)(h\cdot x_0-x_0)}^2.\label{eq:ThreeThree}
\end{align}
By \eqref{eq:ThreeOne}, \eqref{eq:ThreeTwo} and \eqref{eq:ThreeThree} there exists a constant $c>0$ (independent of $u$ and $N$) such that
\begin{align}
\newnablanorm u\RR^2&\ge\frac c{\abs{\CC_N}}\sum_{g\in\CC_N}\sum_{h\in\RR'}\norm[\big]{\rot(g_0)\proj{u}(gg_0)-\proj{u}(g)
-\rot(g)^{\mathsf T}S_u(g)\rot(g)(g_0\cdot x_0-x_0)\nonumber \\
&\qquad -\rot(g_0h)\proj{u}(gg_0h)+\proj{u}(g)
+\rot(g)^{\mathsf T}S_u(g)\rot(g)((g_0h)\cdot x_0-x_0)\nonumber\\
&\qquad +\rot(g_0h)\proj{u}(gg_0h)-\rot(g_0)\proj{u}(gg_0)
-\rot(g)^{\mathsf T}S_u(gg_0)\rot(gg_0)(h\cdot x_0-x_0)}^2\nonumber\\
&=\frac c{\abs{\CC_N}}\sum_{g\in\CC_N}\norm[\bigg]{(S_u(g)-S_u(gg_0))\rot(gg_0)\parens[\bigg]{\begin{matrix}0\\A\end{matrix}}}^2.\label{eq:ThreeFour}
\end{align}
By Lemma~\ref{Lemma:WLOGdOStwo} for all $g\in\CC_N$ there exist $B(g)\in\O(\daff)$, $T_1(g)\in\R^{(d-\daff)\times\daff}$ and $T_2(g)\in\Skew(\daff)$ such that
\[\rot(gg_0)=\parens[\bigg]{\begin{matrix}I_{d-\daff}&0\\0&B(g)\end{matrix}}\text{ and }S_u(g)-S_u(gg_0)=\parens[\bigg]{\begin{matrix}0&T_1(g)\\-T_1(g)^{\mathsf T}&T_2(g)\end{matrix}}.\]
By \eqref{eq:ThreeFour} we have
\begin{align}
\newnablanorm u\RR^2&\ge\frac c{\abs{\CC_N}}\sum_{g\in\CC_N}\parens[\Big]{\norm{T_1(g)B(g)A}^2+\norm{T_2(g)B(g)A}^2}\nonumber\\
&\ge\frac{c\sigma_{\min}^2(A)}{\abs{\CC_N}}\sum_{g\in\CC_N}\parens[\Big]{\norm{T_1(g)}^2+\norm{T_2(g)}^2}\nonumber\\
&\ge\frac{c\sigma_{\min}^2(A)}{2\abs{\CC_N}}\sum_{g\in\CC_N}\norm{S_u(g)-S_u(gg_0)}^2,\label{eq:ThreeFive}
\end{align}
where $\sigma_{\min}(A)>0$ denotes the minimum singular value of $A$.
Theorem~\ref{Theorem:StrongerEquivalence}\ref{Theorem:NewEquivalenceAll} and \eqref{eq:ThreeFive} imply the assertion.
\end{proof}
We are now in a position to prove Theorem~\ref{Theorem:dtwodtwo}. 
\begin{proof}[Proof of Theorem~\ref{Theorem:dtwodtwo}]
We first remark that for $d=1+d_2$ the assertion is a direct consequence of 
Theorem~\ref{Theorem:StrongerEquivalence}\ref{Theorem:NewEquivalenceAll} and Proposition~\ref{Proposition:BoundedOne} since then
\[\Unewrot\RR=\set[\big]{u\colon\RR\to\R^d}{\forall g\in\RR:\rot (g)\proj{u}(g)=0}\]
and thus $\newnablanorm\fdot\RR=\norm{\nabla_\RR\fdot}_2$. This establishes Remark~\ref{Remark:Hessianbounds}\ref{Remarkitem:donedtwo}. 

Since $E''(\xx)$ is symmetric, it suffices to show that there exists a constant $C>0$ such that
\begin{equation}\label{eq:TheoremZero}
E''(\xx)(u,u)\le C\newnorm u\RR^2\qquad\text{for all }u\in\UPer.
\end{equation}
Recall that $\MM=m_0\N$.
Let $\{t_1,\dots,t_{d_2}\}$ be a generating set of $\T^{m_0}$.
Without loss of generality we specifically choose 
\[\CC_{n{m_0}}=\bigcupdot_{n_1,\dots,n_{d_2}\in\{0,\dots,n-1\}}t_1^{n_1}\dots t_{d_2}^{n_{d_2}}\CC_{m_0}\qquad\text{for all }n\in\N.\]
For all $g\in\G$ there exist $n_{1,1},\dots,n_{\abs{\CC_{m_0}},d_2}\in\Z$ such that
\[\CC_{m_0}g=\bigcupdot_{i=1}^{\abs{\CC_{m_0}}} \{t_1^{n_{i,1}}\dots t_{d_2}^{n_{i,d_2}}h_i\},\]
where $h_1,\dots,h_{\abs{\CC_{m_0}}}$ are the elements of $\CC_{m_0}$.
Thus and since $\T^{m_0}$ is abelian, for all $g\in\G$ we have
\begin{equation}\label{eq:TheoremOne}
\lim_{n\to\infty}\frac{\abs{\CC_{n{m_0}}\cap(\CC_{n{m_0}}g)}}{\abs{\CC_{n{m_0}}}}=1.
\end{equation}
Let $\{B_1,\dots,B_m\}$ be an orthonormal basis of 
\[\set[\bigg]{\parens[\bigg]{\begin{matrix}0&A_1\\-A_1^{\mathsf T}&A_2\end{matrix}}\oplus 0_{d_2,d_2}}{A_1\in\R^{(d-\daff)\times(\daff-d_2)},A_2\in\Skew(\daff-d_2)}.\]
Let $u\in\UPer$ and $N\in\MM$ such that $u$ is $\T^N$-periodic. 
For all $i\in\{1,\dots,m\}$ and $g\in\G$ let $S_{u,i}(g)=\angles{S_u(g),B_i}B_i$. 
For all $n\in\N$ and $i\in\{1,\dots,m\}$ we define the $\T^{nN}$-periodic function $v_{u,n,i}\in\UPer$ by the condition 
\[\rot(g)v_{u,n,i}(g)=S_{u,i}(g)(g\cdot x_0-x_0)\qquad\text{for all }g\in\CC_{nN}.\]
Moreover let $v_{u,n}=\sum_{i=1}^mv_{u,n,i}$. 

Since $\trans(\G)\subset\{0_{d_1}\}\times\R^{d_2}$, for all $S\in\Skew(d_1)\times\{0_{d_2,d_2}\}$ and $g,h\in\G$ we have 
\begin{equation}\label{eq:TheoremSeven}
S(g\cdot x_0)=S\rot(g)x_0
\qquad\text{and}\qquad
S((gh)\cdot x_0)=S\rot(g)(h\cdot x_0).
\end{equation}
Since the bilinear form $E''(\xx)$ is positive semidefinite, for all $n\in\N$ we have 
\begin{equation}
E''(\xx)(u,u)\le 2E''(\xx)(u-v_{u,n},u-v_{u,n})+ 2m\sum_{i=1}^mE''(\xx)(v_{u,n,i},v_{u,n,i}).\label{eq:TheoremTwo}
\end{equation}
In the following, $C>0$ denotes a sufficiently large constant, which is independent of $u$ and may vary from line to line.
We write $\bar{S}_{u,i}(g)=\frac1{\abs{\G_{x_0}}}\sum_{h'\in\G_{x_0}}S_{u,i}(gh')$ and $\bar{S}_u(g)=\sum_{i=1}^m\bar{S}_{u,i}(g)$ for the averages of $S_{u,i}$, respectively, $S_u$ over cosets. Since $\pproj{v_{u,n}}(g)=\rot(g)^{\mathsf T}\bar{S}_{u,i}(g)(g\cdot x_0-x_0)$, we have
\begin{align*}
\limsup_{n\to\infty}&E''(\xx)(u-v_{u,n},u-v_{u,n})\le\limsup_{n\to\infty}C\norm[\big]{\nabla_\RR(u-v_{u,n})}_2^2\\
&=\limsup_{n\to\infty}\frac C{\abs{\CC_{nN}}}\sum_{g\in\CC_{nN}}\sum_{h\in\RR}\norm[\big]{\rot(h)\nabla_\RR u(g)(h)\\
&\qquad-\rot(g)^{\mathsf T}\bar{S}_u(gh)((gh)\cdot x_0-x_0)+\rot(g)^{\mathsf T}\bar{S}_u(g)(g\cdot x_0-x_0)}^2,
\end{align*}
where in the first step we used Proposition~\ref{Proposition:BoundedOne} and in the second \eqref{eq:TheoremOne}, which implies that $\lim_{n\to\infty}\abs{\set{g\in\CC_{nN}}{gh\in\CC_{nN}\text{ for all $h\in\RR$}}}/\abs{\CC_{nN}}=1$.) 
From \eqref{eq:TheoremSeven} and \eqref{eq:TheoremOne} it now follows that 
\begin{align}
\limsup_{n\to\infty}&E''(\xx)(u-v_{u,n},u-v_{u,n})\nonumber\\
&\le\limsup_{n\to\infty}\frac C{\abs{\CC_{nN}}}\sum_{g\in\CC_{nN}}\sum_{h\in\RR}\parens[\Big]{\norm[\big]{\rot(h)\nabla_\RR u(g)(h)-\rot(g)^{\mathsf T}S_u(g)\rot(g)(h\cdot x_0-x_0)}^2\nonumber\\
&\qquad+\norm{\bar{S}_u(gh)-\bar{S}_u(g)}^2+\norm{\bar{S}_u(g)-S_u(g)}^2}
\le C\newnorm u\RR^2,\label{eq:TheoremThree}
\end{align}
where the last estimate is implied by Lemma~\ref{Lemma:Three} and Theorem~\ref{Theorem:StrongerEquivalence}\ref{Theorem:NewEquivalenceAll}.

Let $i\in\{1,\dots,m\}$.
Using Lemma~\ref{Lemma:FrechetE}, \eqref{eq:TheoremOne} and \eqref{eq:TheoremSeven} we have
\begin{align}
&\limsup_{n\to\infty}E''(\xx)(v_{u,n,i},v_{u,n,i})\nonumber\\
&\quad=\limsup_{n\to\infty}\frac1{\abs{\CC_{nN}}}\sum_{g\in\CC_{nN}}V''(y_0)\parens[\Big]{\parens[\big]{\rot(h)\proj{v}_{u,n,i}(gh)-\proj{v}_{u,n,i}(g)}_{h\in\Gstar}^2}\nonumber\\
&\quad=\limsup_{n\to\infty}\frac1{\abs{\CC_{nN}}}\sum_{g\in\CC_{nN}}V''(y_0)\parens[\Big]{(a_i(g,h)+b_i(g,h))_{h\in\Gstar}^2}\nonumber\\
&\quad=\limsup_{n\to\infty}(s_{1,n,i}+s_{2,n,i}),\label{eq:TheoremFour}
\end{align}
where $\proj{v}_{u,n,i}=\pproj{v_{u,n,i}}$,
\begin{align*}
&a_i(g,h):=\rot(g)^{\mathsf T}(\bar{S}_{u,i}(gh)-\bar{S}_{u,i}(g))((gh)\cdot x_0-x_0),\\
&b_i(g,h):=\rot(g)^{\mathsf T}\bar{S}_{u,i}(g)\rot(g)(h\cdot x_0-x_0),
\intertext{and $\bar{S}_{u,i}(g\G_{x_0}):=\bar{S}_{u,i}(g)$ for all $g\in\G$ and $h\in\Gstar$, and}
&s_{1,n,i}:=\frac1{\abs{\CC_{nN}}}\sum_{g\in\CC_{nN}}V''(y_0)\parens[\Big]{(a_i(g,h))_{h\in\Gstar}^2},\\
&s_{2,n,i}:=\frac1{\abs{\CC_{nN}}}\sum_{g\in\CC_{nN}}V''(y_0)\parens[\Big]{(b_i(g,h))_{h\in\Gstar},(2a_i(g,h)+b_i(g,h))_{h\in\Gstar}}
\end{align*}
for all $n\in\N$.
Let $\RRVcoset\subset\Gstar$ be a finite interaction range of $V$.
We have
\begin{align}
\limsup_{n\to\infty}s_{1,n,i}&\le\limsup_{n\to\infty}\frac C{\abs{\CC_{nN}}}\sum_{g\in\CC_{nN}}\sum_{h\in\RRVcoset}\norm{a_i(g,h)}^2\nonumber\\
&\le\limsup_{n\to\infty}\frac C{\abs{\CC_{nN}}}\sum_{g\in\CC_{nN}}\sum_{h\in\RRVcoset}\norm{\bar{S}_{u,i}(gh)-\bar{S}_{u,i}(g)}^2
\le C\newnorm u\RR^2,\label{eq:TheoremFive}
\end{align}
where we used \eqref{eq:TheoremSeven} in the second and Lemma~\ref{Lemma:Three} in the last step.
Observe that 
\begin{align}
&\sum_{g\in\CC_{nN}}V'(y_0)\parens[\big]{\rot(g)^{\mathsf T}\bar{S}_{u,i}^2(g)(g\cdot x_0-x_0)}_{h\in\Gstar}\nonumber\\
&\quad=\sum_{g\in\CC_{nN}}\sum_{h\in\RRVcoset}\frac1{\abs{\G_{x_0}}}\sum_{h'\in\G_{x_0}}\partial_hV(y_0)\parens[\big]{\rot(h')\rot(g)^{\mathsf T}\bar{S}_{u,i}^2(g)(g\cdot x_0-x_0)}+o(\abs{\CC_{nN}})\nonumber\\
&\quad=\sum_{g\in\CC_{nN}}\sum_{h\in\RRVcoset}\frac1{\abs{\G_{x_0}}}\sum_{h'\in h}\partial_hV(y_0)\parens[\big]{\rot(h')\rot(g)^{\mathsf T}\bar{S}_{u,i}^2(g)(g\cdot x_0-x_0)}+o(\abs{\CC_{nN}})\nonumber\\
&\quad=\sum_{g\in\CC_{nN}}V'(y_0)\parens[\big]{\rot(g)^{\mathsf T}\bar{S}_{u,i}(gh)((gh)\cdot x_0-x_0)}_{h\in\Gstar}+o(\abs{\CC_{nN}}),\label{eq:TheoremSix}
\end{align}
where in the first and third step we used \eqref{eq:TheoremOne} to substitute $g$ by $g{h'}^{-1}$ for $h'\in\G_{x_0}$, respectively, $g$ by $gh'$ for $h'\in h$, and in the second that $\ee=0$ due to Lemma~\ref{Lemma:energyderivativeY}. We also wrote $o(\abs{\CC_{nN}})$ for terms that vanish as $n\to\infty$ when divided by $\abs{\CC_{nN}}$.
The definition of $S_{u,i}$ implies
\begin{equation}\label{eq:TheoremEight}
S_{u,i}(g)S_{u,i}(h)=S_{u,i}(h)S_{u,i}(g)\qquad\text{for all }g,h\in\G.
\end{equation}
We have
\begin{align}
&\limsup_{n\to\infty}s_{2,n,i}
=\limsup_{n\to\infty}\frac1{\abs{\CC_{nN}}}\sum_{g\in\CC_{nN}}V'(y_0)\parens[\big]{-\rot(g)^{\mathsf T}\bar{S}_{u,i}(g)\rot(g)(2a_i(g,h)+b_i(g,h))}_{h\in\Gstar}\nonumber\\
&=\limsup_{n\to\infty}\frac1{\abs{\CC_{nN}}}\sum_{g\in\CC_{nN}}V'(y_0)\parens[\Big]{-2\rot(g)^{\mathsf T}\bar{S}_{u,i}(g)\bar{S}_{u,i}(gh)((gh)\cdot x_0-x_0)\nonumber\\
&\qquad+\rot(g)^{\mathsf T}\bar{S}_{u,i}^2(g)((gh)\cdot x_0-x_0)+\rot(g)^{\mathsf T}\bar{S}_{u,i}^2(g)(g\cdot x_0-x_0)}_{h\in\Gstar}\nonumber\\
&=\limsup_{n\to\infty}\frac1{\abs{\CC_{nN}}}\sum_{g\in\CC_{nN}}V'(y_0)\parens[\big]{\rot(g)^{\mathsf T}(\bar{S}_{u,i}(g)-\bar{S}_{u,i}(gh))^2((gh)\cdot x_0-x_0)}_{h\in\Gstar}\nonumber\\
&\le\limsup_{n\to\infty}\frac C{\abs{\CC_{nN}}}\sum_{g\in\CC_{nN}}\sum_{h\in\RRVcoset}\norm{\bar{S}_{u,i}(g)-\bar{S}_{u,i}(gh)}^2
\le C\newnorm u\RR^2,\label{eq:TheoremNine}
\end{align}
where in the first step we used Lemma~\ref{Lemma:PropertiesV}, in the second step we used \eqref{eq:TheoremSeven}, in the third step we used \eqref{eq:TheoremSix} and \eqref{eq:TheoremEight}, in the forth step we used \eqref{eq:TheoremSeven}, and in the last step we used Lemma~\ref{Lemma:Three}.

Since $i\in\{1,\dots,m\}$ was arbitrary, the equations \eqref{eq:TheoremTwo}, \eqref{eq:TheoremThree}, \eqref{eq:TheoremFour}, \eqref{eq:TheoremFive} and \eqref{eq:TheoremNine} imply the assertion \eqref{eq:TheoremZero}.
\end{proof}
\begin{Remark}
The above assumption $E'(\xx)=0$ can be replaced by the weaker assumption $E'(\xx)u=0$ for all $u\in\UPer$ with $u(\G)\subset\R^{d_1}\times\{0_{d_2}\}$.
\end{Remark}
\begin{Example}\label{Example:counterV}
We present an example with an infinite interaction range, $E'(\xx)=0$ and $\lambdaa=\lambdanewa=-\infty$.
In particular $E''(\xx)$ is not bounded with respect to $\newnorm\fdot\RR$.

Let $d=d_2=1$, $d_1=0$, $t=\iso{I_1}1\in\E(1)$, $\G=\set{t^n}{n\in\Z}<\E(1)$ and $x_0=0\in\R$.
We have $\G_{x_0}=\{\id\}$ and $\MM=\N$.
Let $\alpha>1$ and $V\colon(\R^d)^\GstarExample\to\R$ be the interaction potential such that $V$ has the properties \ref{item:Rotation}, \ref{item:Vfrechet} and \ref{item:infty} and 
\[V''(y_0)(z_1,z_2)=-\sum_{n\in\N}n^{-\alpha}z_1(t^n)z_2(t^n)\qquad\text{for all }z_1,z_2\in L^\infty(\GstarExample,\R^d),\]
see Remark~\ref{Remark:VFrechet}\ref{item:RemarkGTrivial}.
Since $\G<\{I_d\}\ltimes\R^d$ we have $\ee=0$ and thus $\xx$ is a critical point of $E$ by Lemma~\ref{Lemma:energyderivativeY}.
Let $N\in\N$ be even.
The set $\{t^0,\dots,t^{N-1}\}$ is a representation set of $\G/\T^N$.
We define the $\T^N$-periodic function $u\in\UPer$ by
\begin{align*}
u(t^n)=\begin{cases}\frac nN&\text{for }n\in\{0,\dots,N/2-1\},\\
1-\frac nN&\text{for }n\in\{N/2,\dots,N-1\}.\end{cases}
\end{align*}
Let $\RR=\{\id,t,t^2\}$ and $\RR'=\{t\}$.
Then $\RR$ is an admissible $\id$-neighborhood and $\RR'$ generates $\G$.
By Theorem~\ref{Theorem:StrongerEquivalence}\ref{Theorem:seminormequivalence} and Theorem~\ref{Theorem:StrongerEquivalence}\ref{Theorem:nablaequivalent}, the seminorms $\norm\fdot_\RR$ and $\norm{\nabla_{\RR'}\fdot}$ are equivalent and thus there exists a constant $C>0$ such that $\norm\fdot_\RR\le C\norm{\nabla_{\RR'}\fdot}$.
We have
\begin{equation}\label{eq:hvm}
\norm u_\RR\le C\norm{\nabla_{\RR'} u}_2=C\parens[\bigg]{\frac1N\sum_{n=0}^{N-1}\norm{\nabla_{\RR'} u(t^n)}^2}^{\frac12}=\frac CN.
\end{equation}
We have
\begin{align}\label{eq:hvn}
&E''(\xx)(u,u)=\frac1N\sum_{n=0}^{N-1}V''(y_0)\parens[\Big]{\parens[\big]{u(t^ns)-u(t^n)}_{s\in\GstarExample},\parens[\big]{u(t^ns)-u(t^n)}_{s\in\GstarExample}}\nonumber\\
&\qquad=-\frac1N\sum_{n=0}^{N-1}\sum_{m\in\N}m^{-\alpha}\abs[\big]{u(t^{n+m})-u(t^n)}^2
\le-\frac1N\sum_{n=0}^{N-1}(N/2)^{-\alpha}\abs[\big]{u(t^{n+N/2})-u(t^n)}^2\nonumber\\
&\qquad\le-\frac12 (N/2)^{-\alpha} \frac1{4^2}
=-2^{\alpha-5}N^{-\alpha}.
\end{align}
By \eqref{eq:hvm} and \eqref{eq:hvn} we have
\[\lambdaa\le \frac{E''(\xx)(u,u)}{\norm u_\RR^2}\le -cN^{2-\alpha},\]
where $c=C^{-2}2^{\alpha-5}$.
For all $\alpha\in(1,2)$ we have $\lambdaa=-\infty$ as $N\in2\N$ was arbitrary.
Since $\norm\fdot_\RR=\newnorm\fdot\RR$, for all $\alpha\in(1,2)$ we also have $\lambdanewa=-\infty$.
\end{Example}
\subsection{Stability criteria}
In this section we prove our main results that characterize the stability constants $\lambdaa$ and $\lambdanewa$ in the Fourier transform domain: Theorem~\ref{Theorem:LambdaIrreducible} and, in particular, Theorem~\ref{Theorem:LambdaInduced}.

We first establish \eqref{eq:E-and-norm-as-conv}. 
Recall the definition of $\ff$ and $\gggg,\newgggg$ from Definition~\ref{Definition:eeff} and~\ref{Definition:gggg}, respectively.

\begin{Lemma}\label{Lemma:onetwoL}
Let $u,v\in\UPer$ and set $u_0=u(\fdot^{-1})$ and $v_0=v(\fdot^{-1})$. Then 
\begin{enumerate}
\item\label{Lemma-item:oneL}
$E''(\xx)(u,v)=\angles{\ff*v_0,u_0}$, 
\item\label{Lemma-item:twoL}
$\norm u_{\RR}=\norm{\gggg*u_0}_2$ and $\newnorm u\RR=\norm{\newgggg*u_0}_2$.
\end{enumerate}
\end{Lemma}
\begin{proof}
\ref{Lemma-item:oneL} Let $u,v\in\UPer$.
Let $N\in \MM$ such that $u$ and $v$ are $\T^N$-periodic.
Let $u_0=u(\fdot^{-1})$ and $v_0=v(\fdot^{-1})$.
By Lemmas~\ref{Lemma:energyderivativeY} and~\ref{Lemma:Convolution} we have
\begin{align*}
E''(\xx)(u,v)&=\sum_{g,h\in\CC_N}u(g)^{\mathsf T}\partial_{g\T^N}\partial_{h\T^N}E(\xx)v(h)\\
&=\frac1{\abs{\CC_N}}\sum_{g,h\in\CC_N}\sum_{t\in\T^N}u_0(g^{-1})^{\mathsf T}\ff(g^{-1}ht)v_0(h^{-1})\\
&=\frac1{\abs{\CC_N}}\sum_{g\in\CC_N}u_0(g^{-1})^{\mathsf T}\ff*v_0(g^{-1})
=\angles{\ff*v_0,u_0},
\end{align*}
where in the third step we used that $v_0((ht)^{-1})=v_0(h^{-1})$ for all $h\in\CC_N$ and $t\in\T^N$.

\ref{Lemma-item:twoL}
Let $u\in\UPer$ and $N\in \MM$ such that $u$ is $\T^N$-periodic and set $u_0=u(\fdot^{-1})$.
With $P\in\R^{(d\abs\RR)\times(d\abs\RR)}$ as in Definition~\ref{Definition:gggg} and $\delta_g\colon\G\to\{0,1\}$, $h\mapsto\delta_{h,g}$ we have 
\begin{align}\label{eq:Azu2}
P(u(gh))_{h\in \RR}
&=P(u_0(h^{-1}g^{-1}))_{h\in\RR}\nonumber
=P((\delta_hI_d)*u_0(g^{-1}))_{h\in\RR}\nonumber\\
&=(P(\delta_hI_d)_{h\in\RR})*u_0(g^{-1})
=\gggg*u_0(g^{-1})
\end{align}
for any $g\in\G$. So by \eqref{eq:Azu1} and \eqref{eq:Azu2} we obtain
\[\norm u_\RR^2=\frac1{\abs{\CC_N}}\sum_{g\in\CC_N}\norm{\gggg*u_0(g^{-1})}^2=\norm{\gggg*u_0}_2^2.\]
Analogously we have $\newnorm u\RR=\norm{\newgggg*u_0}_2$.
\end{proof}
The following lemma shows that we can consider complex-valued instead of real-valued functions. Its standard proof is included for completeness. 
\begin{Lemma}\label{Lemma:complexvalued}
We have (with matching choices of indices $0$)
\begin{align*}
&\lambda_{\textnormal{a}(,0,0)}=\sup\set[\big]{c\in\R}{\forall\+ u\in\UPerC:c\norm{g_{\RR(,0,0)}*u}_2^2\le\angles{\ff*u,u}}.
\end{align*}
\end{Lemma}
\begin{proof}
By Lemma~\ref{Lemma:onetwoL} and since $\UPer=\set{u(\fdot^{-1})}{u\in\UPer}$, we have
\[\lambdaa=\sup\set[\big]{c\in\R}{\forall\+ u\in\UPer:c\norm{\gggg*u}_2^2\le\angles{\ff*u,u}}\]
and hence,
\[\lambdaa\ge\sup\set[\big]{c\in\R}{\forall\+ u\in\UPerC:c\norm{\gggg*u}_2^2\le\angles{\ff*u,u}}=:\text{RHS}.\]
Now we show that $\lambdaa\le\text{RHS}$.
For all $u\in\UPerC$, observing that $\angles{\ff*\Re(u),\Im(u)}=\angles{\ff*\Im(u),\Re(u)}$ due to Lemma~\ref{Lemma:onetwoL}\ref{Lemma-item:oneL}, we have
\begin{align*}
\angles{\ff*u,u}
&=\angles{\ff*\Re(u),\Re(u)}+\angles{\ff*\Im(u),\Im(u)}\\
&\ge\lambdaa\norm{\gggg*\Re(u)}_2^2+\lambdaa\norm{\gggg*\Im(u)}_2^2
=\lambdaa\norm{\gggg*u}_2^2.
\end{align*}

The proof of the characterization of $\lambdanewa$ is analogous.
\end{proof}
Recall that by our definition all representations are unitary.
\begin{Lemma}\label{Lemma:rft}%\label{Lemma:ffffW}
For all $g\in\G$ we have $\ff(g^{-1})=\ff(g)^{\mathsf T}$ and for all representations $\rho$ of $\G$ the matrix $\fourier\ff(\rho)$ is Hermitian.
\end{Lemma}
\begin{proof}
The first claim follows by noting that $\delta_{g^{-1},h_2'^{-1}h_1'}=\delta_{g,h_1'^{-1}h_2'}$ in the defining formula for $\ff(g^{-1})$ from Definition~\ref{Definition:eeff}. For the second we note that for all representations $\rho$ of $\G$ we have
\begin{align*}
\fourier\ff(\rho)
&=\sum_{g\in\G}\ff(g)\otimes\rho(g)
=\sum_{g\in\G}\ff(g^{-1})\otimes\rho(g^{-1})\\
&=\sum_{g\in\G}\ff(g)^{\mathsf H}\otimes\rho(g)^{\mathsf H}
=\parens[\bigg]{\sum_{g\in\G}\ff(g)\otimes\rho(g)}^{\mathsf H}
=\fourier\ff(\rho)^{\mathsf H},
\end{align*}
where in the third step we used the first assertion and that $\rho$ is unitary.
\end{proof}
We now prove Theorem~\ref{Theorem:LambdaIrreducible}. 
\begin{proof}[Proof of Theorem~\ref{Theorem:LambdaIrreducible}]
By Lemma~\ref{Lemma:rft} for all $\rho\in\EE$ the matrix $\fourier\ff(\rho)$ is Hermitian and thus the term $\lambdamin(\fourier\ff(\rho),\fourier\gggg(\rho))$ is well-defined.
We have to show that
\[\lambdaa=\inf\set[\Big]{\lambdamin\parens[\Big]{\fourier\ff(\rho),\fourier\gggg(\rho)}}{\rho\in\EE}=:\text{RHS}.\]
By Lemma~\ref{Lemma:complexvalued} we have
\[\lambdaa=\sup\set[\big]{c\in\R}{\forall\+ u\in\UPerC:c\norm{\gggg*u}_2^2\le\angles{\ff*u,u}}.\]
First we show that $\lambdaa\le\text{RHS}$.
Let $\rho\in\EE$ and $a\in\C^{dd_\rho}$.
We define $u\in\UPerC$ by
\[\fourier u(\rho')=\begin{cases}
(\begin{matrix}a&0_{dd_\rho,d_\rho-1}\end{matrix}) & \text{if }\rho'=\rho\\
      0_{dd_{\rho'},d_{\rho'}} & \text{else}
    \end{cases}\]
for all $\rho'\in\EE$.
By Lemma~\ref{Lemma:Convolution} and Proposition~\ref{Proposition:TFplancherelmatrix} we have
\begin{multline*}
\scalar[\Big]{\fourier\ff(\rho)a}a=\scalar[\Big]{\fourier\ff(\rho)\fourier u(\rho)}{\fourier u(\rho)}=\scalar[\Big]{\fourier{\ff*u}(\rho)}{\fourier u(\rho)}=\frac1{d_\rho}\scalar{\ff*u}u\\
\ge\frac\lambdaa{d_\rho}\norm{\gggg*u}_2^2=\lambdaa\norm{\fourier{\gggg*u}(\rho)}^2=\lambdaa\norm{\fourier\gggg(\rho)\fourier u(\rho)}^2=\lambdaa\norm{\fourier\gggg(\rho)a}^2.
\end{multline*}
Since $a\in\C^{dd_\rho}$ was arbitrary, we have $\lambdamin(\fourier\ff(\rho),\fourier\gggg(\rho))\ge\lambdaa$.

Now we prove that $\lambdaa\ge\text{RHS}$.
Let $u\in\UPerC$.
For a matrix $A$ we denote its $i$th column by $A_i$.
We have
\begin{align*}
\scalar{\ff*u}u
&=\sum_{\rho\in\EE}d_\rho\scalar[\Big]{\fourier{\ff*u}(\rho)}{\fourier u(\rho)}\\
&=\sum_{\rho\in\EE}d_\rho\scalar[\Big]{\fourier\ff(\rho)\fourier u(\rho)}{\fourier u(\rho)}
=\sum_{\rho\in\EE}d_\rho\sum_{i=1}^{d_\rho}\scalar[\Big]{\fourier\ff(\rho)\fourier u(\rho)_i}{\fourier u(\rho)_i}\\
&\ge\text{RHS}\sum_{\rho\in\EE}d_\rho\sum_{i=1}^{d_\rho}\norm{\fourier\gggg(\rho)\fourier u(\rho)_i}^2
=\text{RHS}\sum_{\rho\in\EE}d_\rho\norm{\fourier\gggg(\rho)\fourier u(\rho)}^2
=\text{RHS}\norm{\gggg*u}_2^2.
\end{align*}

The proof of the characterization of $\lambdanewa$ is analogous.
\end{proof}
We now turn to the proof of Theorem~\ref{Theorem:LambdaInduced}. 
For the remainder of this section, we fix a complete set of representatives of the cosets of $\T\F$ in $\G$.
In the following we write $\Ind\rho$ for $\Ind_{\T\F}^\G\rho$ for all representations $\rho$ of $\T\F$.
Let $n_0=|\G:\T\F|$.
\begin{Lemma}\label{Lemma:fgcontinuous}
For all representations $\rho$ of $\T\F$ the three functions
\begin{alignat*}{2}
&\R^{d_2}\to\C^{(dn_0d_\rho)\times(dn_0d_\rho)},\quad&&k\mapsto\fourier\ff(\Ind(\chi_k\rho))\quad\text{and}\\
&\R^{d_2}\to\C^{(\abs\RR n_0d_\rho)\times(dn_0d_\rho)},\quad&&k\mapsto\fourier{g_{\RR(,0,0)}}(\Ind(\chi_k\rho))
\end{alignat*}
are continuous and the two functions
\begin{align*}
&\R^{d_2}\to\R\cup\{\pm\infty\},\quad k\mapsto\lambdamin\parens[\Big]{\fourier\ff(\Ind(\chi_k\rho)),\fourier{g_{\RR(,0,0)}}(\Ind(\chi_k\rho))}
\end{align*}
are upper semicontinuous.
\end{Lemma}
\begin{proof}
Let $\rho$ be a representation of $\T\F$. 
The first three functions are continuous since $\ff\in L^1(\G,\R^{d\times d})$ and $g_{\RR(,0,0)}\in L^1(\G,\R^{(d\abs\RR)\times d})$.

Let $f_1(k):=\fourier\ff(\Ind(\chi_k\rho))$, $f_2(k):=\fourier{g_{\RR(,0,0)}}(\Ind(\chi_k\rho))$, $f(k):=\lambdamin\parens[\normalsize]{f_1(k),f_2(k)}$, $(k_n)_{n\in\N}$ a sequence in $\R^{d_2}$ and $k\in\R^{d_2}$ such that $\lim_{n\to\infty}k_n=k$.
We assume that $\limsup_{n\to\infty}f(k_n)>-\infty$ and $\limsup_{n\to\infty}f(k_n)=\lim_{n\to\infty}f(k_n)$ without loss of generality.
Let $\lambda\in\R$ such that $\lambda<\limsup_{n\to\infty}f(k_n)$.
We have $\lambda f_2(k_n)^{\mathsf H}f_2(k_n)\le f_1(k_n)$ for all $n\in\N$ large enough.
Since the Loewner order is closed, \ie the set $\set{(A,B)\in X^2}{A\le B}$ is closed, where $X=\set{A\in\C^{(dn_0d_\rho)\times(dn_0d_\rho)}}{A\text{ is Hermitian}}$, we have $\lambda f_2(k)^{\mathsf H}f_2(k)\le f_1(k)$.
Thus we have $\lambda\le f(k)$.
\end{proof}
\begin{proof}[Proof of Theorem~\ref{Theorem:LambdaInduced}]
We may assume that $K_\rho$ is a representation set of $\R^{d_2}/\G_\rho$. By Lemma~\ref{Lemma:fgcontinuous} the assertion then directly extends to the case that $\overline{K_\rho}$ contains such a set. 
Recall that $\MM=m_0\N$. 
By Lemma~\ref{Lemma:RepSys}\ref{item:aaa} there exists a representation set $R'$ of a representation set of $\dual{\T\F}/{\simrg}$ such that $\rho$ is $\T^{m_0}$-periodic for all $\rho\in R'$. 
Due to the existence of fundamental domains, see, \eg, \cite[Theorem 6.6.13]{Ratcliffe2006}, for all $\rho\in R'$ there exists a representation set $K_\rho'$ of $\R^{d_2}/\G_\rho$ such that $L_\rho'$ is a dense subset of $K_\rho'$, where $L_\rho'=\set{k\in K_\rho'}{\exists\+N\in \MM:k\in\dualL/N}$.
By Theorem~\ref{Theorem:MainRepDual}\ref{Theorem-item:RepSet} applied to $R$ and $R'$, there exist a bijection
\begin{equation}\label{eq:toxw}
\phi\colon\bigsqcup_{\rho\in R'}K_\rho'\to\bigsqcup_{\rho\in R}K_\rho,\quad(k,\rho)\mapsto(\phi_1(k,\rho),\phi_2(k,\rho))
\end{equation}
and for all $\rho\in R'$ and $k\in K_\rho'$ some $T_{k,\rho}\in\U(d_{\Ind(\chi_k\rho)})$ such that
\begin{equation}\label{eq:toxy}
\Ind(\chi_{\phi_1(k,\rho)}\phi_2(k,\rho))=T_{k,\rho}^{\mathsf H}\Ind(\chi_k\rho)T_{k,\rho}.
\end{equation}
By \eqref{eq:toxw} and \eqref{eq:toxy} we have
\begin{align}\label{eq:starac}
&\text{RHS}:=\inf\set[\Big]{\lambdamin\parens[\Big]{\fourier\ff(\Ind(\chi_k\rho)),\fourier\gggg(\Ind(\chi_k\rho))}}{\rho\in R,k\in K_\rho}\nonumber\\
&=\inf\set[\Big]{\lambdamin\parens[\Big]{\fourier\ff\parens[\big]{\Ind\parens[\big]{\chi_{\phi_1(k,\rho)}\phi_2(k,\rho)}},\fourier\gggg\parens[\big]{\Ind\parens[\big]{\chi_{\phi_1(k,\rho)}\phi_2(k,\rho)}}}}{\rho\in R',k\in K_\rho'}\nonumber\\
\begin{split}
&=\inf\braces[\Big]{\lambdamin\parens[\Big]{\parens[\big]{I_d\otimes T_{k,\rho}^{\mathsf H}}\fourier\ff(\Ind(\chi_k\rho))\parens[\big]{I_d\otimes T_{k,\rho}},\parens[\big]{I_d\otimes T_{k,\rho}^{\mathsf H}}\fourier\gggg(\Ind(\chi_k\rho))\parens[\big]{I_d\otimes T_{k,\rho}}}\,\Big|\,\\
&\qquad\rho\in R',k\in K_\rho'}
\end{split}\nonumber\\
&=\inf\set[\Big]{\lambdamin\parens[\Big]{\fourier\ff(\Ind(\chi_k\rho)),\fourier\gggg(\Ind(\chi_k\rho))}}{\rho\in R',k\in K_\rho'}.
\end{align}
For all $\rho\in R'$ we define the function
\begin{align*}
f_\rho\colon K_\rho'\to\R\cup\{\pm\infty\},\qquad
k\mapsto\lambdamin\parens[\Big]{\fourier\ff(\Ind(\chi_k\rho)),\fourier\gggg(\Ind(\chi_k\rho))}.
\end{align*}
By Lemma~\ref{Lemma:fgcontinuous} for all $\rho\in R'$ the function $f_\rho$ is upper semicontinuous and thus we have
\begin{equation}\label{eq:rhc}
\inf f_\rho=\inf f_\rho|_{L_\rho'}.
\end{equation}
By \eqref{eq:starac} and \eqref{eq:rhc} we have
\[\text{RHS}=\inf\set[\big]{f_\rho(k)}{\rho\in R',k\in L_\rho'}.\]
By Theorem~\ref{Theorem:LambdaIrreducible} we have
\begin{equation}\label{eq:dak}
\lambdaa=\inf\set[\Big]{\lambdamin\parens[\Big]{\fourier\ff(\rho),\fourier\gggg(\rho)}}{\rho\in\EE}.
\end{equation}
By Lemma~\ref{Lemma:KronPlus}\ref{item:KronPlus} there exists a permutation matrix $P_{n,p_1,\dots,p_k}\in\O(n(p_1+\dots+p_k))$ for all $n,p_1,\dots,p_k\in\N$ such that
\[A\otimes(B_1\oplus\dots\oplus B_k)=P_{m,p_1,\dots,p_k}^{\mathsf T}((A\otimes B_1)\oplus\dots\oplus(A\otimes B_k))P_{n,p_1,\dots,p_k}\]
for all $A\in\C^{m\times n}$ and $B_i\in\C^{p_i\times p_i}$, $i\in\{1,\dots,k\}$.

Now we show that $\lambdaa\le\text{RHS}$.
Let $\rho\in R'$, $k\in L_\rho'$ and $\rho'=\Ind(\chi_k\rho)$.
Let $N\in \MM$ such that by Remark~\ref{Remark:VNK} and the construction of $L_\rho'$ we have $\chi_k|_{\T^N}=1$.
The map $\rho'$ is $\T^N$-periodic.
There exist some $\rho_1,\dots,\rho_n\in\EE$ and $T\in\U(d_{\rho'})$ such that
\[\rho'(g)=T^{\mathsf H}(\rho_1(g)\oplus\dots\oplus\rho_n(g))T\qquad\text{for all }g\in\G.\]
We have
\begin{align}\label{eq:1ac}
\fourier\ff(\rho')&=\sum_{g\in\G}\ff(g)\otimes\rho'(g)\nonumber\\
&=\sum_{g\in\G}\ff(g)\otimes\parens[\big]{T^{\mathsf H}(\rho_1(g)\oplus\dots\oplus\rho_n(g))T}\nonumber\\
&=(I_d\otimes T)^{\mathsf H}\parens[\bigg]{\sum_{g\in\G}\ff(g)\otimes\parens[\big]{\rho_1(g)\oplus\dots\oplus\rho_n(g)}}(I_d\otimes T)\nonumber\\
&=P^{\mathsf H}\parens[\bigg]{\parens[\bigg]{\sum_{g\in\G}\ff(g)\otimes\rho_1(g)}\oplus\dots\oplus\parens[\bigg]{\sum_{g\in\G}\ff(g)\otimes\rho_1(g)}}P\nonumber\\
&=P^{\mathsf H}\parens[\Big]{\fourier\ff(\rho_1)\oplus\dots\oplus\fourier\ff(\rho_n)}P,
\end{align}
where $P$ is the unitary matrix $P_{d,d_{\rho_1},\dots,d_{\rho_n}}(I_d\otimes T)$.
Analogously to \eqref{eq:1ac} we have
\begin{equation}\label{eq:2ac}
\fourier\gggg(\rho')=Q^{\mathsf H}\parens[\Big]{\fourier\gggg(\rho_1)\oplus\dots\oplus\fourier\gggg(\rho_n)}P,
\end{equation}
where $Q$ is the unitary matrix $P_{d\abs\RR,d_{\rho_1},\dots,d_{\rho_n}}(I_{d\abs\RR}\otimes T)$.
By \eqref{eq:1ac}, \eqref{eq:2ac} and \eqref{eq:dak} we have
\begin{align*}
f_\rho(k)&=\lambdamin\parens[\Big]{\fourier\ff(\rho_1)\oplus\dots\oplus\fourier\ff(\rho_n),\fourier\gggg(\rho_1)\oplus\dots\oplus\fourier\gggg(\rho_n)}\\
&=\min\set[\Big]{\lambdamin\parens[\Big]{\fourier\ff(\rho_i),\fourier\gggg(\rho_i)}}{i\in\{1,\dots,n\}}
\ge\lambdaa.
\end{align*}

Now we show that $\lambdaa\ge\text{RHS}$.
Let $\rho_1\in\EE$.
By Theorem~\ref{Theorem:MainRepDual}\ref{Theorem-item:RepSys} the set $\set{\Ind(\chi_k\rho)}{\rho\in R',k\in L_\rho'}$ is a representation set of $\Ind\parens{\set{\rho\in\dual{\T\F}}{\rho\text{ is periodic}}}$.
By Theorem~\ref{Theorem:MainRepDual}\ref{Corollary-item:Subreps-induced} there exist some $\rho\in R'$ and $k\in L_\rho'$ such that $\rho_1$ is isomorphic to a subrepresentation of $\Ind(\chi_k\rho)$.
Let $\rho'=\Ind(\chi_k\rho)$.
There exist some $\rho_2,\dots,\rho_n\in\EE$ and $T\in\U(d_{\rho'})$ such that
\[\rho'(g)=T^{\mathsf H}(\rho_1(g)\oplus\dots\oplus\rho_n(g))T\qquad\text{for all }g\in\G.\]
Analogously to \eqref{eq:1ac} and \eqref{eq:2ac} we have
\begin{align*}
&\fourier\ff(\rho')=P^{\mathsf H}\parens[\Big]{\fourier\ff(\rho_1)\oplus\dots\oplus\fourier\ff(\rho_n)}P
\shortintertext{and}
&\fourier\gggg(\rho')=Q^{\mathsf H}\parens[\Big]{\fourier\gggg(\rho_1)\oplus\dots\oplus\fourier\gggg(\rho_n)}P,
\end{align*}
where $P$ and $Q$ are the unitary matrices $P_{d,d_{\rho_1},\dots,d_{\rho_n}}(I_d\otimes T)$ and $P_{d\abs\RR,d_{\rho_1},\dots,d_{\rho_n}}(I_{d\abs\RR}\otimes T)$, respectively.
We have
\[\lambdamin\parens[\Big]{\fourier\ff(\rho_1),\fourier\gggg(\rho_1)}\ge\min\set[\Big]{\lambdamin\parens[\Big]{\fourier\ff(\rho_i),\fourier\gggg(\rho_i)}}{i\in\{1,\dots,n\}}=f_\rho(k)\ge\text{RHS}.\]

The proof of the characterization of $\lambdanewa$ is analogous.
\end{proof}
%
%%%%%%%%%%%%%%%%%%%%%%%%%%%%%%
%%%%%%%%%%%%%%%%%%%%%%%%%%%%%%
\appendix
\section{Appendix}
\subsection*{Kronecker product}
For $A=(a_{ij})\in\C^{m\times n}$ and $B=(b_{ij})\in\C^{p\times q}$ the \emph{Kronecker product} $A\otimes B\in\C^{(mp)\times(nq)}$ of $A$ and $B$ is the partitioned matrix
\begin{align}\label{eq:Kronecker-product}
A\otimes B:=\begin{bmatrix}
a_{11}B & \cdots & a_{1n}B\\
\vdots & \ddots & \vdots\\
a_{m1}B & \cdots & a_{mn}B
\end{bmatrix}.
\end{align}
Identifying $\C^n$ with $\C^{n\times 1}$, the Kronecker product is also defined if $A$ or $B$ is a vector. 
For the basic properties of the Kronecker product we refer to \cite{Bernstein2009}.
\begin{Lemma}\label{Lemma:KronPlus}
For all $m,n\in\N$ let $P_{m,n}\in\O(mn)$ be the Kronecker permutation matrix such that
\[P_{p,m}(A\otimes B)P_{n,q}=B\otimes A\qquad\text{for all }A\in\C^{m\times n}\text{ and }B\in\C^{p\times q},\]
see \cite[Fact 7.4.30]{Bernstein2009}.
For all $m,n_1,\dots,n_k\in\N$ let $Q_{m,n_1,\dots,n_k}\in\O(m(n_1+\dots+n_k))$ be the permutation matrix $(P_{m,n_1}\oplus\dots\oplus P_{m,n_k})P_{n_1+\dots+ n_k,m}$.
Then the following statements hold:
\begin{enumerate}
\item For all $A_i\in\C^{m_i\times n_i}$, $i\in\{1,\dots,k\}$, and $B\in\C^{p\times q}$ we have
\[(A_1\oplus\dots\oplus A_k)\otimes B=(A_1\otimes B)\oplus\dots\oplus(A_k\otimes B).\]
\item\label{item:KronPlus} For all $A\in\C^{m\times n}$ and $B_i\in\C^{p_i\times q_i}$, $i\in\{1,\dots k\}$, we have
\[A\otimes(B_1\oplus\dots\oplus B_k)=Q_{m,p_1,\dots,p_k}^{\mathsf T}((A\otimes B_1)\oplus\dots\oplus(A\otimes B_k))Q_{n,q_1,\dots,q_k}.\]
\item For all $A\in\C^{m\times n}$ and $B_1,\dots,B_k\in\C^{p\times q}$ we have
\[A\otimes(B_1\oplus\dots\oplus B_k)=(P_{m,k}\otimes I_p)((A\otimes B_1)\oplus\dots\oplus(A\otimes B_k))(P_{k,n}\otimes I_q).\]
\end{enumerate}
\end{Lemma}
\begin{proof}
\begin{enumerate}
\item This is easy to check.
\item For all $A\in\C^{m\times n}$ and $B_i\in\C^{p_i\times q_i}$, $i\in\{1,\dots,k\}$, we have
\begin{align*}
&A\otimes(B_1\oplus\dots\oplus B_k)=P_{m,p_1+\dots+p_k}((B_1\oplus\dots\oplus B_k)\otimes A)P_{q_1+\dots+q_k,n}\\
&=P_{m,p_1+\dots+p_k}((B_1\otimes A)\oplus\dots\oplus(B_k\otimes A))P_{q_1+\dots+q_k,n}\\
&=P_{m,p_1+\dots+p_k}((P_{p_1,m}(A\otimes B_1)P_{n,q_1})\oplus\dots\oplus(P_{p_k,m}(A\otimes B_k)P_{n,q_k}))P_{q_1+\dots+q_k,n}\\
&=Q_{m,p_1,\dots,p_k}^{\mathsf T}((A\otimes B_1)\oplus\dots\oplus(A\otimes B_k))Q_{n,q_1,\dots,q_k}.
\end{align*}
\item By Fact 7.4.30$\+viii)$ in \cite{Bernstein2009} we have
\[Q_{n,q,\dots,q}=(I_k\otimes P_{n,q})P_{kq,n}=P_{k,n}\otimes I_q.\qedhere\]
\end{enumerate}
\end{proof}
\bibliography{Lit-StabObjStrc}

\end{document}